\documentclass[12pt]{amsart}

\usepackage{mathtools}
\mathtoolsset{showonlyrefs}

\usepackage{xcolor}
\usepackage{amssymb}
\usepackage{enumitem}
\usepackage[margin=1in]{geometry}
\usepackage[utf8]{inputenc}
\usepackage[cyr]{aeguill}
\usepackage[pdfdisplaydoctitle=true,
            colorlinks=true,
            urlcolor=blue,
            citecolor=blue,
            linkcolor=blue,
            pdfstartview=FitH,
            pdfpagemode= UseNone,
            bookmarksnumbered=true]{hyperref}
\usepackage{todonotes}
\usepackage[all]{xy}
\usepackage{stmaryrd}
\usepackage{wrapfig}

\usepackage[draft]{pdfcomment}

\usepackage[subrefformat=parens]{subcaption}
\newcommand{\rev}[1]{#1}

\newtheorem{theorem}{Theorem}[section]

\newtheorem{corollary}[theorem]{Corollary}
\newtheorem{lemma}[theorem]{Lemma}

\newtheorem{proposition}[theorem]{Proposition}
\theoremstyle{definition}

\newtheorem{example}[theorem]{Example}

\newtheorem{remark}[theorem]{Remark}

\newtheorem{question}[theorem]{Open Problem}

\newcommand{\R}{\mathbb R}
\newcommand{\Z}{\mathbb Z}

\newcommand{\Diff}{\operatorname{Diff}(S^1)}
\newcommand{\Diffp}{\operatorname{Diff}^+(S^1)}

\newcommand{\id}{\mathrm{id}}

\newcommand{\Imm}{\mathrm{Imm}}

\newcommand{\OmMW}{\Omega^{\operatorname{MW}}}

\def\ga{\gamma}

\def\ze{\zeta} 
 
\def\th{\theta} 
 
\def\ka{\kappa} 
\def\la{\lambda}

\def\ta{\tau} 
\def\ph{\varphi}

\def\om{\omega} 
\def\Ga{\Gamma} 
 
\def\Th{\Theta} 
\def\La{\Lambda} 
 
\def\Ph{\Phi} 
 
\def\Om{\Omega}

\def\o{\operatorname{\circ}} 
\def\i{^{-1}} 
\def\x{\times}
\def\p{\partial} 
\def\X{{\mathfrak X}}
  
\def\L{\mathcal{L}}
\def\F{\mathcal{F}}
\def\J{\mathcal{J}}
\def\F{\mathcal{F}}

\def\R{{\mathbb R}}

\def\exp{\operatorname{exp}}

\let\on=\operatorname

\newcommand{\simtime}[1]{\vspace{-5mm}\subcaption*{$t={#1}$}}

\allowdisplaybreaks
\title{Symplectic structures on the space of space curves}


\author{Martin Bauer, Sadashige Ishida, and Peter W. Michor}
\begin{document}

\begin{abstract}
   We present symplectic structures on the shape space of unparameterized space curves that generalize the classical Marsden-Weinstein structure. Our method integrates the Liouville 1-form of the Marsden-Weinstein structure with Riemannian structures that have been introduced in mathematical shape analysis. We also derive Hamiltonian vector fields for several classical Hamiltonian functions with respect to these new symplectic structures. 
\end{abstract}

\maketitle


\setcounter{tocdepth}{1}
\tableofcontents
\section{Introduction}
\subsection*{Motivation and background:}
The space of unparametrized space curves $$\rev{B_i(S^1,\R^3)\coloneqq\Imm(S^1,\mathbb R^3)/\Diff}$$ as an infinite dimensional orbifold is known to have a symplectic structure called the Marsden-Weinstein structure 
 (MW-structure) ~\cite{MarsdenWeinstein1983}. It is thought of as a \emph{canonical} symplectic structure as it is formally a Kirillov-Kostant-Souriau form by regarding space curves as linear functionals on the space of divergence-free vector fields in $\R^3$;  see eg.~\cite[Theorem 4.2]{MarsdenWeinstein1983} and \cite[Chapter VI, Proposition 3.6]{ArnoldKhesin2021}. Another incentive for studying the MW symplectic structure can be found in its appearance in mathematical fluid dynamics:  for example, one can interpret vortex filaments as the MW flow of the kinetic energy of the velocity field induced by vorticity concentrated on the curve. 
 Via so-called localized induction approximation vortex filaments reduce to the binormal flow, which is a completely integrable system and is again an MW flow for the length functional as the Hamiltonian, see eg.~\cite[Chapter 11]{Saffman1993vortex} or \cite[Chapter 7]{Majda_Bertozzi2001incompressibleFlow} and the references therein.

To the best of the authors' knowledge, to date no symplectic structures other than the MW form have been studied on the space of unparametrized space curves. Riemannian structures on this space, on the other hand, have attracted a significant amount of interest;  primarily due to their relevance to mathematical shape analysis~\cite{younes2010shapes,srivastava2016functional,bauer2014overview}. The arguably most natural such metric, the reparametrization invariant $L^2$-metric admits a surprising degeneracy: the geodesic distance between any pair of curves vanishes on both the space of parametrized and unparametrized curves~\cite{michor2005vanishing,bauer2012vanishing}. This result renders the $L^2$-metric unsuited as a basis for mathematical shape analysis and thus started a quest for stronger Riemannian metrics, which induce a non-degenerate distance function and consequently can be used for applications in these areas, see eg.~\cite{michor_mumford2006,yezzi2005conformal,srivastava2010shape,bauer2017numerical} and the references therein. 

 The aforementioned $L^2$-Riemannian metric and the MW symplectic structure are related via an almost complex structure, which is induced on shape space by the cross-product with the unit tangent vector of the curve $c$, i.e., $J_c(h):=\frac{c_\theta}{|c_\theta|} \x h$; here $c:S^1\to\mathbb R^3$ is a space curve and $h:S^1\to\mathbb R^3$ is a tangent vector to $c$. Furthermore, the MW symplectic structure $\bar \Om^{\on{MW}}$ has a Liouville 1-form $\bar\Theta^{\operatorname{MW}}$ i.e., $\bar \Om^{\on{MW}}=-d \bar\Theta^{\operatorname{MW}}$, which arises from the $L^2$-metric $G$ and the almost complex structure via 
 \begin{equation}
 \bar\Theta^{\operatorname{MW}}_{\bar c}(\bar h):= -\frac{1}{3} G_c(J_c(c),h)
 \end{equation}
 \rev{where $\bar c$ and $\bar h$ are an element and a tangent vector on $B_i(S^1,\R^3)$ related to $c$ and $h$ by the projection from $\Imm(S^1,\R^3)$ to $B_i(S^1,\R^3)$.}
\subsection*{Main contributions:} 
These relations between Riemannian geometry and symplectic geometry on the space of space curves are the starting point of the present article: our principal goal is to construct new symplectic structures on the space of unparametrized curves by combining the above classical construction with more recent advances in Riemannian geometry of these spaces, i.e., we construct new presymplectic structures by modifying the Liouville form of the MW form using different Riemannian metrics from mathematical shape analysis. This construction automatically leads to a closed 2-form (and thus a presymplectic form) on the space of parametrized curves. Under certain assumptions on the Riemannian metric this form then descends to a presymplectic structure on the space of unparametrized space curves and thus it only remains to check the non-degeneracy of this 2-form to conclude that it is (weakly) symplectic. Proving this property turns out to be surprisingly difficult and provides the main technical contribution of the present article. Interestingly, in some cases the presymplectic form still has a nontrivial kernel on the shape space, but becomes symplectic when the quotient by a further 2-dimensional foliation is taken.

We also derive formulae for Hamiltonian vector fields of several classical Hamiltonian functions generated by our new symplectic structures and provide numerical illustrations to qualitatively show a few simple examples among these new Hamiltonian flows. The Riemannian counterparts can be found in the area of geometric gradient flows on the space of curves, where the investigations of gradient flows for certain well-known Energy functionals (e.g. the entropy or length energy functional) for Riemannian metrics other than the $L^2$-Riemannian metric has been recently initiated, see e.g. the work of ~Okabe, Schrader, Wheeler and Wheeler~\cite{okabe2023sobolev,schrader2023h}.
In our investigations we observe that, for certain choices of Hamiltonian and symplectic structure, we obtain a new representation of well-known Hamiltonian flows, i.e., we may reproduce Hamiltonian flows of the MW symplectic structure from a different pair of a symplectic structure and a Hamiltonian function. For other examples, we obtain genuinely new Hamiltonian flows, which do not seem to be represented as a Hamiltonian flow for the MW symplectic structure. 


\rev{A seemingly more straightforward approach to obtain new symplectic structures  can be found in directly defining a new skew-symmetric 2-form via alternating the Riemannian metric and combining it with the almost complex structure $\J$ from the MW symplectic structure. This approach turns out to be somewhat unsuccessful as the resulting skew-symmetric 2-form is usually not closed and thus not even pre-symplectic. We discuss this approach and the resulting 2-forms in Appendix~\ref{sec:almost_symp_approach}. This further highlights the non-trivial challenge of finding a 2-form which is both closed and non-degenerate, rather than one being merely closed or non-degenerate. This observation was our original incentive to follow the slightly more complicated procedure described above. Finally, in} Appendix \ref{app:weak_symplectic}, we provide a short introduction to infinite-dimensional weak symplectic geometry, including a new assumption that was overlooked in previous research.

\subsection*{Future directions:}
In this article, we introduced new (pre)symplectic structures on the shape space of space curves. Our procedure of modifying the Liouville form of a (pre)symplectic form and taking the exterior derivative is not limited to such shape spaces. It would be interesting to apply the same machinery for other infinite-dimensional (weak-)symplectic manifolds that admit Liouville forms such as the space of complex functions on a domain or the cotangent bundle of an infinite-dimensional Riemannian manifold. 

At this point, the connection between new Hamiltonian flows and existing physical theories seems to remain unclear. Hence we are also keen to use this new framework to find new interpretations of physically relevant quantities as Hamiltonian flows, in a similar way that certain compressible fluids are modeled as geodesic flows of higher order metrics \cite{MichorMumford2013EPDiff}. 

\subsection*{Structure of the article:}
In Section \ref{sec:Liouville_and_presymplectic}, we introduce Liouville forms via the modification of the $L^2$-Riemannian metric, and then compute presymplectic forms by taking the exterior derivative. In Section \ref{sec:conformal}, we show that, a class of presymplectic structures attained by conformal factors on the shape space are indeed weekly symplectic. We also derive Hamiltonian vector fields with respect to these weak symplectic structures. In Section \ref{sec:symp_length}, we describe more concretely symplectic structures induced by the length function as a special case of conformal factors and provide several examples of Hamiltonian  vector fields. In Section \ref{sec:curvature_weighted}, we discuss the presymplectic structure induced by the curvature-weighted metric, where we leave the non-degeneracy open for future research.
Finally, in Section \ref{sec:numerics}, we numerically illustrate simple Hamiltonian flows with respect to symplectic structures induced by length functions.

\subsection*{Acknowledgements:}
The authors are grateful to Boris  Khesin for valuable comments on the MW symplectic structure and S. Ishida thanks Albert Chern for insightful discussions on space curves and Chris Wojtan for his continuous support.
M. Bauer was partially supported by NSF grant DMS-1953244 and by the Binational Science Foundation (BSF). S.~Ishida was partially supported by ERC Consolidator Grant 101045083 ``CoDiNA'' funded by the European Research Council. Some figures were generated by the software Houdini and its education license was provided by SideFX.
\section{Liouville structures and (pre)symplectic structures}\label{sec:Liouville_and_presymplectic}

\subsection*{The space of parametrized and unparametrized curves}
We consider the space of regular space curves:
\begin{equation}
  \Imm(S^1,\mathbb R^3):=\left\{c\in C^{\infty}(S^1,\mathbb R^3): |\rev{\partial_{\theta}c}|\neq 0\right\},
\end{equation}
which consists of immersions of $S^1$ into $\R^3$.
The space $\Imm(S^1,\mathbb R^3)$ is an open subset of the vector space $C^{\infty}(S^1,\mathbb R^3)$ and thus, similar as in finite dimensions, it  is a manifold with tangent space \rev{given by the sections of the pullback bundle by $c$},
\begin{align}
  T_c\Imm(S^1,\mathbb R^3)=C^{\infty}(S^1,\mathbb R^3). 
\end{align}
\rev{From now on, we will denote differentiation w.r.t $\theta$ by a subscript, i.e., we write $\partial_\theta c = c_\theta$ and $\partial_\theta h=h_\theta$ for $c\in \Imm(S^1,\R^3)$ and $h\in T_c \Imm(S^1,\R^3)$. Furthermore, we will occasionally consider constant vector fields on $\Imm(S^1,\mathbb R^3)$ obtained by extending tangent vectors $h\in T_c\Imm(S^1,\R^3)$ at some $c$ to the entire $\Imm(S^1,\mathbb R^3)$. We will denote this vector field also by $h$ for simplicity, without explicitly stating so whenever it is clear from the context.}

On the manifold of immersions we consider the action of the group of orientation-preserving diffeomorphisms $\Diffp$ by composition from the right. This leads us to consider the quotient (shape) space \rev{via the projection}
\begin{equation}
\rev{\pi: \Imm(S^1,\mathbb R^3)\to}\, B_i(S^1,\mathbb R^3):=\Imm(S^1,\mathbb R^3)/\Diffp,
\end{equation}
which is an infinite dimensional orbifold with finite cyclic groups at the orbifold singularities, see~\cite{cervera1991action} and \cite[7.3]{Michor20}. The tangent space to the vertical fiber through $c$ consist exactly of all fields $h$ that are tangent to its foot point $c$, i.e., $h=a.c_\theta$ with $a\in C^{\infty}(S^1)$. 
 
\subsection*{Reparametrization invariant Riemannian metrics on spaces of curves}
On the space of parametrized curves we will consider reparameterization invariant (weak)-Riemannian metrics of the form:
\begin{align}
  G_c^L(h,k)&= \int_{S^1} \langle L_c h,k\rangle |c_\theta| d\theta=\int_{S^1} \langle h,L_ck\rangle |c_\theta| d\theta
\end{align}
where $ L \in \Gamma(\on{End}(T\Imm(S^1,\mathbb R^3))$ is an operator field, \rev{depending smoothly on $c\in \Imm(S^1,\mathbb R^3)$ such that for each fixed curve $c$} the operator 
\begin{equation}
 L_c:T_c\Imm(S^1,\mathbb R^3)=C^\infty(S^1,\mathbb R^3) \to T_c\Imm(S^1,\mathbb R^3)=C^\infty(S^1,\mathbb R^3)   
\end{equation}
is an elliptic pseudo differential operator that is equivariant under the right action of the diffeomorphism group $\Diffp$ and also under the left action of  $SO(3)$, and which is also self-adjoint with respect to the $L^2$-metric, i.e.,
\begin{equation}
L_{c\circ\varphi}(h\circ\varphi)=(L_c(h))\circ\varphi \quad \text{ and }
\int \langle L_c h, k\rangle ds = \int \langle h , L_c k\rangle ds\;.
\end{equation}
\begin{remark}[Sobolev metrics]\label{rem:Sobmetric}
An important class of such metrics is the class of Sobolev $H^m$-metrics,
where $L=(1-(-1)^m D_s^{2m})$ with $D_s=\frac1{|c_\theta|}\partial_{\theta}$ being the arclength derivative. 
Using the notation $ds=|c_\theta|d\theta$ for the arclength measure  we 
obtain for $m=0$ the metric
\begin{align}
  G_c^{\operatorname{id}}(h,k)&= \int_{S^1} \langle h,k\rangle |\rev{c_\theta}| d\theta=\int_{S^1} \langle h,k\rangle ds  
\end{align}  
and for $m=1$ the metric
\begin{align}
  G_c^{\operatorname{id}-D_s^2}(h,k)&
  = \int_{S^1} \langle h,k\rangle  +\langle - D_s^2 h, k \rangle ds
  = \int_{S^1} \langle h,k\rangle  +\langle D_s h,D_s k \rangle ds.
\end{align}
All these metrics can be written in terms of arc-length derivative \rev{$D_s=\frac{1}{|c_\theta|}\partial_\theta$} and arc-length integration $ds=|c_\theta|d\theta$ only. It has been shown that each such metric induces a corresponding metric on the shape space $B_i(S^1,\R^3)$ such that the projection \rev{$\pi:\Imm(S^1,\R)\to B_i(S^1,\R^3)$} is a Riemannian submersion~\cite{michor_mumford2007}. In finite dimension this would follow directly from the invariance of the metric, but in this infinite dimensional situation one has to show in addition the existence of the horizontal complement (w.r.t. the Riemannian metric). We will see, however, that this particular class of metrics will not be suited for the purpose of the present paper, as the induced symplectic structure will not descend to a symplectic structure on the quotient space. 
\end{remark}
\subsection*{The induced Liouville 1-form}
Next we will use the metric $G^L$ to define a (Liouville) 1-form on $\Imm(S^1,\mathbb R^3)$. Therefore we consider for $c\in \Imm(S^1,\mathbb R^3)$ and $h\in T_c\Imm(S^1,\mathbb R^3)$ the 1-form: 
\begin{equation}\label{eq:Liouvilleform}
\Theta^L_c(h)\coloneqq G^L_c(c\times D_s  c, h)=\int \langle c\times D_sc, L_c h\rangle ds=\int \operatorname{det}(c, D_sc, L_c h) ds,
\end{equation}
where $\times$ denotes the vector cross product on $\mathbb R^3$. We have the following result concerning its invariance properties:
\begin{lemma}[Liouville 1-form]\label{lem:Liouville_one_form}
For any inertia operator $L$, which is equivariant under the right action of the group of all orientation preserving diffeomorphisms and the left action of the rotation group $SO(3)$, the induced Liouville 1-form $\Theta^L$ is invariant under the right action of $\Diffp$ and the left action of $SO(3)$, i.e., for any $c\in \Imm$, $h\in T_c\Imm$, $\varphi\in \Diffp$ and $O\in SO(3)$ we have 
\begin{equation}
\Theta^L_{O(c\circ\varphi)}(O(h\circ\varphi))=\Theta^L_c(h).
\end{equation}
\begin{proof}
We will only show the reparametrization invariance, the invariance under $SO(3)$ is similar but easier.
Using the equivariance of both $L$ and $D_s$ we calculate
\begin{equation*}
\Theta^L_{c\circ\varphi}(h\circ\varphi)=\int \langle c\circ\varphi \times (D_sc)\circ\varphi, (L_ch)\circ\varphi\rangle|c_\theta|\circ\varphi |\varphi'|\; d\theta=\int \langle c\times D_sc, L_c h\rangle ds=\Theta^L_c(h).\qedhere
\end{equation*}
\end{proof}
\end{lemma}

\begin{remark}
    If $L$ is equivariant under the left action of not only $SO(3)$ but of the larger group $SL(3)=\{M\in GL(3,\R)\mid \det(M)=1\}$, then also $\Th^L$ is invariant under $SL(3)$. This is the case for the Marsden-Weinstein structure $L=\id$ (see Remark \ref{rmk:Marsden-Weinstein}), but in general not for the inertia operators we deal with in this article.
\end{remark}

\subsection*{The induced (pre)symplectic form on $\Imm(S^1,\mathbb R^2)$}
Once we have defined the 1-form $\Theta$ we can formally consider the induced symplectic form
\begin{equation}\label{eq:definition_of_Omega}
\Omega^L_c(h,k)\coloneqq -d\Theta^L_c(h,k)=-D_{c,h}\Theta^L_c(k)+D_{c,k}\Theta^L_c(h)+\Theta^L_c([h,k]),
\end{equation}
where $d$ denotes the exterior derivative, $D_{c,h}$ denotes the directional derivative at  $c\in \Imm(S^1,\R^3)$ in the direction $h$, and when applied to a function $f\colon \Imm(S^1,\R^3)\to \R$, we have $D_{c,h}f=\L_h f (c)$. The bracket $[h,k]$ is the Lie-bracket in $\X(\Imm(S^1,\mathbb R^3))$ given by  $[h,k]=D_{c,h}k-D_{c,k}h$.

 In the following theorem we calculate  this 2-form explicitly:
\begin{theorem}[The (pre)symplectic form $\Omega^L$ on parametrized curves]\label{lem:formula_sympl}
Let $c\in \Imm(S^1,\mathbb R^3)$ and $h,k\in T_c\Imm(S^1,\mathbb R^3)$. We have
\begin{equation}\label{eq:Omega_L}
\begin{aligned}
\Omega^L_c(h,k)&=\int\Big(\langle D_s c, L_c h \x k + h\x L_c k\rangle 
     - \langle c, D_s h \x L_c k + L_ch\x D_s k  \rangle 
    \\&\qquad\qquad\qquad\qquad
     + \langle c\x D_s c, (D_{c,k}L_c)h-(D_{c,h}L_c)k\rangle\Big) ds 
\end{aligned}
\end{equation}
Furthermore, $\Omega^L$ is invariant under the right action of $\Diffp$ and under the left action of $SO(3)$. 
\end{theorem}
\begin{remark}[Marsden-Weinstein symplectic structure]\label{rmk:Marsden-Weinstein}
It is known that for the invariant $L^2$-metric, i.e., $L=\operatorname{id}$,
one obtains three times the Marsden-Weinstein (weak)-symplectic structure with this procedure (See \cite{tabachnikov2017bicycle,Padilla2019bubbleRings} for example), i.e.,
\begin{equation}\label{eq:Marsden-Weinstein form for space curves}
3\Omega^{\operatorname{MW}}_c(h,k):=\Om^\id_c(h,k)= 3 \int_{S^1}\langle D_sc\times h,k\rangle ds = 3 \int\det(D_sc,h,k)ds.
\end{equation}

Its kernel consists exactly of all vector fields along $c$ which are tangent to $c$, so by reduction it induces a presymplectic structure on 
shape space $\Imm(S^1,\mathbb R^3)/\Diffp$ which is easily seen to be weakly non-degenerate and thus is a symplectic structure there. 
\end{remark}

\begin{proof}[Proof of \autoref{lem:formula_sympl}]
To prove the formula for $\Om^L$ we first collect several variational formulas, see eg.~\cite{michor_mumford2006} for a proof:
\begin{align*}
  ds &= |c_\th|d\th, \quad 
  D_{c,h} ds = \frac{\langle h_\th,c_\th\rangle}{|c_\th|} d\th 
  = \langle D_s h, D_s c\rangle ds 
  \\
  D_s &=\frac{1}{|c_\th|}\p_\th, \quad 
  D_{c,h}D_s = \frac{-\langle h_\th, c_\th\rangle}{|c_\th|^3}\p_\th
  = -\langle D_s h, D_s c\rangle D_s.
\end{align*}
Since $\Imm(S^1, \mathbb R^3)$ is open in $C^\infty(S^1,\mathbb R^3)$, we can choose globally constant $h,k$ i.e., independent of the location $c$ on $\Imm(S^1, \mathbb R^3)$, namely \rev{$D_{c,h}(k)=D_{c,k}(h)=0$ and} $[h,k]=0$. \rev{Using $D_{c,h}(L_c k)=(D_{c,h}L_c)k+L_c (h(k))=(D_{c,h}L_c)k$, we} compute
\begin{align*}
    D_{c,h}\Th^L_c(k) &= \int\Big(\det(h, D_s c, L_c k) 
     -\langle D_s h, D_s c\rangle \det(c, D_s c, L_c k)
     +\det(c, D_s h, L_c k)
    \\&\qquad\qquad
     +\det(c, D_s c, (D_{c,h}L_c)k)
     +\langle D_s h, D_s c\rangle\det(c, D_s c, L_c k)\Big)ds 
    \\&
    = \int\Big(\det(h, D_s c, L_c k)+\det(c, D_s h, L_c k) +\det(c, D_s c, (D_{c,h}L_c)k)\Big)ds.
\end{align*}
Thus we get for $\Om^L$:
\begin{equation}
\begin{aligned}    
\Om^L_c(h,k)& = -D_{c,h}\Th^L_c(k) + D_{c,k}\Th^L_c(h) + 0
    \\&
    = \int\Big(-\det(h, D_s c, L_c k) + \det(k, D_s c, L_c h) 
     -\det(c, D_s h, L_c k) + \det(c, D_s k, L_c h)
    \\&\qquad\qquad
     -\det\big(c, D_s c, (D_{c,h}L_c)k - (D_{c,k}L_c)h\big)\Big)ds
    \\&
    =\int\Big(\langle D_s c, L_c h \x k + h\x L_c k\rangle 
     - \langle c, D_s h \x L_c k - D_s k \x L_c h \rangle 
    \\&\qquad\qquad
     - \langle c\x D_s c, (D_{c,h}L_c)k - (D_{c,k}L_c)h\rangle\Big) ds,
\end{aligned}
\end{equation}
which yields the desired formula for $\Om^L$. The invariance properties of $\Om^L$ follow directly from the corresponding invariance properties of $\Theta^L$.
\end{proof}


\subsection*{The induced (pre)symplectic structure on $B_i(S^1,\mathbb R^3)$}\label{symplectic_shape}
In the previous part we have calculated a (pre)symplectic form on the space of parametrized curves $\Imm(S^1,\mathbb R^3)$; we are, however, rather interested to construct symplectic structures on the shape space of geometric curves $B_i(S^1,\mathbb R^3)$. The following result contains necessary and sufficient conditions for the forms $\Th^L$ and $\Om^L$ to descend to this quotient space:

\begin{theorem}[The (pre)symplectic structure on unparametrized curves]\label{th:descending_presymplectic_form}
The form $\Om^L$ factors to a (pre)symplectic form $\bar\Om^L$ on $B_i(S^1,\mathbb R^3)$ if the inertia operator $L$ \rev{is equivariant under the $\Diffp$-action and} maps vertical tangent vectors to $\operatorname{span}\{c,c_\theta\}$, i.e., if for all $c\in \Imm(S^1,\mathbb R^3)$ and $a\in C^{\infty}(S^1)$ we have
$L_c(a.c_\theta)=a_1 c_\theta+a_2 c$ for some functions $a_i\in C^{\infty}(S^1)$.
\end{theorem}

\begin{proof}
The Liouville form $\Th^L$ on $\Imm(S^1,\mathbb R^3)$ factors to a smooth 1-form $\bar\Th^L$ on shape space $B_i(S^1,\mathbb R^3)$ with $\Th^L = \pi^* \bar \Th^L$
if and only if $\Th^L$ is invariant under under the reparameterization group $\Diffp$ and is \emph{horizontal} in the sense that it vanishes on each vertical tangent vector $h=a.c_\theta$ for a in $C^\infty(S^1,\mathbb R)$. 

Since $\Th^L$ is invariant under the reparameterization group $\Diffp$ by construction it only  remains to determine a condition on $L$ such that $\Th^L$ vanishes on all vertical $h$, i.e., we want \begin{equation}
\Theta^L_c(ac_\theta)=\int \langle c\times D_sc, L_c(ac_\theta)\rangle ds=0.
\end{equation}
From here it is clear that this holds if $L_c(a.c_\theta)=a_1 c_\theta+a_2 c$ for some functions $a_i\in C^{\infty}(S^1)$.

In that case also its exterior derivative satisfies 
\begin{equation*}
 \Om^L=- d \Th^L = - d \pi^* \bar\Th^L = - \pi^* d \bar\Th^L =:  \pi^* \bar \Om^L   
\end{equation*}
for the presymplectic form $\bar\Om^L = - d\bar\Th^L$ on $B_i(S^1,\mathbb R^3)$.
\end{proof}
\begin{example}[Inertia operators with a prescribed horizontal bundle]
 There are several different examples of operators that satisfy these conditions, including in particular 
 the class of almost local metrics:
\begin{align*}
    L_c(h) &= F(c).h \text{ for }F\in C^\infty(\Imm(S^1,\mathbb R^3),\mathbb R_{>0}), \text{ for example}
\\
    L_c (h) &= \Ph(\rev{\ell(c)})h, \quad L_c (h) = \Ph(\int_{S^1} \frac{\ka^2}{2} ds)h, \quad L_c(h) = (1+ A\ka^2)h,
\end{align*}
\rev{where $\kappa_c=|D_s^2c|$ denotes the curvature and $\Phi\colon \R_{\geq 0}\to \R_{>0}$ is a suitable smooth function.}
Note, that the class of Sobolev metrics, as introduced in Remark~\ref{rem:Sobmetric} does not satisfy the conditions of the above theorem. Thus these metrics do not induce a (pre)symplectic form on the quotient space. 
By including a projection operator in their definition one can, however, modify these higher-order metrics to still respect the vertical bundle:
\begin{align*}
    L_c h &= \Big(\on{pr}_c(1- (-1)^k D_s^{2k})\on{pr}_c + (1-\on{pr}_c)(1- (-1)^k D_s^{2k})(1-\on{pr}_c)\Big ) h,
\end{align*}
where 
$\on{pr}_c h= \langle D_sc, h\rangle D_sc$  is the $L^2$-orthogonal projection to the vertical bundle. For more details see~\cite{bauer2015metrics}, where metrics of this form were studied in detail.
\end{example}

\begin{remark}[Horizontal $\Om^L$-Hamiltonian vector fields and  $\bar\Om^{L}$-Hamiltonian vector fields]\label{horizontalHamiltonian} 
In the following we assume that the inertia operator $L \in \Gamma(\on{End}(T\Imm(S^1,\mathbb R^3))$ induces a (weak) symplectic structure on $B_i(S^1,\mathbb R^3)$, i.e., it satisfies the conditions of Theorem \ref{th:descending_presymplectic_form}
and is moreover weakly non-degenerate in the sense that $\bar\Om^L: \rev{TB_i(S^1,\R^3) \to T^*B_i(S^1,\R^3) }$ is injective. Since $T_c^*\pi\o \bar\Om^L_{\pi(c)}\o T_c\pi = \Om^L_c$,  this is equivalent to the kernel of $\Om^L_c:T_c\Imm \to T_c^*\Imm$ being equal to the tangent space to the $\Diffp$-orbit $c\o \Diffp$  for all $c$. Thus $\Om^L_c$ restricted to the $G^L$-orthogonal complement of $T_c(c\o\Diffp)$ is injective. See \cite[Section 48]{kriegl_michor_1997convenient} for more details.

Assume that $H$ is a $\Diffp$-invariant smooth function on $\Imm(S^1,\mathbb R^3)$. Then $H$  induces a Hamiltonian function $\bar H$ on the quotient space $B_i(S^1,\mathbb R^3)$ with $\bar H\o\pi = H$.
Since the 2-form $\Om^L$ on $\Imm(S^1,\mathbb R^3)$ is only presymplectic it does not directly define a Hamiltonian vector field. However, if each $dH_c$ lies in the image of $\Om^L:T\Imm(S^1,\mathbb R^3) \to T^*\Imm(S^1,\mathbb R^3)$, then  a unique smooth \emph{horizontal Hamiltonian vector field} $X\in \X(\Imm)$  is determined by 
\begin{equation}\label{eq:horizontal_vector_field}
    dH = i_X\Om^L = \Om^L(X,\;) \text{ and } G^L_c(X_c,Tc.Y)=0,\quad \forall Y\in \X(S^1)
\end{equation}
which we will denote by $\on{hgrad}^{\Om^L}(H)$. Obviously we then have
\begin{equation}
    \on{grad}^{\bar\Om^L}(\bar H)\o \pi = T\pi\o \on{hgrad}^{\Om^L}(H).
\end{equation}
\rev{Here and in the rest of the article, we write $\on{grad}^A E$ for the vector field satisfying $A(\on{grad}^A E,\cdot)=dE$ for a given non-degenerate bilinear form $A$ such as a Riemannian metric and a symplectic form, and a function $E$.}

Sometimes the kernel of $\Om^L$ will be larger than the tangent spaces to the $\Diff$-orbits; then $\on{hgrad}^{\Om^L}(H)$ will be chosen $G^L$-perpendicular to the kernel of $\Om^L$. This will happen in Theorem \ref{thm:sympl_conf}, for example, where $L$ is a function of $c$ such that $\Th_c^L$ is also invariant under scaling. 
The Hamiltonian $H$  factors to the corresponding space $\Imm(S^1,\mathbb R^3)/\ker\Om^L$  (which  denotes the quotient by the foliation generated by $\ker\Om^L$) if $H$ is additionally invariant under each vector in $\ker\Om^L$.
\end{remark}

\begin{remark}\label{rmk:horizontal_MW_field}
    For the Marsden-Weinstein structure $\OmMW=-d\Th^{\frac{1}{3}\id}$, we have 
    \begin{align*}
        \on{hgrad}^{\OmMW} H = - D_s c \x  \on{grad}^{G^\id} H
    \end{align*}
    since 
    \begin{align*}
        G^\id(D_s c \x \cdot, \cdot)=\OmMW(\cdot,\cdot).
    \end{align*}
\end{remark}

\begin{remark}[Momentum mappings]\label{rem:momentum_general}
If a Lie group $\mathcal G$ acts on $\Imm(S^1,\R^3)$ and  preserves $\Th^L$,
 the corresponding momentum mapping $J$ can be expressed in terms of $\Th^L$ and the fundamental vector field mapping $\ze:\mathfrak g \to \X(\Imm(S^1,\mathbb R^3))$. For $Y\in \mathfrak g$, we have
\begin{equation*}
    \langle J(c),Y\rangle = \Th^L(\ze_Y)_c = \int \langle c\x D_sc, L_c\zeta_Y\rangle ds,
\end{equation*}
where $\langle\cdot,\cdot\rangle: \mathfrak g^* \x \mathfrak g\to \mathbb R$ is \rev{the duality product and $\zeta_Y$ is the fundamental vector field generated by $Y$}.
Namely,
\begin{align*}
    d \Th^L(\ze_Y) &= d i_{\ze_Y} \Th^L = \mathcal{L}_{\ze_Y} \Th^L -  i_{\ze_Y} d\Th^L = 0 - i_{\ze_Y} \Om^L.
\end{align*}

Lemma \ref{lem:Liouville_one_form} asserts that $\Th^L$ is invariant under the right action of $\Diffp$ and the left action of $SO(3)$.

Thus for $X=a.\p_\th\in \X(S^1)= C^\infty(S^1)\p_\th$  the \emph{reparameterization momentum} is given as follows:
\begin{align}
\ze_{a.\p_\th}(c) &= D_{c,a.c_\th} \qquad\text{as derivation at $c$ on }
C^\infty(\Imm, \mathbb R) 
\\&
= a. c_\th = a.|c_\th| D_s c \in T_c\Imm = C^\infty(S^1,\mathbb R^3)
\\
L_{c\o \ph}(h\o \ph) &= (L_ch)\o \ph \implies 
(D_{c,a.c_\th} L_c)(h) + L_c(a.h_\th) = a.(L_ch)_\th
\\
  \langle J^{\Diffp}(c), a.\p_\th\rangle &= \Th^L_c(\ze_{a.\p_\th}(c)) = \Th^L_c (a.c_\th) = 
  \int \langle c\x D_sc, L_c(a.c_\th)\rangle ds
  \\&
  = \int \langle c\x D_sc, a.(L_cc)_\th - (D_{c,a.c_\th} L_c)(c)\rangle ds.
 \end{align}
 For $Y\in \mathfrak {so}(3)$ the \emph{angular momentum} is  
 \begin{align}
 \langle J^{SO(3)}(c), Y\rangle &= \Th^L(Y\o c) = \int \langle c\x D_sc, L_c(Y\o c)\rangle ds 
 \\&
 = \int \langle c\x D_sc, \rev{Y}\o L_c(c) 
 - D_{c,Y\o c}L_c(c)
 \rangle ds
\end{align}
\rev{where $Y \circ c=\zeta_Y(c)$ denotes the multiplication of $Y$ as a matrix with $c$ as a vector.
}
For a correct interpretation of the angular momentum recall (from  \cite[4.31]{Michor08}, e.g.) that the action of  $Y\in \mathbb R^3 \cong \mathfrak{so}(3)\cong L_{\text{skew}}(\mathbb R^3,\mathbb R^3)$ on $\mathbb R^3$ is given by $X\mapsto 2Y\x X$.

If $L$ is also invariant under translations, then the \emph{linear momentum}, for $y\in \mathbb R^3$, is 
\begin{align}
\langle J^{\mathbb R^3}(c), y\rangle &= \Th^L_c(y) = \int \langle c\x D_sc, L_c(y)\rangle ds\,.    
\end{align}
Note that the above also furnishes conserved quantities on $B_i$, if $\bar\Om^L$ is non-degenerate.
\end{remark}


\section{Symplectic structures induced by conformal factors}\label{sec:conformal}
In this section we consider symplectic structures induced by Riemannian metrics, that are conformally equivalent to the $L^2$-metric, i.e.,  we consider the $G^L$ metric for $L_c=\la(c)$ where $\la\colon \Imm(S^1,\R^3)\to \R_{>0}$ is invariant under reparametrization. Thus $\la$  factors to a function $\bar \la\colon B_i(S^1,\R^3)\to \R_{>0}$ by $\pi^* \bar \la= \rev{\bar\la \o \pi =}\la $. Moreover, if $\on{grad}^{G^\id}_c\la$ exists (which we assume) it is pointwise perpendicular to $D_s c$.

We first study the scale invariance of the corresponding Liouville 1-form, which will be of importance for the calculation of the induced (pre)symplectic structure. We say $\Th^L$ is scale-invariant at $c\in \Imm(S^1,\R^3)$ if $\L_I \Th^L_c=0$ where $I\in \Gamma(T\Imm(S^1,\R^3))$ is the scaling vector field $I_c \coloneqq c$  with flow $\on{Fl}^I_t(c)= e^t.c$. Depending on the context, we use both $I$ and $c$ for scaling as a tangent vector in this article.
\begin{lemma}[Scale invariance of $\Th^\la$]
\label{prop:scale_invariance_of_conformal_Liouville}
    Let $L_c=\la(c)\operatorname{id}$. Then the following are equivalent:
    \begin{enumerate}[label=(\alph*)]
        \item \label{item:1} $\Th^\la$ is invariant under scalings.
        \item \label{item:2}$3\la(c) + \L_I\la(c)=3\la(c) + D_{c,c}\la\rev{(c)}=0$ for all $c\in \Imm(S^1,\R^3)$.
        \item \label{item:3} $\la(c) = \La(c/\ell(c)).\ell(c)^{-3}$ for a smooth function $\La: \{c\in \Imm:\ell(c)=1\}\to \mathbb R_{>0}$, 
    \end{enumerate}
    \rev{where $\ell(c)$ is the length of $c$.}
    \end{lemma}
\begin{proof} We have the following equivalences.
\\
$\ref{item:1}\iff \ref{item:2}$:
\begin{align}
       \L_I \Th^\la 
        &=\L_I (\la \Th^\id ) 
        = d i_I (\la \Th^\id ) +i_I d (\la \Th^\id ) 
        = 0 + i_I (d\la \wedge \Th^\id + \la d \Th^\id) \\
        &= i_I d\la \wedge \Th^\id +0 +  \la i_I d\Th^\id
        = (i_I d\la) \Th^\id -\la i_I \Om^\id = (i_I d\la +3\la)\Th^\id.
    \end{align}    
$ \ref{item:2}\iff \ref{item:3}$: Let $\ell(c)=1$.   
\begin{align}
    \p_t\la(tc) &=d\la_{tc}= D_{c,tc}\la = \tfrac1t D_{tc,tc}\la = \tfrac{-3}{t}\la(tc)
    \\\iff & \p_t\log(\la(tc)) = \tfrac{-3}{t} \iff \log(\la(tc)= \log(\La(c)t^{-3}) \iff \la(tc) = \La(c).t^{-3}. \qedhere
\end{align}
\end{proof}
Equipped with the above Lemma we are now ready to calculate the induced symplectic structure $\Omega^\la$, where we will distinguish between the scale-invariant and non-invariant case. 

\begin{theorem}[The (pre)symplectic structure $\Omega^\la$]\label{thm:sympl_conf}
 Let $L_c=\la(c)\operatorname{id}$ be $\Diff$-invariant. Then the induced (pre)symplectic structure on $\Imm(S^1,\mathbb R^3)$ is given by
    \begin{align}\label{eq:Om_lambda}
        \Omega^\la=\la \Omega^\id + \Th^\id\wedge d\la.
    \end{align}
Furthermore we have
\begin{enumerate}[label=(\alph*)]
\item \label{thm:sympl_conf:casea}
If $3\rev{\la(c)} + (\L_I\la)\rev{(c)}=3\la(c)+D_{c,c}\la(c)\neq 0$ on any open subset of $\Imm(S^1,\R^3)$, then  $\Omega^\la$ induces a non-degenerate 2-form on $B_i(S^1,\R^3)$, which is thus symplectic.
\item \label{thm:sympl_conf:caseb} Assume in addition,  that $X:=\on{hgrad}^{\Om^\id}\la$ exists, is smooth, \rev{admits a  flow}, and that $3\la\rev{(c)} + (\L_I\la)\rev{(c)} =0$ for all $c$. 
Denote by $\mathcal F$ the involutive 2-dimensional vector sub-bundle spanned by the vector fields $I$ and $\on{hgrad}^{\Om^\id}\la$.  
Then $\Om^\la$  induces a non-degenerate 2-form on 
$\Imm(S^1,\mathbb R^3)/(\Diffp \x \mathcal F
)$. \rev{If $\L_X\ell = 0$,} 
it is also non-degenerate on $\{\bar c\in B_i(S^1,\R^3): \bar\ell_{\bar c}=1\}/\on{span}(\on{grad}^{\bar\Om^\id}\bar\la)$ \rev{where $\ell =\bar\ell\o \pi$ denotes the length function $\bar\ell$ on $B_i(S^1,\R^3)$.} 
\end{enumerate}
\end{theorem}

\begin{remark}[\rev{Smooth structure of the orbit space}] 
In case \ref{thm:sympl_conf:caseb}, 
the vector field $X:=\on{hgrad}^{\Om^\id}\la$ exists in $\X(\Imm(S^1,\R^3))$ if and only if  $\on{grad}^{G^{\on{id}}}\la$ exists and is smooth as we have $\on{hgrad}^{\Om^{\on{id}}}\la = \on{hgrad}^{3\OmMW}\la= -\frac{1}{3}D_s c \x \on{grad}^{G^\id}\la$. This is equivalent to the fact that $\bar\la\in C^\infty(B_i(S^1,\mathbb R^3),\mathbb R)$ by \ref{app:om-smooth}.
Moreover, the vector fields $I$ and 
$\on{hgrad}^{\Om^\id}\la(=\frac{1}{3}\on{hgrad}^{\OmMW}\la)$ are linearly independent at any $c$
because $\Om^{\id}_c(\on{hgrad}^{\Om^\id}_c\la, I_c) = i_I d\la(c) = -3\la(c) \ne 0$ by  assumption.   So the dimension of $\mathcal F$ is always 2.
We project to the leaf space of the 2-dimensional distribution if it is integrable. This is the case, if the flow of \rev{$X=\on{hgrad}^{\Om^\id}_c\la$ and thus also of} $\on{grad}^{\bar\Om^\id}\bar\la$ exists; then the flows of $I$ and $\on{grad}^{\bar\Om^\id}\bar\la$ combine to a 2-dimensional $(ax+b)$-group acting on $\Imm(S^1,\R^3)$. We assume that this is the case; to prove existence of the flow one has first to specify $\la$ and then solve a non-linear PDE. 

\rev{Furthermore, we note that the smooth structure of the corresponding quotient space $\Imm(S^1,\mathbb R^3)/(\Diffp \x \mathcal F$ is slightly subtle: it is always a Fr\"olicher space with  tangent bundle, see \cite[Section 23]{kriegl_michor_1997convenient}. If local smooth sections of the projection to the leaf space exist and if the $(ax+b)$-orbits admit slices, then we get a principal bundle with structure group the $(ax+b)$-group in the category of orbifolds, so the leaf space is also an orbifold.}
\end{remark}

\begin{proof}
    The formula directly follows from the product rule applied to $d(\Th^\la)=d(\la \Th^\id)$.

    Case \ref{thm:sympl_conf:casea}: We now show the non-degeneracy; if a tangent vector $h$ satisfies  $h\perp D_s c$ pointwise and $\Om^\la_c(h,k)= 0$ for any $k$, then $h=0$.
    First, choosing $k=a.c$ with some non-zero constant $a\in \R^\x$ we get from $a.c\in \ker \Th^\id$ that, 
    \begin{align}
        0=\Om^\la_c(h,ac)=\la\Om^\id(h,ac)+\Th^\id(h)i_{a.c}d\la-0=a [3\la + D_{c,c}\la]\Th^\id(h).
    \end{align}
    With our assumption $3\la + D_{c,c}\la\neq 0$ we see $h\in\ker \Th^\id$. 

    Next, we test for $h\in\ker \Th^\id$ and $k=a.c$ with some function $a\in C^\infty(S^1)$ to see 
    \begin{align}
        \Om^\la_c(h,a.c)
        =\la\Om^\id(h,a.c)
        =3\la\int a\langle c \x D_s c, h\rangle ds.
    \end{align}
    If this vanishes for any function $a$, we have $\langle c \x D_s c, h\rangle=0$ everywhere. 
    We now consider the regions:
\begin{enumerate}
    \item[(i)] The open subset $U = \{\th\in S^1: c(\th)\x D_sc(\th)\ne 0 \}$,
    \item[(ii)] The closed set  $S^1\setminus U=\{\th\in S^1: c(\th)\x D_sc(\th) = 0\}$.
\end{enumerate}
    Any $h$ satisfying both $h\perp D_s c$ and $h \perp  (c \x D_s c)$ pointwise is of the form $h=b.c+v$ with a function $b\in C^\infty(S^1)$ supported on $U$ and a vector field $v\in C^\infty(S^1,\R^3)$ supported on $S^1\setminus U$ and $v\perp D_s c$ (and hence $v\perp c$ as well).
    Then we have
    \begin{align}
        \Om^\la_c(h,k)
        &=\la \Om^\id(h,k)+\Th^\id(h)i_k d\la-\Th^\id(k)i_h d\la\\
        &=\la \Om^\id(b.c,k)+0-\Th^\id(k)i_{b.c}d\la\\
        &\quad + \la \Om^\id(v,k)+0-\Th^\id(k)i_v d\la\\
        &=\int_{S^1} \langle (3\la.b + D_{c,b.c}\la + D_{c,v}\la) D_s c \x c+ 3 \la. D_s c \x v, k\rangle ds.
    \end{align}
    We assumed that $\Om^\la_c(h,k)=0$ for all $k$, in particular, for ones supported on $S^1\setminus U$. Hence we have $v\equiv 0$. With this we have
    \begin{align}
        \Om^\la_c(h,k)=\int_{U} (3\la.b + D_{c,b.c}\la)  \langle D_s c \x c, k\rangle ds.
    \end{align}
     In order that $ \Om^\la_c(h,k)=0$ for any $k$, we must have $3\la.b + D_{c,b.c}\la\equiv 0$ on $U$. Since $D_{c,b.c}\la \in \mathbb R$ is constant, $b$ is constant. Hence we have  $b(3\la + D_{c,c}\la)\equiv 0$ and get $b\equiv 0$ from our assumption $3\la + D_{c,c}\la\neq 0$. Thus we obtained $h=0$.

    Case \ref{thm:sympl_conf:caseb}: 
By assumption $X:=\on{hgrad}^{\Om^\id}\la$ exist; i.e., $d\la$ is in the image of $\Om^\id: T\Imm \to T^*\Imm$ and satisfies 
$d\la = i_{X}\Om^\id$ and $\langle X, D_sc\rangle = 0$. 
Then we see $X\in \ker \Om_c^\id$ by direct computation using the assumed condition $3\la_c+D_{c,c}\la=0$;
\begin{align}
(i_{X}\Th^\id)_c &= \int \langle c\x D_sc, X\rangle ds = \tfrac13 \Om^\id_c(X_c,c) = \tfrac13 i_I d\la_c = -\la(c)
\text{ by \ref{prop:scale_invariance_of_conformal_Liouville}.}
\\  
(i_{X}\Om^\la)_c &= i_{X_c}(\la.\Om^\id + \Th^\id \wedge d\la)_c =\la(c). i_{X_c}\Om^\id_c + \Th^\id_c(X_c).d\la_c -i_{X_c}d\la\rev{(c)}.\Th^\id_c
\\&
=\la(c).d\la_c - \la(c).d\la_c - 0 = 0. 
\end{align} 

Note also that the scaling field $I_c\coloneqq c$ with flow $\on{Fl}^I_t (c) = e^t.c$ is  in the kernel of $\Om^ \lambda_c$ as we have
\begin{align}
i_I\Om^\la_c=\la i_I \Om^\id_c +\Th^\id(c)d\la-D_{c,c}\la \Th^\id=-3\la \Th^\id +0-D_{c,c}\la\Th^\id=0.
\end{align}
Thus  $\bar I$ and $\bar X$, the $\pi$-related versions of $I$ and $X$, are in the kernel of $\bar\Om^\la$. 
\begin{align}
(\L_I\la)(c) &= d\la(c) = -3\la(c) 
\\
\L_I\Th^\id &= i_I d \Th^\id=-i_I \Om^\id=3\Th^\id
\\
\L_I\Om^\id &= -\L_I d\Th^\id  = - d\L_I \Th^\id =  -d(3 \Th^\id) = 3\Om^\id
\\
-3 d\la &= \L_I d\la = \L_I (i_X\Om^\id) = (i_X \L_I + i_{[I,X]})\Om^\id = 3 i_X\Om^\id + i_{[I,X]}\Om^\id
\\
i_{[I,X]}\Om^\id &= -6 d\la = -6 i_X\Om^\id
\end{align}
Thus $i_{[I,X]+6X}\Om^\id = 0$, so $[I,X]+6X$ is in the kernel of $\Om^\id$. Their $\pi$-related version $[\bar I,\bar X] + 6\bar X$ is in the kernel of $\bar \Om^\id$ which is weakly non-degenerate on $B_i$. So $[\bar I,\bar X]= -6\bar X$ and also $[I,X]=-6 X$.
Thus if the Frobenius integrability theorem applies in this situation (equivalently,  if the local flow of $X$ exists), then the fields $I$ and $X$ span an integrable distribution, and the leaf space exists. 

Now we shall make use of $\rev{\bar\la}(\bar c)= \La(c/\rev{\ell(c)}).\rev{\ell(c)}^{-3}$, \rev{where $\bar c = \pi(c)\in B_i$.}
The function $\La$ is defined on the $\ell$-unit sphere $\{c\in \Imm: \ell(c) =1\}$. To simplify notation, extend it constantly to $\Imm$ so that $\La(c) = \La|_{\{\ell=1\}}(c/\ell(c))$, \rev{and we let $\la = \bar\la\o \pi$ and $\La = \bar \La\o\pi$}. Then we have
\begin{align}
d\la_c(h) &= \ell(c)^{-3}\Big(d\La_c(h) - 3\La(c)\tfrac1{\ell(c)}\int\langle D_sh,D_sc\rangle ds\Big)
\\&
=d\La(\tfrac{c}{\ell(c})\Big(-\ell(c)^{-2}.\int\langle D_sh,D_sc\rangle ds.c + \ell(c)^{-1}h \Big) -  3\La(\tfrac{c}{\ell(c})\ell(c)^{-4}.\int\langle D_sh,D_sc\rangle ds
\\
\Om^\la_c(h,k) &= \la(c)\Om^\id_c(h,k) + (\Th^\id_c\wedge d\la_c)(h,k) 
\\&
= \La(c/\ell(c)).\ell(c)^{-3}.\Om^\id_c(h,k) 
+ \ell(c)^{-3}\Th^\id_c(h).\Big(d\La_c(k) - 3\La(c)\tfrac1{\ell(c)}\int\langle D_sk,D_sc\rangle ds\Big)
\\&\qquad
- \ell(c)^{-3}\Big(d\La_c(h) - 3\La(c)\tfrac1{\ell(c)}\int\langle D_sh,D_sc\rangle ds\Big).\Th^\id_c(k).
\end{align}

    We have   diffeomorphisms which are equivariant under scalings
\begin{align}
    \Imm(S^1,\mathbb R^3)/\Diffp &\cong \Imm(S^1,\mathbb R^3)/(\Diffp\x \mathbb R_{>0})\x \mathbb R_{>0}
    \\&
    \cong \{\bar c\in \Imm(S^1,\mathbb R^3)/\Diffp: \rev{\bar\ell}(\bar c) =1\}\x \mathbb R_{>0}
    \\&
    \cong \{\bar c\in \Imm(S^1,\mathbb R^3)/\Diffp: \rev{\bar\la}(\bar c) =1\}\x \mathbb R_{>0}
    \\
    \bar c\quad &\longleftrightarrow\quad \big(\frac1{\rev{\bar\ell}(\bar c)}\bar c, \rev{\bar\ell}(\bar c)\big)
    \longleftrightarrow \big(\bar\La(\bar c/\rev{\bar \ell(\bar c)})^3\bar c, \rev{\bar\ell}(\bar c)\big)
\end{align}
and pre-symplectomorphisms  
\begin{align}
 &(\{\bar c\in B_i: \rev{\bar\ell}(\bar c) =1\},\bar\Om^\la)\ni \bar c\mapsto F(\bar c)=\rev{\bar\La}(\bar c)^{1/3}\bar c\in (\{\bar c\in B_i: \rev{\bar\la}(\bar c) =1\},\bar\Om^\id)  
 \\
 &(\{\bar c\in B_i: \rev{\bar\ell}(\bar c) =1\},\bar\Om^\la)
 \xhookrightarrow{i_{\ell}} (B_i, \bar \Om^\la)
 \\
 &(\{\bar c\in B_i: \rev{\bar\la}(\bar c) =1\},\bar\Om^\id)
 \xhookrightarrow{i_{\la}} (B_i, \bar \Om^\la)\quad \text{since}
\\
&dF\rev{_c}(k) = \tfrac13\La(c)^{-2/3}d\La(c)(k).c + \La(c)^{1/3} k
\\
&D_s F(c) = \tfrac13\La(c)^{-2/3}d\La(c)(D_sc).c + \La(c)^{1/3} D_sc
\\
&(F^*\Om^\id)_c(h,k) = \Om^\id_{F(c)}(dF\rev{_c}(h),dF\rev{_c}(k))
\\&
= 3\int\Big\langle 
(\tfrac13\La(c)^{-2/3}d\La\rev{_c}(D_sc).c + \La(c)^{1/3} D_sc) \x 
\x (\tfrac13\La(c)^{-2/3}d\La\rev{_c}(h).c + \La(c)^{1/3} h),
\\&\qquad\qquad\qquad
(\tfrac13\La(c)^{-2/3}d\La\rev{_c}(k).c + \La(c)^{1/3} k)
\Big\rangle\,ds
\\&
= \La(c)\Om^\id_c(h,k) +\Th^\id_c(h).d\La(c)(k) - \Th^\id_c(k).d\La(c)(h)
\end{align}
Since $(B_i,\bar\Om^\id)$ is weakly symplectic and  $\{\bar c\in B_i: \bar\la(\bar c) =1\}$ is a codimension 1 sub-orbifold diffeomorphic to $\{\bar c\in B_i: \bar\ell(\bar c) =1\}$, the kernel of $(i_\la^*\bar\Om^\id_{\bar c})$ is 1-dimensional, and we have already found it as $\bar X= \on{grad}^{\bar\Om^\id}\bar\la$ which is tangent to $\{\bar c\in B_i: \bar\la(\bar c) =1\}$.
\end{proof} 

\begin{remark}[Symplectic reduction]
    Our reduction of the space $B_i(S^1,\R^3)$ in the second case of \autoref{thm:sympl_conf} can be  seen as an infinite-dimensional instance of the Marsden-Weinstein-Meyer symplectic reduction. 
    To see this, let us set $\bar X\coloneqq \on{grad}_{\bar c}^{\bar \Om^{\on{id}}} \bar \lambda$ and take the momentum map $\bar J\colon B_i(S^1,\R^3)\to \R$ by $\bar J(\bar c)\coloneqq \bar \lambda(\bar c)$, with the corresponding group action being the time-$t$ flow of $\bar X$ with $\bar c$ as initial data. We have shown that $\bar \Omega^{\lambda}$ is degenerate on the codimension-1 sub-orbifold $\bar J^{-1}(1)=\{\bar c \in B_i(S^1,\R^3) \mid \bar \lambda(\bar c)=1\}$, and that it becomes symplectic when factored onto the codimension-2 sub-orbifold $\bar J^{-1}(1)/ \on{grad}^{\bar \Om^{\on{id}}} \bar \lambda$.
    
    We also remark that the dual product for the momentum map $\bar J$ is just the multiplication of scalar values as we have
    \begin{align}
        \langle \bar J(\bar c), t\rangle = -\bar \Th_{\bar c}^\id(t.\bar X)=t. \bar \lambda(\bar c)
    \end{align}
    for $t\in \R$ such that the time-$t$ flow map of $\bar X$ exists.  Here we used the invariance of $\bar \Th_{\bar c}^\id$ under the flow of $\bar X$, which is shown by $\L_{\bar X} \bar \Th_{\bar c}^\id=d i_{\bar X} \bar \Th_{\bar c}^\id + i_{\bar X} \bar \Om_{\bar c}^\id= -d\bar \lambda +d\bar \lambda=0$ mimicking computations in the proof of \autoref{thm:sympl_conf}. We may get the same result also using $\bar \Th^\lambda$ and $ \on{grad}_{\bar c}^{\bar \Om^\lambda} \bar \lambda = T_c \pi (\on{hgrad}_c^{\Om^\lambda} \lambda)$ (cf. Proposition \ref{prop:hamVF_conformal}) instead of $\bar \Th^{\id}$ and $\bar X= \on{grad}_{\bar c}^{\bar \Om^{\on{id}}} \bar \lambda$.
\end{remark}

\begin{remark}[A pseudo-Riemannian metric via $\Om^L$ and $\J=D_s c \x \cdot$]
Using the presymplectic form $\Om^L$ and the \rev{standard} almost complex structure
\begin{align*}
    \J \colon T_{\rev{\bar c}}B_i(S^1,\R^3) &\to T_{\rev{\bar c}} B_i(S^1,\R^3)\\
    \rev{\bar h} &\mapsto \rev{\overline{D_s c \times h}},
\end{align*}
we may define a pseudo-Riemannian metric $\rev{\bar R}$, \rev{which} is compatible with $\bar\Om^L$ \rev{via $\J$}. Note that such \rev{$\bar R$} is different from the Riemannian metric $G^L$ \rev{factored onto $B_i(S^1,\R^3)$} we used to define the Liouville form $\Th^L$.

We here compute \rev{$\bar R$} for the conformal factor $L_c=\la(c)$. In the computation, we identify the tangent space at $\rev{\bar c}$ of $B_i(S^1,\mathbb R^3)$ with the space of tangent vectors $h$ in $T_c\Imm(S^1,\mathbb R^3)$, 
such that $\langle D_sc,h\rangle=0$, \rev{and denote $J=D_s c\times \cdot$}.

We then have
\begin{align*}
\rev{\bar  R_{\bar c}}(\rev{\bar h,\bar k})&
\coloneqq \Om^{\la}_c(h,\rev{J} k)
=\la(c)\Omega_c^{\operatorname{id}}(h,\rev{J} k)
+\Th_c(h) \L_{\rev{J} k} \la(c)-\Th_c(\rev{J} k) \L_h \la(c).
\end{align*}
By design $\rev{\bar R}$ is non-degenerate. The symmetry follows from $\Om_c^\la(\rev{J} h,k)=-\Om_c^\la(h,\rev{J} k)$. It is, however, not clear if \rev{$\bar R$}  is positive-definite, i.e., if it is a Riemannian metric. We leave this question open for future research. 
\end{remark}

\subsection{Hamiltonian vector fields}\label{sec:hamiltonian_conformal}
Now we compute the horizontal Hamiltonian vector field $\on{hgrad}^{\Om^\la}H$ for a given reparametrization-invariant Hamiltonian $H$. We express $\on{hgrad}^{\Om^\la}H$ in terms of $\on{grad}^{G^\id}H$ since the latter is in general relatively easy to obtain.

\begin{proposition}[Horizontal Hamiltonian vector fields for $\Om^{\la}$]
\label{prop:hamVF_conformal}  
Assume that $\on{grad}^{G^\id}\la$ exists.
 \begin{enumerate}[label=(\alph*)]
     \item \label{prop:hamVF_conformal_casea} Consider a $\Diffp$-invariant Hamiltonian $H\colon \Imm(S^1,\R^3)\to \R^3$. If $3\la + \L_I \la\neq 0$ on any open subset of $\Imm$ then
\begin{align}
    \on{hgrad}^{\Om^\la}H
    =-\frac{1}{3\rev{\lambda(c)}} \Bigg\{  D_s c\x \on{grad}^{G^\id} H& \\
    +  \frac{1}{3\rev{\lambda(c)}+D_{c,c}\rev{\lambda(c)}} \Big[& \langle \on{grad}_c^{G^\id}\la, D_s c\x \on{grad}^{G^\id}H\rangle_{L^2_{ds}(S^1)} D_s c \x (D_s c \x c) \\
     &- \langle  c, \on{grad}^{G^\id} H\rangle_{L^2_{ds}(S^1)} D_s c\x \on{grad}_c^{G^\id}\la \Big]\Bigg\}\,.
\end{align}
\item \label{prop:hamVF_conformal_caseb}Consider a Hamiltonian $H\colon \Imm(S^1,\R^3)\to \R^3$ invariant under $\Diffp$ and the flows of the scaling vector field $I$ and $\on{hgrad}^{\OmMW}\la= -D_sc\x \on{grad}^{G^\id}\la$. If $3\la\rev{(c)} + \rev{\L_{c,I}\la(c)} =0$ for all $c$ then $\on{hgrad}^{\Om^\la}_cH$ is the orthonormal projection of 
\begin{align}
    X^H_c
    =-\frac{1}{3\la_c} D_s c\x \on{grad}^{G^\id}_c H
    = \frac{1}{\la_c} \on{hgrad}^{\Om^\id}_c H
\end{align}
to the $G^\id_c$-orthogonal complement of the kernel of $\Om^\la$, which is spanned by $I$, $\on{hgrad}^{\OmMW}\lambda$, and $\{a.D_s c \mid a\in C^\infty(S^1)\}$, namely
\begin{align}
\on{hgrad}^{\Om^\la}_cH &=\frac{1}{3\rev{\lambda(c)}} \Big(- D_s c\x \on{grad}^{G^\id}_c H + a_c.(1-\on{pr}_c)I_c - b_c. D_s c \x \on{grad}^{G^\id}\la\Big)
\end{align}
where the pair $(a_c,b_c)\in \R^2$ is given by
\begin{align}
\begin{pmatrix}
a_c \\
b_c
\end{pmatrix}
=
\begin{pmatrix}
\langle v,v\rangle_{L_2} & \langle v,w\rangle_{L_2} \\
\langle v,w\rangle_{L_2} &\langle w,w\rangle_{L_2}
\end{pmatrix}^{-1}
\begin{pmatrix}
\langle u,v\rangle_{L_2}\\
\langle u,w\rangle_{L_2}
\end{pmatrix}.
\end{align}
with
\begin{align}
    &u=-D_s c \x \on{grad}^{G^\id}H=\on{hgrad}^{\OmMW}H\\
    &v=(1-\on{pr}_c)I_c\\
    &w=-D_s c \x \on{grad}^{G^\id}\la=\on{hgrad}^{\OmMW}\la
\end{align}
where the matrix appearing here is invertible because $v_c$ and $w_c$ are linearly independent at every $c\in \Imm(S^1,\R^3)$.

\end{enumerate}
\end{proposition}
Note that in the scale-invariant case (Case \ref{prop:hamVF_conformal_caseb}), the flow of the field $Y^H$ projects to the Hamiltonian flow of $\bar H$ on  $\{\bar c\in B_i(S^1,\R^3): \bar\la\rev{(\bar c)}=1\}/\on{grad}^{\bar\Om^\id}\bar\la$ with respect to a multiple of the Marsden-Weinstein symplectic structure. 

\begin{proof}
Let us denote for simplicity $A\coloneqq \on{grad}^{G^\id}\la$ and $X_H\coloneqq \on{hgrad}^{\Om^\la}H$. 
We can isolate out $k$ from $\Om^\la_c (X_H,k)$ by
    \begin{align}
        \Om^\la (X_H,k)
        &=\la \Om^\id(X_H,k)+\Th^\id(X_H)D_{c,k}\la - \Th^\id(k)D_{c,X_H}\la\\
        &=\int_{S^1} \langle 3\la. D_s c \x X_H - D_{c,X_H}\la. c \x D_s c +\Th^\id (X_H) A, k\rangle ds.
    \end{align}
    Using $\Om^\la_c (X_H,k)=dH(k)=G^\id (\on{grad}^{G^\id} H,k)$, we get 
    \begin{align}
        0 &= \Om^\la (X_H,k) - dH(k)\\
        &=\int_{S^1} \langle 3\la. D_s c \x X_H - D_{c,X_H}\la. c \x D_s c +\Th^\id (X_H) A - \on{grad}^{G^\id} H, k\rangle ds.
    \end{align}
    This must be satisfied for any $k$, namely we have
    \begin{align}\label{eq:equation_with_XH}
        3\la.D_s c \x X_H - D_{c,X_H}\la. c \x D_s c +\Th^\id (X_H) A - \on{grad}^{G^\id} H=0.
    \end{align}
    Our goal is to solve this for $X_H$. Applying $- D_s c \x$ reads
    \begin{align}\label{eq:equation_in_XH}
       3\la. X_H - D_{c,X_H}\la. D_s c \x( D_s c \x c) -\Th^\id (X_H) D_s c\x A + D_s c \x \on{grad}^{G^\id} H=0.
    \end{align}
    
    Let us set 
    \begin{align}\label{eq:temporal_XH}
        X_H=\frac{-1}{3\la} D_s c \x \on{grad}^{G^\id} H + K_1 D_s c \x( D_s c \x c) +K_2 D_s c \x A_c
    \end{align}
    with some coefficients $K_1, K_2$ to be determined. 
    
    From 
    \begin{align}
        D_{c,X_H}\la=\int \langle A_c,X_H\rangle ds, \quad
        \Th^\id(X_H)=\int \langle c\x D_s c, X_H \rangle ds,
    \end{align}
    we get 
    \begin{align}
        0
        &=3 \la K_1. D_s c \x (D_s c \x c) + 3\la K_2. D_s c \x A_c\\  
        &\qquad - \int \langle A_c, \frac{-1}{3\la} D_s c \x \on{grad}^{G^\id} H + K_1 D_s c \x( D_s c \x c) \rangle ds. D_s c \x (D_s c \x c)\\
        & \qquad - \int \langle A_c, \frac{-1}{3\la} D_s c \x \on{grad}^{G^\id} H + K_2 D_s c \x A_c\rangle ds. D_s c \x A_c\\
        &= \Big[ K_1  \left(3\la-\int \langle A_c, D_sc \x (D_s c \x c)\rangle ds \right) 
        \\&\qquad\qquad\qquad  
        +\frac{1}{3\la}\int \langle A_c,D_s c \x \on{grad}^{G^\id} H\rangle ds \Big] D_s c \x (D_s c \x c)
        \\&\qquad 
        + \Big[ K_2  \left(3\la+\int \langle D_ s c \x c, D_sc \x A_c \rangle ds \right) 
        \\&\qquad  
        -\frac{1}{3\la}\int \langle D_s c \x c,  D_s c \x \on{grad}^{G^\id} H\rangle ds \Big] D_s c \x A_c 
        \\& \label{eq:temporal_XH_with_K} 
        =\left[ K_1  \left(3\la+D_{c,c}\la \right) +\frac{1}{3\la}\int \langle A_c,D_s c \x \on{grad}^{G^\id} H\rangle ds \right] D_s c \x (D_s c \x c)
        \\&\qquad  
        + \left[ K_2  \left(3\la+D_{c,c}\la \right) -\frac{1}{3\la}\int \langle D_s c \x c,  D_s c \x \on{grad}^{G^\id} H\rangle ds \right] D_s c \x A_c.
    \end{align}
    In the last step we used  
    \begin{align}
        -\int \langle  D_sc \x (D_s c \x c), A_c\rangle ds  
        = \int \langle D_ s c \x c, D_sc \x A_c \rangle ds 
        = D_{c,(1-\on{pr}_c) c}\la
        =D_{c,c}\la
    \end{align}
    where the last equality is due to the reparametrization-invariance of $\la$.

    Case (a): 
    Observe that 
    \begin{align}
        K_1&=-\frac{1}{(3\la+D_{c,c}\la)3\la}\int \langle A_c, D_sc \x (D_s c \x c)\rangle ds, 
        \\
        K_2&=\frac{1}{(3\la+D_{c,c}\la)3\la}\int \langle D_s c \x c,  D_s c \x \on{grad}^{G^\id} H\rangle ds
    \\&
    =\frac{1}{(3\la+D_{c,c}\la)3\la}\int \langle  c,   \on{grad}^{G^\id} H\rangle ds\quad\text{since } \on{grad}^{G^\id} H \bot D_sc.
    \end{align}
    satisfy the equality. Substituting $K_1$ and $K_2$ to \eqref{eq:temporal_XH}, we obtain the stated formula. 
    Note that the choice of the pair $(K_1,K_2)$ is unique since $ -(1-\on{pr}_c)I_c=D_sc \x (D_s c \x c)$ and $-\on{hgrad}^{\OmMW}\la= D_s c \x A_c$ are linearly independent at least for some $\theta$, namely in a small neighborhood. This follows from the linear independence of these two tangent vectors on $T_c\Imm(S^1,\R^3)$, which is seen by the argument in the comment after Theorem \ref{thm:sympl_conf} \ref{thm:sympl_conf:caseb} with the reparametrization invariance of $\Om^\id$.


     Case (b): By assumption $3\la+D_{c,c}\la=0$ we see from \eqref{eq:temporal_XH_with_K} that,
     \begin{multline}
         0=\left[\int \langle A_c,D_s c \x \on{grad}^{G^\id} H\rangle ds \right] D_s c \x (D_s c \x c) - 
        \\
        -\left[\int \langle D_s c \x c,  D_s c \x \on{grad}^{G^\id} H\rangle ds \right] D_s c \x A_c.
     \end{multline}
     Using this equality, it is easy to check that
     \begin{align}
         X_H\coloneqq \frac{-1}{3\la} D_s c \x \on{grad}^{G^\id} H
     \end{align}
     satisfies \eqref{eq:equation_with_XH}. At this point there are up to two degrees of freedom in vector fields that satisfy \eqref{eq:equation_with_XH}.
     We can make $X_H$ the unique horizontal lift  of $\on{grad}^{\bar\Om^{\la}}\bar H$ by performing the $G_c^\id$-orthogonal projection with respect to $(1-\on{pr}_c)I_c$ and $\on{hgrad}_c^{\OmMW}\lambda$, and hence obtain the stated expression. The resulting vector field $X_H$ is $G^\id$-orthogonal to $\{a.D_sc \mid a \in C^\infty(S^1)\}$, $\on{hgrad}^{\OmMW}\la$ and $I_c$. 
\end{proof}

\section{Symplectic structures induced by length weighted metrics}\label{sec:symp_length}
 Next we study a special class of symplectic structures induced by conformal factors introduced in the previous section; namely we consider length-weighted metrics as studied in~\cite{yezzi2005conformal,michor_mumford2006,shah2008H0}. More precisely, we consider  operators of the form  $L_c=\Ph(\rev{\ell(c)})$ where $\rev{\ell(c)}=\int_{S^1}|c_\theta|d\theta$ denotes the length of the curve $c$ and $\Ph: \mathbb R_{> 0}\to \mathbb R_{> 0}$ is a suitable function.
 Using Theorem~\ref{thm:sympl_conf} we obtain the following result concerning the induced symplectic structure~$\Omega^{\Ph(\ell)}$:

\begin{corollary}[The (pre)symplectic structure $\Om^{\Ph(\ell)}$]\label{cor:length_symp}
Let $\Ph\in \rev{C^\infty(\R_{>0},\R_{> 0})}$. The induced (pre)symplectic structure of the $G^{\Ph(\ell)}$-metric is given by:
\begin{align} \label{eq:Om_Phi(ell)}
\Om^{\Ph(\ell)}_c(h,k)
     &=\Ph(\rev{\ell(c)})\Omega^{\operatorname{id}}(h,k)
     -\Ph'(\rev{\ell(c)})\left(\int_{S^1} \langle D_sh,D_sc\rangle ds\; \Theta^{\operatorname{id}}(k)
     -\int_{S^1} \langle D_sk,D_sc\rangle ds\; \Theta^{\operatorname{id}}(h) \right)\\
      &=\Ph(\rev{\ell(c)})\Omega^{\operatorname{id}}(h,k)
     +\Ph'(\rev{\ell(c)})\left(\int_{S^1} \langle h,D^2_sc\rangle ds\; \Theta^{\operatorname{id}}(k)
     -\int_{S^1} \langle k,D^2_sc\rangle ds\; \Theta^{\operatorname{id}}(h) \right).
\end{align}
Furthermore, we have:
\begin{enumerate}[label=(\alph*)]
\item \label{cor:length_symp:casea} If $\Ph(\ell)\neq C\ell^{-3}$
 then the presymplectic structure $\bar\Om^{\Ph(\ell)}$ on $B_i(S^1,\mathbb R^3)$ is non-de\-ge\-nerate and thus symplectic. 
\item \label{cor:length_symp:caseb}
If $\Ph(\ell) = C\ell^{-3}$, then $\Om^\la$ induces a non-degenerate 2-form on $\Imm(S^1,\mathbb R^3)/(\Diffp \x \F) \simeq \{\bar c\in B_i(S^1,\R^3): \bar\ell=1\}/\on{span}(\on{grad}^{\bar\Om^\id}\bar\ell)$, where it agrees with a multiple of the Marsden-Weinstein symplectic structure. Here $\F$ is the 2-dimensional vector subbundle spanned by the scaling vector field $I$ and $\on{hgrad}^{\OmMW}\ell=D_s c \x D_s^2 c$.
\end{enumerate}
\end{corollary}
The Liouville form $\Th^{C\ell^{-3}}$ is invariant under the scaling action $c\mapsto a.c$ for $a\in \mathbb R_{>0}$, which is equivalent to $\L_I\Th^{C\ell^{-3}}=0$. 
Note also that we have a  diffeomorphism which is equivariant under scalings:
\begin{align}
    \Imm(S^1,\mathbb R^3)/\Diffp &\cong \Imm(S^1,\mathbb R^3)/(\Diffp\x \mathbb R_{>0})\x \mathbb R_{>0}
    \\&
    \cong \{\bar c\in \Imm(S^1,\mathbb R^3)/\Diffp: \rev{\bar\ell}(\bar c) =1\}\x \mathbb R_{>0}
    \\
    \bar c\quad &\longleftrightarrow\quad \big(\frac1{\rev{\bar\ell}(\bar c)}\bar c, \rev{\bar\ell}(\bar c)\big)
\end{align}

\begin{proof}
To calculate the formula for $\Omega^{\Ph(\ell)}$ we first need to calculate the variation of the length $\rev{\ell(c)}$. We have:
\begin{equation*}
D_{c,h}\rev{\ell(c)}=\int_{S^1} \langle D_sh,D_sc\rangle ds, \qquad D_{c,h}\Ph(\rev{\ell(c)}) = \Ph'(\rev{\ell(c)})\int_{S^1} \langle D_sh,D_sc\rangle ds\,.
\end{equation*}
Applying this to \eqref{eq:Omega_L} using integration by parts, we get 
\begin{align*} 
\Om^{\Ph(\ell)}_c(h,k)
&= 
    \int_{S^1}2\Ph(\rev{\ell(c)})\langle D_s c,  h \x k\rangle 
     - \Ph(\rev{\ell(c)})\langle c, D_s h \x  k - D_s k \x  h \rangle ds
    \nonumber\\&\qquad\qquad
     - \int_{S^1}\langle c\x D_s c, (D_{c,h}\Ph(\rev{\ell(c)}))k \rangle ds +\int_{S^1}\langle c\x D_s c,  (D_{c,k}\Ph(\rev{\ell(c)}))h\rangle ds\nonumber\\
     &= 
    3\Ph(\rev{\ell(c)})\int_{S^1}\langle D_s c,  h \x k\rangle ds
     -\Ph'(\rev{\ell(c)}) \int_{S^1} \langle D_sh,D_sc\rangle ds \int_{S^1}\langle c\x D_s c, k \rangle ds 
    \nonumber\\&\qquad\qquad
     +\Ph'(\rev{\ell(c)})\int_{S^1} \langle D_sk,D_sc\rangle ds\int_{S^1}\langle c\x D_s c,  h\rangle ds\nonumber\\
     &=\Ph(\rev{\ell(c)})\Omega^{\operatorname{id}}(h,k)
     -\Ph'(\rev{\ell(c)})\int_{S^1} \langle D_sh,D_sc\rangle ds\; \Theta^{\operatorname{id}}(k)
     +\Ph'(\rev{\ell(c)}) \int_{S^1} \langle D_sk,D_sc\rangle ds\; \Theta^{\operatorname{id}}(h) ,
\end{align*}
which proves the first formula for $\Omega$. 
We may directly draw the last expression applying \eqref{eq:Om_lambda} to $\la=\Phi(\ell)$.

Case \ref{cor:length_symp:casea}: 
It follows from Theorem \ref{thm:sympl_conf} \ref{thm:sympl_conf:casea} that $\ker\Omega^{\Phi(\ell)}=\{a. D_s c\mid a \in C^\infty(S^1)\}$,
namely $\Omega^{\Ph(\ell)}$ induces a symplectic form $\bar\Om^{\Ph(\ell)}$ on $B_i(S^1,\mathbb R^3)$. 

Case \ref{cor:length_symp:caseb}: By direct computation, we have $\on{hgrad}^{\Om^\id}C\ell^p=3 Cp \ell^{p-1} D_s c \x D_s^2 c$, which is a constant multiple of the Marsden-Weinstein flow $\on{hgrad}^{\OmMW}\ell=D_s c \x D_s ^2 c$, so these two vector fields span the same distribution. Now the statements follow directly from Theorem~\ref{thm:sympl_conf} \ref{thm:sympl_conf:caseb}.
\end{proof}

Now we will compute Hamiltonian vector fields. Therefore we  note that the conditions of Remark~\ref{horizontalHamiltonian} are satisfied, which allows us to obtain the following result:

\begin{corollary}
    [Horizontal Hamiltonian Vector Fields for $\Om^{\Ph(\ell)}$]
\label{cor:hamVF_Gphi}
    Consider a $\Diffp$-invariant Hamiltonian $H\colon \Imm(S^1,\R^3)\to \rev{\R}$. 
\begin{enumerate}[label=(\alph*)]
     \item \label{cor:hamVF_Gphi_casea}  If $\Ph(\ell)\ne C\ell^{-3}$, then:
\begin{multline}\label{eq:ham_flow_length_weight}
    \on{hgrad}_c^{\Om^{\Ph(\ell)}}H
    =\frac{1}{3\Phi(\rev{\ell(c)})} \Bigg\{ - D_s c\x \on{grad}^{G^\id} H \\
    +  \frac{\Phi'(\rev{\ell(c)})}{3\Phi(\rev{\ell(c)})+\Ph'(\rev{\ell(c)})\rev{\ell(c)}} \Big[ \langle D_s^2 c, D_s c\x \on{grad}^{G^\id} H\rangle_{L^2_{ds}(S^1)} D_s c \x (D_s c \x c) \\
     + \langle  c,  \on{grad}^{G^\id} H\rangle_{L^2_{ds}(S^1)} D_s c\x D^2_s c \Big]\Bigg\}\,.
\end{multline}
\item \label{cor:hamVF_Gphi_caseb} If $\Ph(\ell)= C\ell^{-3}$, and if the Hamiltonian $H\colon \Imm(S^1,\R^3)\to \R^3$ invariant under $\Diffp$ and the flows of $I$ and $\on{hgrad}^{\OmMW}\ell= D_sc\x D_s^2 c$, then $\on{hgrad}^{\Om^\la}_cH$ is the orthonormal projection of 
\begin{align}
    X^H_c
    =-\frac{\ell^3}{3 C} D_s c\x \on{grad}^{G^\id}H
\end{align}
to the $G^\id_c$-orthogonal complement of the kernel of $\Om^\la$, which is spanned by  $I$ and $\on{hgrad}^{\OmMW}\ell$, and $\{a.D_s c \mid a\in C^\infty(S^1)\}$.
\end{enumerate}
\end{corollary}

\begin{proof}
    The stated formula follows from Proposition \ref{prop:hamVF_conformal} with $\on{grad}_c^{G^\id}\Phi(\rev{\ell(c)})=-\Phi'(\rev{\ell(c)})D_s^2 c$ and that $\on{hgrad}^{\OmMW}C\ell^{p}$ is a constant multiple of $\on{hgrad}^{\OmMW}\ell$.
\end{proof}

\begin{remark}
From the above Proposition it follows \(\on{hgrad}^{\Om^{\Phi(\ell)}}H\) agrees with \(\on{hgrad}^{\Om^{\on{id}}}H\) up to a constant scaling if $\Phi'(\rev{\ell(c)})=0$.
If  \(\Phi'(\rev{\ell(c)})\neq 0\) and \(\langle D_s^2 c, \on{grad}^{\Om^{\on{id}}}H\rangle_{L^2_{ds}(S^1)}\neq 0\) then it is, however, genuinely different, i.e., it does not seem realizable as a Hamiltonian vector field for the Marsden-Weinstein form \(\OmMW\). To formally prove that a given vector field $X_H$ is never attained by the Marsden-Weinstein structure one needs to show that $\mathcal{L}_{X_H}\OmMW \neq 0$. Using the closeness of $\OmMW$ and Cartan's formula, this can be reduced to show that $di_{X_H}\OmMW\neq 0$. However the necessary computations for this turn out to become extremely cumbersome and not very insightful. We refrain from providing them here. 
\end{remark}
Next we will consider several explicit examples, that will further highlight the statement of the above remark. We acknowledge that many of the Hamiltonian functions we consider were studied for the Marsden-Weinstein structure in~\cite{chern_knöppel_pedit_pinkall_2020}.

\begin{example}[Length function]
  We start with the arguably simplest Hamiltonian, namely  we assume that $H$ is a function of the total length $\ell$, i.e., \(H(c)=f\circ\ell(c)\) for some function \(f\). In this case we calculate:
\begin{align}
    dH_c (k)= d [f\circ\ell] _c(k) = D_{c,k}f(\rev{\ell(c)}) = f'(\rev{\ell(c)})\int \langle D_sk, D_sc\rangle ds=- f'(\rev{\ell(c)})\int \langle D^2_s c, k\rangle ds,
\end{align}
hence 
\begin{align}
    \on{grad}_c^{G^\id} H=-f'(\rev{\ell(c)})D^2_s c.
\end{align}
Using Corollary~\ref{cor:hamVF_Gphi}, we thus have
\begin{align}
    \on{hgrad}_c^{\Om^{\Phi(\ell)}} H = \frac{f'(\rev{\ell(c)})}{3\Phi(\rev{\ell(c)})}\left(1+ \frac{\Phi'(\rev{\ell(c)})\rev{\ell(c)}}{3\Phi(\rev{\ell(c)})+\Phi'(\rev{\ell(c)})\rev{\ell(c)}} \right)D_s c \x D^2_s c.
\end{align}
 If $f'(\rev{\ell(c)})=0$ for the initial length of the curve $\rev{\ell(c)}$, it is a zero vector field. If  $f'(\rev{\ell(c)})\neq 0$, then the length $\rev{\ell(c)}$ is conserved along the flow as $H=f\circ \ell$ is conserved. Note that  $\on{hgrad}^{\Om^{\Phi(\ell)}} H $ is a constant multiple of the binormal equation (also known as the vortex filament equation), 
\begin{align} \label{eq:binormal_equation}
    \on{hgrad}_c^{\OmMW}\ell=D_s c \x D^2_s c
\end{align}
using the Marsden-Weinstein symplectic structure.

 Thus we have seen that the Hamiltonian vector field of the the symplectic structure $\Om^{\Phi(\ell)}$ is a constant multiple of the Hamiltonian vector field of the Marsden-Weinstein symplectic structure. Note, that this constant factor, i.e., \emph{the relative speed} with respect to the standard binormal equation, depends on the initial length $\rev{\ell(c)}$.
\end{example}

\begin{example}[Flux of a divergence-free vector field on $\R^3$ though a Seifert surface]\label{eg:flux}
Our next examples of Hamiltonians are the fluxes of vector fields through Seifert surfaces. We consider
 for any divergence-free vector field \(V\in\Gamma(T \R^3)\) the closed 2-form \(\xi_V\coloneqq i_V( dx\wedge dy\wedge dz)\). 
 We can then define the corresponding \emph{flux} by 
\begin{align*}
E_V \coloneqq\int_{D^2} \langle V\circ \Sigma, n\rangle= \int_{\Sigma(D^2)} \xi_V
\end{align*}
where $\Sigma\colon D^2 \to \R^3$ is a smooth Seifert surface, i.e., an oriented and connected surface with $\Sigma\mid_{\partial D^2}=c$, and $n$ is the unit surface normal. 

We remark that $E_V$ is independent of the choice of $\Sigma$. 
To see this, first notice that there is a unique 1-form \(\alpha_V\) (up to addition of an exact 1-form) such that \(d\alpha_V = \xi_V\)  as  $H^1_{dR}(\R^3)=0$ and $H^2_{dR}(\R^3)=0$. By Stokes theorem we have,
\begin{align*}
    \int_{\Sigma(D^2)} \xi_V= \int_{D^2}  \Sigma^*d\alpha_V 
=\int_{\Sigma(\partial D^2)}\alpha_V
=\int_{c(S^1)} \alpha_V
\end{align*}
where $\Sigma^*$ denotes the pullback by $\Sigma$. 
For \(E_V\), we have the following formulas from~\cite[Theorem 4]{chern_knöppel_pedit_pinkall_2020}:
\begin{align*}
    &\on{grad}^{G^\id} E_V = D_s c\times(V\circ c),\\
    &\on{hgrad}^{\OmMW} E_V= V\circ c.
\end{align*}
We consider $E_V$ for two specific choices of \(V\), where we use an analogous notation as in~\cite{chern_knöppel_pedit_pinkall_2020}: the translation \(V_{-1}=v\) by some \(v\in\R^3\) and the rotation \(V_{-2}(x)=v\times x\) with some unit \(v\in\R^3\) and we denote the corresponding fluxes by  \(H_{-1}=E_{V_{-1}}\) and \(H_{-2}=E_{V_{-2}}\). 
Next we compute the horizontal Hamiltonian vector fields. From the computation
\begin{align}
    &\langle D^2_s c, D_s c\x ( D_s c \x v) \rangle_{L^2(ds)}=0,\\
    &\langle c, D_s c \x v \rangle_{L^2(ds)}
    =\int_{S^1} \langle D_s c, v\x c \rangle ds = \int_{S^1} \langle D_s c, 2 \on{curl}(v)\circ c\rangle ds =  2H_{-1}(c),
\end{align}
and 
\begin{align}
    &\langle D^2_s c, D_s c\x ( D_s c \x (v\x c)) \rangle_{L^2(ds)}=0,\\
    &\langle c, D_s c \x (v\x c) \rangle_{L^2(ds)}
    =\int_{S^1} \langle D_s c, (v\x c)\x c \rangle ds = \int_{S^1} \langle D_s c, 3 \on{curl}(v\x x)\circ c\rangle ds =  3H_{-2}(c),
\end{align}
     we obtain for $i\in \{-1, -2\}$,
 \begin{align}\label{eq:hgrad_flux}
    \on{hgrad}^{\Om^{\Phi(\ell)}} H_{i}
     &= \frac{w_i}{3\Phi(\rev{\ell(c)})}+\frac{C_i H_{i}(c)\Phi'(\rev{\ell(c)})}{3\Phi(\rev{\ell(c)})(3\Phi(\rev{\ell(c)})+\Ph'(\rev{\ell(c)})\rev{\ell(c)})}D_s c\times D_s^2 c\\
     &= \frac{1}{3\Phi(\rev{\ell(c)})} \on{hgrad}^{\OmMW} H_{i}  + \frac{C_i H_{i}(c)\Phi'(\rev{\ell(c)})}{3\Phi(\rev{\ell(c)})(3\Phi(\rev{\ell(c)})+\Ph'(\rev{\ell(c)})\rev{\ell(c)})}  \on{hgrad}^{\OmMW} \ell
 \end{align}
where $w_{-1}=v, w_{-2}=v\x c$  and $C_{-1}=2, C_{-2}=3$ respectively.

Since all of the three quantities $\ell$, $H_{-1}$, and $ H_{-2}$ are constants in motion along the fields $\on{hgrad}^{\Om^{\on{id}}}H_i$ and $\on{hgrad}^{\Om^{\on{id}}}\ell$ \cite[Corollary 1]{chern_knöppel_pedit_pinkall_2020}, the coefficients of  both terms in \eqref{eq:hgrad_flux}  do not change along $\on{hgrad}^{\Om^{\Phi(\ell)}} H_i$. Hence the Hamiltonian fields $\on{hgrad}^{\Om^{\Phi(\ell)}} H_i$ are weighted sums of the Marsden-Weinstein Hamiltonian fields of \(\ell\) and \(H_{-1}\) (or \(H_{-2}\) respectively). 
\end{example}

\begin{example}[Squared curvature]\label{ex:squaredcurv}
    We next compute the Hamiltonian vector field for the squared curvature 
    \begin{align}
        H(c)\coloneqq \frac{1}{2}\int \ka^2 ds.
    \end{align}
    We have according to \cite{chern_knöppel_pedit_pinkall_2020},
    \begin{align}
        \on{grad}^{G^\id} H=D_s \left(D^3_s c+\frac{3}{2}\ka^2 D_s c \right),
        \quad D_s c \x \on{grad}^{G^\id} H 
        = D_s c\x D_s^4 c+\frac{3}{2}\ka^2 D_s c \x D_s^2 c.
    \end{align}
    Then, from
    \begin{align}
         &\langle D_s^2 c, D_s c\x \on{grad}^{G^\id} H\rangle_{L^2_{ds}(S^1)} = \langle D_s^2 c, D_s c\x D_s^4 c \rangle_{L^2_{ds}(S^1)} +0=0,\\
         &\langle  c, \on{grad}^{G^\id} H\rangle_{L^2_{ds}(S^1)}=\int \ka^2-\frac{3}{2}\ka^2 ds =-H(c),
    \end{align}
    we have
\begin{align}
    \on{hgrad}^{\Om^{\Phi(\ell)}} H
    &=\frac{1}{3\Phi(\rev{\ell(c)})} \left\{ - D_s c\x D_s^4 c-\frac{3}{2}\ka^2 D_s c \x D_s^2 c
      -\frac{H\Phi'(\rev{\ell(c)})}{3\Phi(\rev{\ell(c)})+\Ph'(\rev{\ell(c)})\rev{\ell(c)}} 
     D_s c\x D^2_s c \right\}
     \\&=\frac{1}{3\Phi(\rev{\ell(c)})} \left\{ \on{hgrad}^{\OmMW} H
      -\frac{H\Phi'(\rev{\ell(c)})}{3\Phi(\rev{\ell(c)})+\Ph'(\rev{\ell(c)})\rev{\ell(c)}} 
     \on{hgrad}^{\OmMW} \ell\right\}.
\end{align}
    Since both $H$ and $\ell$ are again constants in motion along both $\on{hgrad}^{\OmMW} \ell$ and $\on{hgrad}^{\OmMW} H$ \cite{chern_knöppel_pedit_pinkall_2020}, $\on{hgrad}^{\Om^{\Phi(\ell)}} H$ is also realized as a Hamiltonian vector field of $\Om^{\on{id}}$.
\end{example}

\begin{example}[Total torsion]\label{ex:torsion}
    We next consider the total torsion
    \begin{align}
        H(c)\coloneqq \int \tau ds.
    \end{align}
    Using the results \cite[Theorem 2]{chern_knöppel_pedit_pinkall_2020}
    \begin{align}
    \on{grad}^{G^\id}H&= -D_s c\x D_s^3 c,\\
    D_s c\x \on{grad}^{G^\id}H&=-D_s c\x (D_s c\x D_s^3 c),
    \end{align}
    we compute
    \begin{align}
    \langle D_s^2 c, D_s c\x \on{grad}^{G^\id} H\rangle_{L^2_{ds}(S^1)}&= -\frac{1}{2}\int D_s\ka^2 ds=  0\\
      \langle c, \on{grad}^{G^\id} H\rangle_{L^2_{ds}(S^1)}
     &= \langle D_s c, D_s c \x D_s^2 c\rangle_{L^2_{ds}(S^1)}+\langle c, D_s^2 c \x D_s^2 c\rangle_{L^2_{ds}(S^1)}= 0+0.
\end{align}
Then we get
\begin{align}
    \on{hgrad}^{\Om^{\Phi(\ell)}} H
    &=\frac{1}{3\Phi(\rev{\ell(c)})} D_s c\x (D_s c\x D_s^3 c)
     =\frac{1}{3\Phi(\rev{\ell(c)})}\on{hgrad}^{\OmMW}H 
     ,
\end{align}
which is a scaled version of the Marsden-Weinstein gradient flow.
\end{example}

\begin{example}[Squared scale]\label{eg:squaredScale}
Next we consider the squared scale
\begin{align}\label{eq:squared_scale}
    E(c)\coloneqq \frac12 \int |c|^2 ds, 
\end{align}
as a Hamiltonian function. This is seen as the total kinetic energy of a moving particle in a periodic orbit in $\R^3$.

We first get by a direct computation that,
\begin{align}
    \on{grad}^{G^\id}E&=c - \langle c,D_s c\rangle D_s c -  \frac{1}{2}|c|^2 D_s^2 c=(1-\on{pr}_c)c  -  \frac{1}{2}|c|^2 D_s^2 c,\\
    D_s c\x \on{grad}^{G^\id}E&=D_s c \x c -  \frac{1}{2}|c|^2 D_s c \x D_s^2 c,
\end{align}
and
\begin{align}
    \langle D_s^2 c, D_s c\x \on{grad}^{G^\id} E\rangle_{L^2_{ds}(S^1)}&= -\Th^\id_c(D_s^2 c),\\
      \langle  c, \on{grad}^{G^\id} E\rangle_{L^2_{ds}(S^1)}
     &= \| D_s c \x c\|_{L^2_{ds}(S^1)}^2 -\frac{1}{2}\langle c, |c|^2 D_s^2 c\rangle _{L^2_{ds}(S^1)}
     \\&=\| D_s c \x c\|_{L^2_{ds}(S^1)}^2 +E(c).
\end{align}
Using them with Corollary ~\ref{cor:hamVF_Gphi} gives us;
   \begin{align}\label{eq:hgrad_squaredScale}
         \on{hgrad}^{\Om^{\Phi(\ell)}} E
        = \frac{1}{3\Phi(\rev{\ell(c)})} \Big\{ - D_s c \times c &+ \frac{1}{2}|c|^2 D_s c \x D_s^2 c\\ 
        +\frac{\Phi'(\rev{\ell(c)})}{3\Phi(\rev{\ell(c)})+\Ph'(\rev{\ell(c)})\rev{\ell(c)}} &\Big[ -\Th^\id_c(D_s^2 c) D_s c\times (D_s c \x c)\\ 
        &-  \left(\|D_s c \x c\|^2 _{L^2_{ds}(S^1)}-\frac{1}{2}\langle c,|c|^2 D_s^2 c\rangle _{L^2_{ds}(S^1)} \right)  D_s c\times D_s^2 c\Big]\Big\}\\
        = \frac{1}{3\Phi(\rev{\ell(c)})} \Big\{ - D_s c \times c &+ \frac{1}{2}|c|^2 D_s c \x D_s^2 c\\ 
        +\frac{\Phi'(\rev{\ell(c)})}{3\Phi(\rev{\ell(c)})+\Ph'(\rev{\ell(c)})\rev{\ell(c)}} &\Big[ \Th^\id_c(D_s^2 c) (1-\on{pr}_c)c\\ 
        & -\left(\|D_s c \x c\|^2 _{L^2_{ds}(S^1)}+E(c) \right)  D_s c\times D_s^2 c\Big]\Big\}.
    \end{align}
    
\end{example}

\begin{example}[Product of length and total squared curvature]
\label{eg:length_squared_curvature}
Our last example is the Hamiltonian given by
\begin{align}
    H(c)=\rev{\ell(c)} K(c)
\end{align}
where $K(c)=\int_{S^1} \ka^2 ds$ is the total squared curvature.  This somewhat unusual Hamiltonian is the only one among our examples that satisfies the condition required in Corollary \ref{cor:hamVF_Gphi} \ref{cor:hamVF_Gphi_caseb} the scale-invariant case. That is, $H$ is invariant under the both flows of $I=c$ and $Y\coloneqq \on{hgrad}^{\OmMW}\ell=D_s c\x D_s^2 c$.
To see this, let us compute
\begin{align}
    \L_{Y}H
    =K\L_Y \ell+\ell \L_Y K
    =K\cdot 0+\ell \cdot 0 =0
\end{align}
as $\ell$ is the Hamiltonian of $Y$ and the last equality follows from a direct computation using \eqref{eq:varka2}. 
This shows the existence of a Hamiltonian vector field horizontal in the sense of Corollary \ref{cor:hamVF_Gphi} \ref{cor:hamVF_Gphi_caseb}.
\end{example}
\begin{question}
    We know from the above examples that some vector fields are realized as Hamiltonian vector fields of \emph{both} $\bar\Om^{\on{MW}}$ and $\bar\Om^{\Phi(\ell)}$. We still do not know whether the spaces of all Hamiltonian vector fields generated by these two symplectic structures coincide, or if one is contained in the other. More generally, the coverage of Hamiltonian vector fields of $\bar\Om^{L}$ for a given operator $L$ is an independent question, which we have not investigated in this article. 
\end{question}

\section{Presymplectic structures induced by curvature weighted Riemannian metrics}\label{sec:curvature_weighted}
In this section we will consider the special case of symplectic structures, that are induced by curvature weighted metrics, i.e., we consider the Riemannian metric
\begin{equation}
G^{1+\ka^2}_c(h,k)=\int_{S^1}(1+\ka^2)\langle h,k\rangle ds,
\end{equation}
where $\ka=\ka_c$ denotes the curvature of the curve $c$. Note, that in the notation of the previous sections, this metric corresponds to the $G^L$ metric with $L=1+\ka^2$. This metric, which is sometimes also called the Michor-Mumford metric, has been originally introduced in~\cite{michor_mumford2006} to overcome the vanishing distance phenomenon of $L^2$-metric, see also \cite{michor2005vanishing}.

\begin{remark}[Relations to the Frenet-Serret formulas] Given $c\in \Imm(S^1,\mathbb R^3)$ we consider the open subset 
$U =\{ \ka>0\} = \{D_s^2c\ne 0\}\subset S^1$. Note that $\ka = 0$ on the boundary $\overline U\setminus U$, and is also 0 on the open complement $S^1\setminus \overline U$ which is a union of at most countably many open intervals in $S^1$; on each of these intervals $c$ is straight line segment since $D_sc$ is constant there. So we may assume that the torsion $\ta$ is defined and 0 on $S^1\setminus U$. 
On $U$
the moving frame and the Frenet-Serret fomulas are given by 
\begin{align}
T&= D_sc,\quad N=\ka^{-1}D_s^2c,\quad B = T\x N = \ka^{-1}D_sc\x D_s^2c
\\
D_s T &= \ka.N = D_s^2 c,  
\\
D_s N &=-D_s \ka. \ka^{-2}.D_s^2c + \ka^{-1}D_s^3c =  -\ka.T + \ta.B = -\ka.D_sc + \ta.\ka^{-1}D_sc\x D_s^2c 
\\
D_sB &= -D_s \ka. \ka^{-2}.D_sc\x D_s^2c  =  -\ta.N = -\ta.\ka^{-1}D_s^2c
\end{align}
This implies the following which are valid on the whole of $S^1$ since both sides vanish on $S^1\setminus U$: 
\begin{align}
D_s^3c &= \langle D_s^3c,T\rangle T + \langle D_s^3c,N\rangle N  
+\langle D_s^3c, B\rangle B\quad\text{ valid on }U
\\&
=\langle D_s^3c,D_sc\rangle D_sc + \ka^{-2}\langle D_s^3c,D_s^2c\rangle D_s^2c  
+\ka^{-2}\langle D_s^3c, D_sc\x D_s^2c\rangle D_sc\x D_s^2c\quad\text{ on }S^1
\\&
= -\ka^2 D_sc + D_s \ka. \ka^{-1}.D_s^2c + \ta.D_sc\x D_s^2c \quad\text{ valid on  }U\text{ but extends smoothly to }S^1
\\
\implies &  \langle D_s^3c,D_sc\rangle = -\ka^2, \quad \langle D_s^3c,D_s^2c\rangle = D_s\ka.\ka,
\quad \langle D_s^3c, D_sc\x D_s^2c\rangle = \ta.\ka^2 \quad\text{ valid on }S^1 \label{eq:frenet}
\\
\ta & = \ka^{-2} \langle D_s^3c, D_sc\x D_s^2c\rangle \qquad\text{ valid on }S^1\,.
\end{align}

\end{remark}

\begin{remark}
Similarly to Remark~\ref{rem:momentum_general} we obtain again conserved quantities and corresponding momentum mappings. Here we want to specifically highlight the momentum map   $J^{SO(3)}$: as an element of $\R^3\approx \mathfrak{so}^*(3)$, the angular momentum $J^{SO(3)}$ is given by
\begin{equation}
 \langle J^{SO(3)}(c), Y\rangle = \int (1+\ka^2)\langle c\x D_sc, Y\o c\rangle ds, 
\end{equation}
which can be understood as the angular momentum of a thickened curve where the thickness (or mass) at each point is a function of $1+\ka^2$. Note, that this is in stark contrast to the previous section, i.e., the length weighted case, where the angular momentum for $\Om^{\Phi(\ell)}$ is just the $\Phi(\ell)$-scaled version of the angular momentum for $\Om^{\id}=3\OmMW$. 
\end{remark}

We have the following result concerning the induced presymplectic structure: 
\begin{theorem}[The presymplectic structure $\Om^{1+\ka^2}$]\label{kappa-non-deg}
The induced (pre)symplectic structure of the $G^{1+\ka^2}$-metric is given by:
\begin{equation}\label{eq:Om_ka}
\begin{aligned} 
&\Om^{1+\ka^2}_c(h,k)
      =\int 3(1+\ka^2) \langle D_s c, h \x k\rangle +(D_s \ka^2) \langle c,h \x k\rangle+
4 \ka^2\langle D_sh,D_sc\rangle\langle c\x D_s c, k \rangle\\&\qquad-2\langle D_s^2h, D_s^2c\rangle\langle c\x D_s c, k \rangle
    -4 \ka^2\langle D_sk,D_sc\rangle\langle c\x D_s c,h\rangle + 2\langle D_s^2k, D_s^2c\rangle\langle c\x D_s c,h\rangle ds,
\end{aligned}
\end{equation}
and  the vertical vectors $\{a.D_s c \mid a\in C^\infty(S^1)\}\subset T_c \Imm$ is in the kernel.


\end{theorem}

\begin{proof}[Proof of Theorem~\ref{kappa-non-deg}.]
To calculate the formula for $\Omega^{1+\ka^2}$ we first need the variation of $\ka^2=\langle D_s^2 c, D_s^2 c\rangle$. Using, that $D_{c,h}D_s = -\langle D_s h, D_s c\rangle D_s$, cf. the proof of Lemma~\ref{lem:formula_sympl},
we calculate:
\begin{align}
    D_{c,h} (D_s^2 c) &= (D_{c,h}D_s).D_sc + D_s\left((D_{c,h}D_s)c\right) + D_s^2 h
    \\&
    =-\langle D_sh,D_sc\rangle.D_s^2c - D_s\left(\langle D_sh, D_sc\rangle D_sc\right) + D_s^2h\\
    &=-\langle D_sh,D_sc\rangle.D_s^2c - \left(D_s\langle D_sh, D_sc\rangle\right) D_sc -\langle D_sh, D_sc\rangle D^2_sc + D_s^2h\\
     &=-2\langle D_sh,D_sc\rangle.D_s^2c - \left(D_s\langle D_sh, D_sc\rangle\right) D_sc + D_s^2h
\end{align}
Thus we obtain
\begin{equation}\label{eq:varka2}
    D_{c,h}\ka^2 = -4\langle D_sh,D_sc\rangle \ka^2 - 0 + 2\langle D_s^2h, D_s^2c\rangle\,.
   \end{equation}
Next we note that 
\begin{align*}
    \Om^{1+\ka^2}_c(h,k) = \Om^{\id}_c(h,k)+\Om^{\kappa^2}_c(h,k)
\end{align*}
as the operation $L_c\mapsto \Th_c ^L$ is linear in $L_c$.
Using \eqref{eq:Omega_L}, we then calculate
\begin{equation}
\begin{aligned}\label{eq:presymplectic_curvature}
\Om^{\kappa^2}_c(h,k) 
    &=\int\langle D_s c, \ka^2 h \x k + h\x  \ka^2 k\rangle 
     - \langle c, D_s h \x  \ka^2 k - D_s k \x  \ka^2 h \rangle 
    \\&\qquad\qquad
     - \langle c\x D_s c, (D_{c,h} \ka^2)k - (D_{c,k} \ka^2)h\rangle ds,
     \\&=\int 2 \kappa^2 \langle D_s c, h \x k\rangle 
     - \ka^2\langle c, D_s h \x   k  \rangle - \langle \ka^2 c, h \x  D_s k\rangle 
     \\& \qquad- D_{c,h} \ka^2 \langle c\x D_s c, k \rangle +D_{c,k} \ka^2\langle c\x D_s c,h\rangle ds
     \\& =\int 2\kappa^2 \langle D_s c, h \x k\rangle 
     - \ka^2\langle c, D_s h \x   k  \rangle 
     \\&\qquad+ \langle D_s (\ka^2 c), h \x   k\rangle +\ka^2\langle c, D_s h \x   k  \rangle 
      \\& \qquad- D_{c,h} \ka^2 \langle c\x D_s c, k \rangle  + D_{c,k} \ka^2  \langle c\x D_s c,h\rangle ds
      \\& =\int 3\kappa^2 \langle D_s c, h \x k\rangle + (D_s \ka^2) \langle c,h \x k\rangle
      \\& \qquad- D_{c,h} \ka^2 \langle c\x D_s c, k \rangle  + D_{c,k} \ka^2 \langle c\x D_s c,h\rangle ds.
\end{aligned}
\end{equation}
Hence 
\begin{align}
     \Om^{1+\ka^2}(h,k) 
     & =\int 3(1+\kappa^2) \langle D_s c, h \x k\rangle +(D_s \ka^2) \langle c,h \x k\rangle
      \\& \qquad - D_{c,h} \ka^2 \langle c\x D_s c, k \rangle  + D_{c,k} \ka^2 \langle c\x D_s c,h\rangle ds.
\end{align}
and~\eqref{eq:Om_ka} follows by using the variation formula \eqref{eq:varka2} for $\ka^2$.

That $\Omega$ decends to a form on $B_i(S^1,\mathbb R^3)$ follows again from \autoref{th:descending_presymplectic_form}; alternatively we can also see this directly from the above formula: a straightforward calculation shows that $h=a.D_sc$ is indeed in the kernel of $\Omega_c^{1+\kappa^2}$.
\end{proof}

\begin{question}
It remains open if the presymplectic structure $\bar\Om^{1+\ka^2}$ on $B_i(S^1,\mathbb R^3)$ is non-de\-ge\-nerate and thus symplectic.   Therefore it remains to show that tangent vectors of the form $aD_sc$ are the whole kernel of $\Omega_c^{1+\kappa^2}$. It seems natural to employ a similar strategy as in the previous section for length weighted metrics, i.e., for given $h$ we test with all $k$ of the form $k=ac$ for $a\in C^{\infty}(S^1)$. This leads to reducing the degeneracy of $\Omega^{1+\kappa^2}$ to solving the equation $P_c(a)=f$ for any given $f\in C^{\infty}(S^1)$, where
\begin{align}
P_c(a) &:= 2 \langle  D_s^2c,c\rangle D_s^2a - 4 \langle D_s c, c \rangle \kappa^2. D_s a  + (3+\kappa^2)a;.
\end{align}
The existence of periodic solutions for the above equation is, however, non-trivial. Note, that the coefficient functions are in general degenerate, e.g.,  $\langle  D_s^2c,c\rangle$ can vanish somewhere. 
\end{question}

\begin{question}
    We may consider a more general version. Suppose $L_c\colon h\mapsto f_c. h$ where $f_c$ is a positive function for any $c$ and is of form $f_c(\theta)=\rho(c(\th),D_s c(\th), D_s^2 c(\th),\ldots, D_s^N c(\th))$ with some finite $N$ and a function $\rho\colon \R^{3N}\to \R_{\geq 0}$. We expect that $\bar \Om^L$ is symplectic on $B_i(S^1,\R^3)$ if $\Th^L$ is not scale-invariant, or on $B_i(S^1,\R^3)/\mathcal{F}$ with a 2-dimensional distribution $\mathcal{F}$ if $\Th^L$ is scale-invariant (cf. Theorem \ref{thm:sympl_conf}).
\end{question}

\section{Numerical illustrations}\label{sec:numerics}
In this section we numerically illustrate two Hamiltonian flows with respect to the new symplectic structures introduced in this article. For interested readers, we share video footage of the  simulations shown in Figure \ref{fig:flow_flux} and \ref{fig:flow_squaredScale}; see \href{https://youtu.be/nu09IwRK-tY}{https://youtu.be/nu09IwRK-tY}.

\begin{figure}[htbp]
    \centering
    \vspace{5pt}
Flow of $\on{hgrad}^{\Om^{10\rev{\ell}^{-2}}}H_{-2}$\\ \vspace{5pt}

    \begin{minipage}[t]{0.22\textwidth}
    \begin{center}
    \includegraphics[width=1.0\textwidth,trim={6cm 0cm 14cm 0cm},clip]{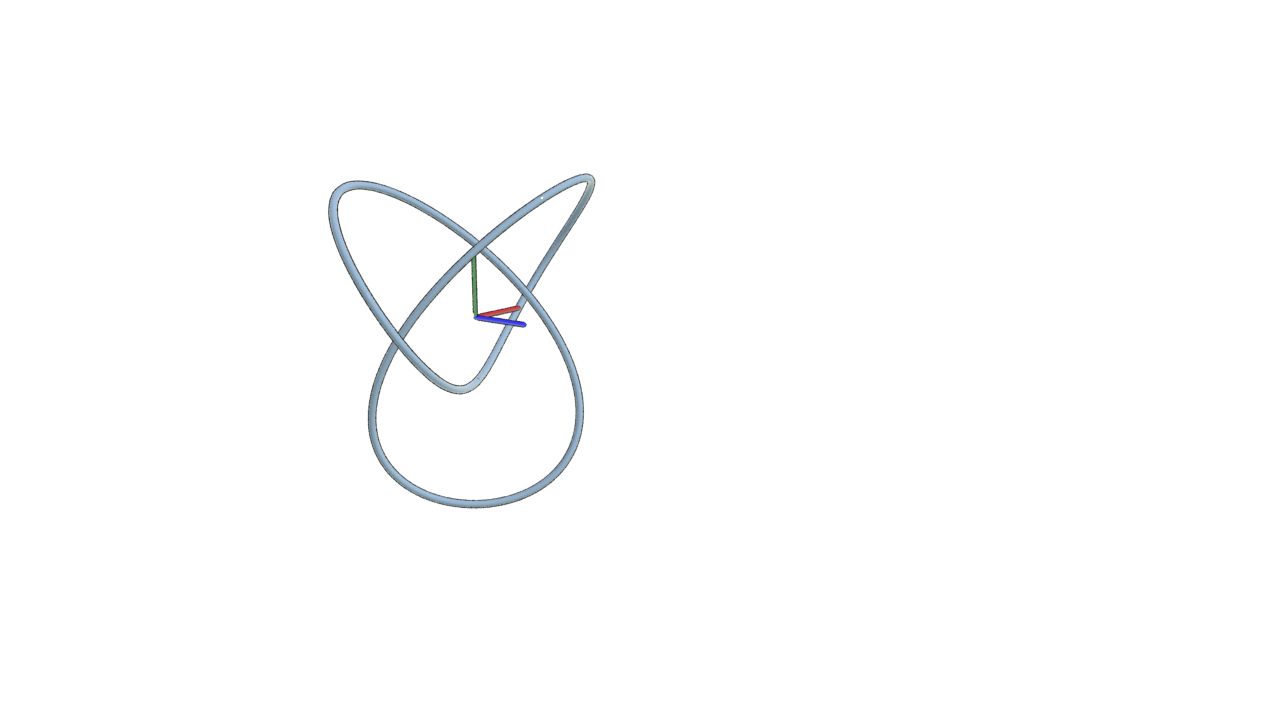}
    \end{center}
    \end{minipage}
    \begin{minipage}[t]{0.22\textwidth}
    \begin{center}
    \includegraphics[width=1.0\textwidth,trim={6cm 0cm 14cm 0cm},clip]{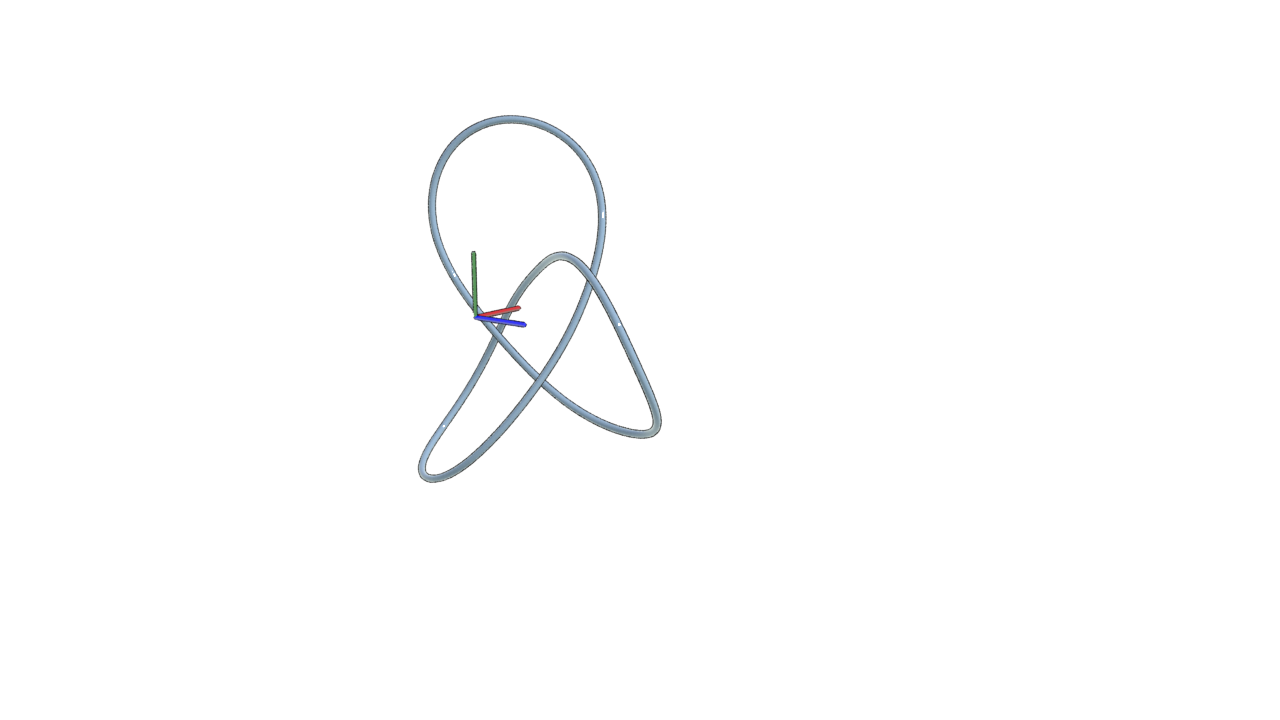}
    \end{center}
    \end{minipage}
    \begin{minipage}[t]{0.22\textwidth}
    \begin{center}
    \includegraphics[width=1.0\textwidth,trim={6cm 0cm 14cm 0cm},clip]{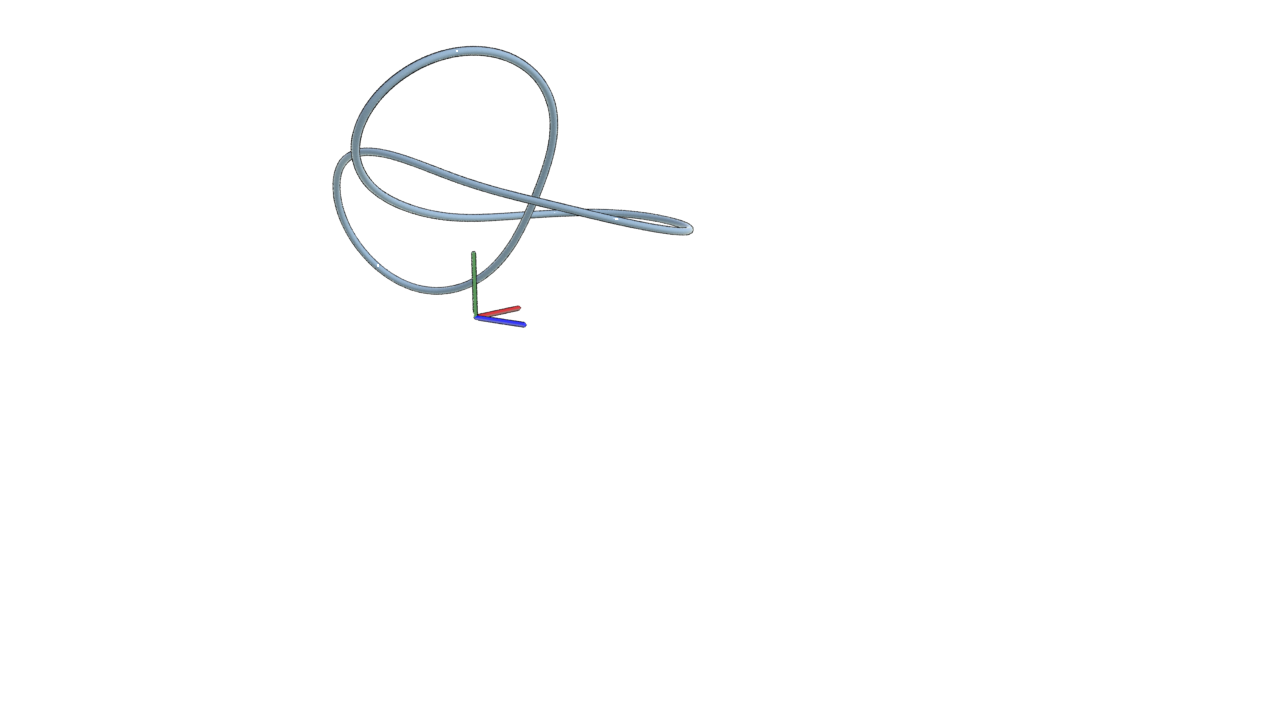}
    \end{center}
    \end{minipage}
    \begin{minipage}[t]{0.22\textwidth}
    \begin{center}
    \includegraphics[width=1.0\textwidth,trim={6cm 0cm 14cm 0cm},clip]{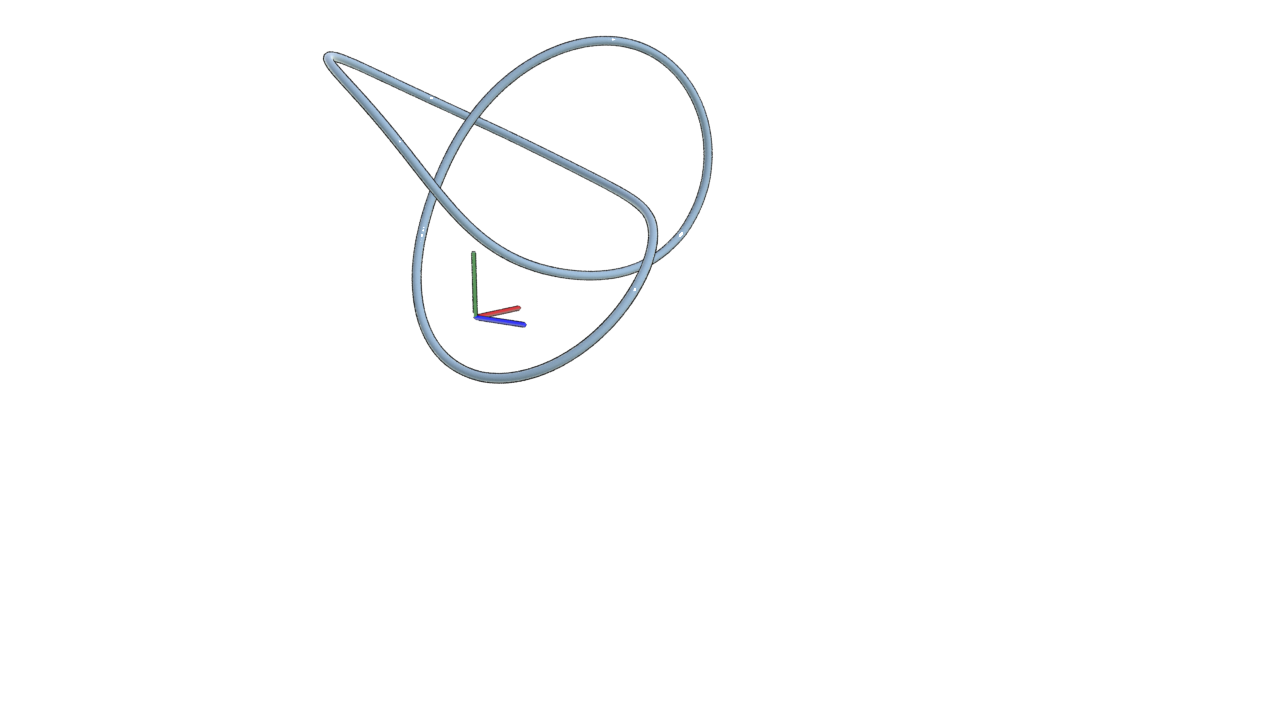}
    \end{center}
    \end{minipage}
    
\vspace{-5pt}
Only the binormal part of $\on{hgrad}^{\Om^{10\rev{\ell}^{-2}}}H_{-2}$\\ \vspace{5pt}
    \begin{minipage}[t]{0.22\textwidth}
    \begin{center}
    \includegraphics[width=1.0\textwidth,trim={6cm 0cm 14cm 0cm},clip]{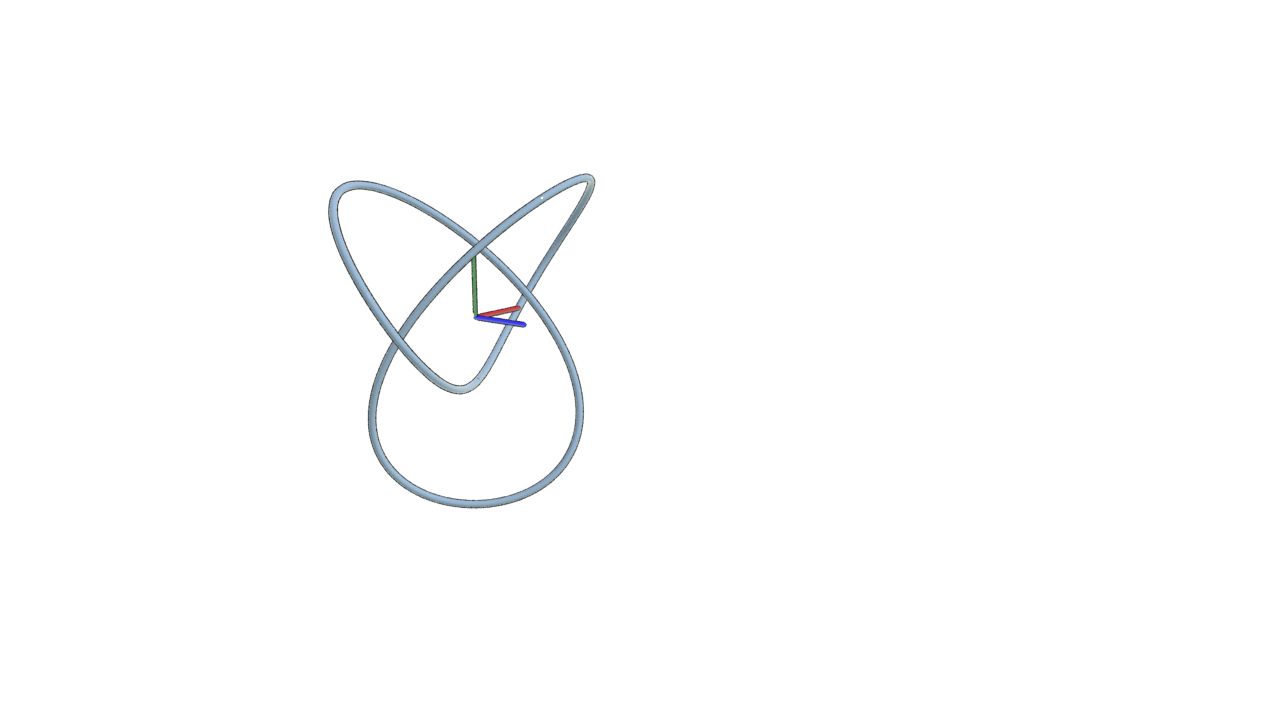}
    \end{center}
    \simtime{0}
    \end{minipage}
    \begin{minipage}[t]{0.22\textwidth}
    \begin{center}
    \includegraphics[width=1.0\textwidth,trim={6cm 0cm 14cm 0cm},clip]{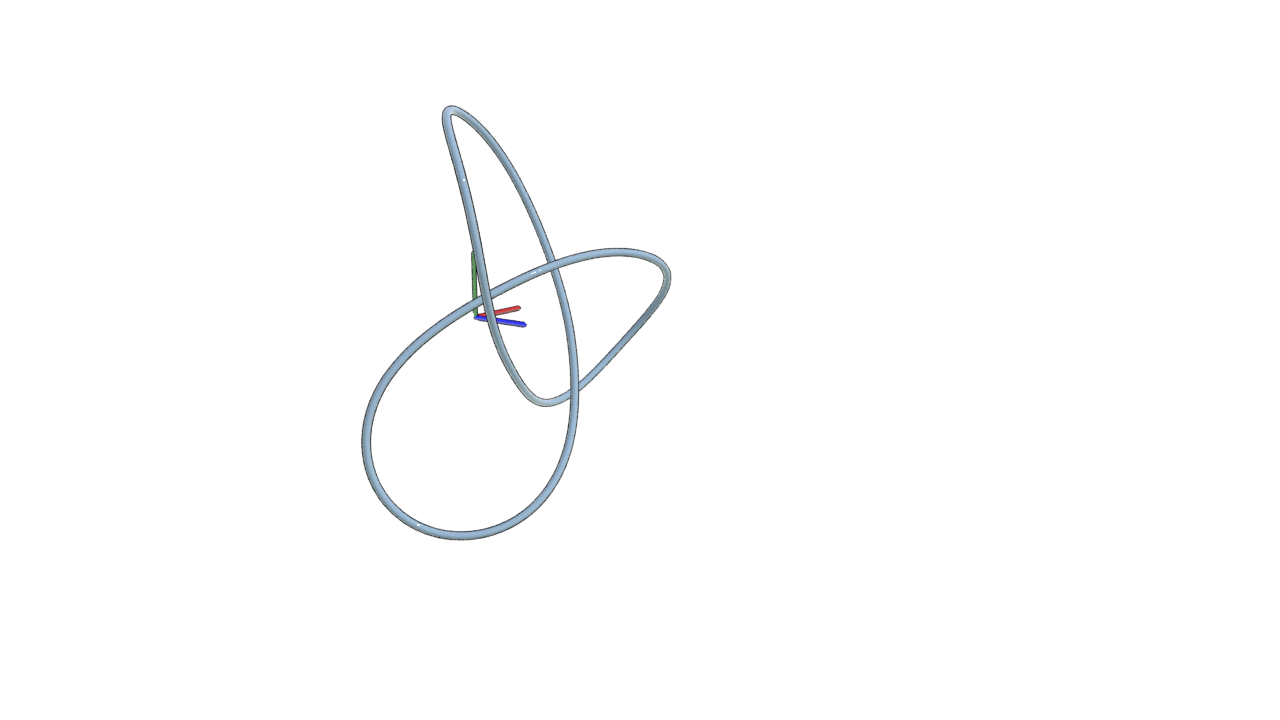}
    \end{center}
    \simtime{15}
    \end{minipage}
    \begin{minipage}[t]{0.22\textwidth}
    \begin{center}
    \includegraphics[width=1.0\textwidth,trim={6cm 0cm 14cm 0cm},clip]{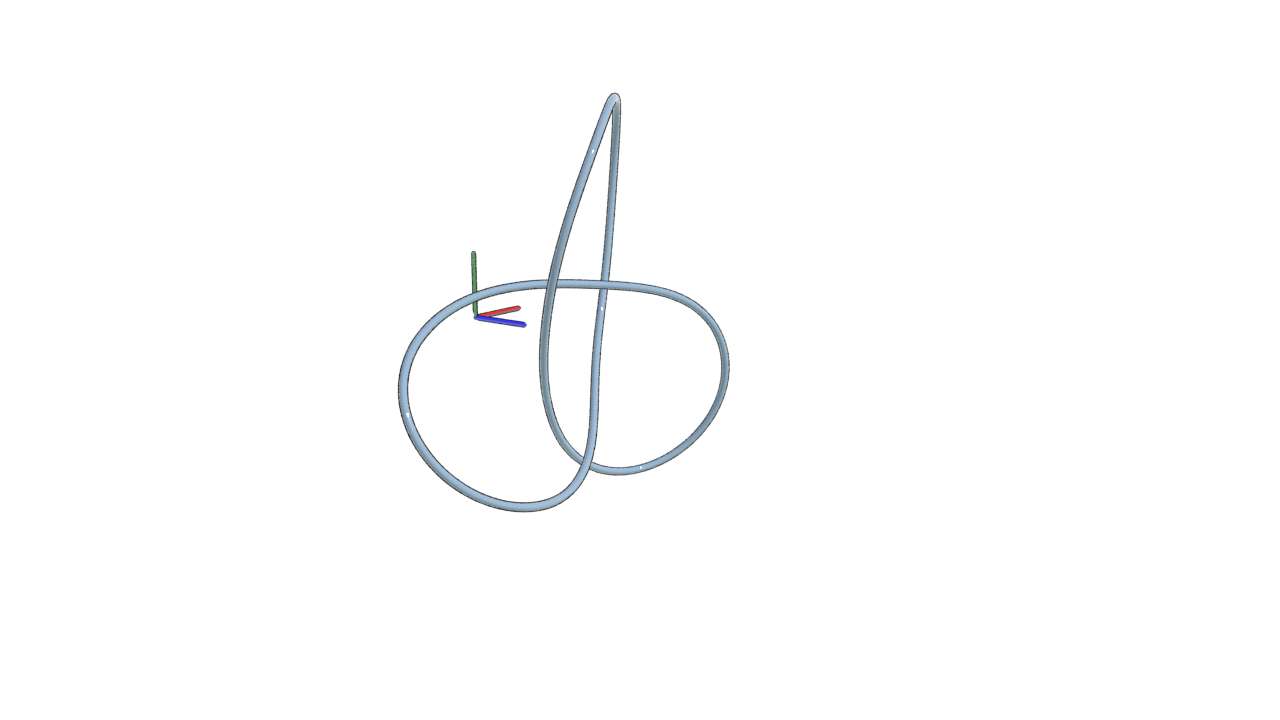}
    \end{center}
    \simtime{30}
    \end{minipage}
    \begin{minipage}[t]{0.22\textwidth}
    \begin{center}
   \includegraphics[width=1.0\textwidth,trim={6cm 0cm 14cm 0cm},clip]{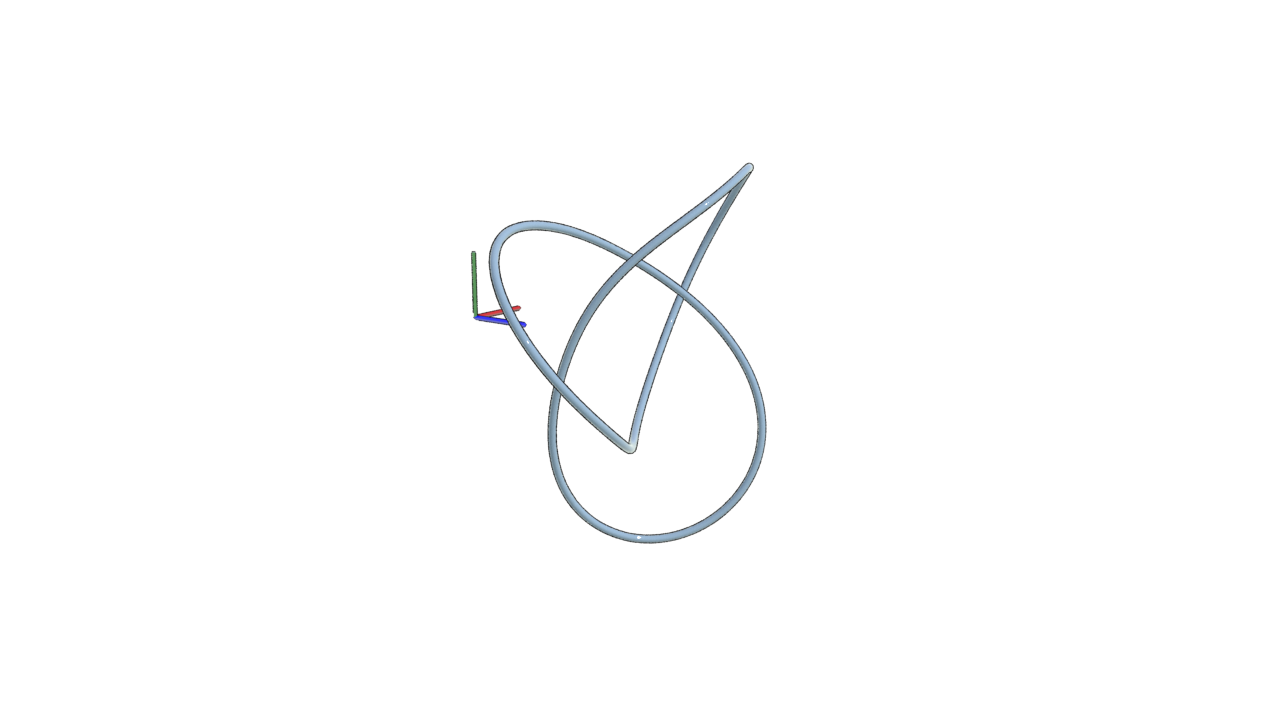}
    \end{center}
    \simtime{45}
    \end{minipage}

    \caption{Hamiltonian flow of $H_{-2}$, the flux of a rotational vector field from Example~\ref{eg:flux} using $\Phi(\rev{\ell(c)})=10\ell(c)^{-2}$ (top), and the flow only with its binormal component (bottom). The red, green, and blue axes are the $x,y,z$ axes respectively. }
    \label{fig:flow_flux}
\end{figure}
For the numerical simulations, we discretized each curve as an ordered sequence of points in $\R^3$. To approximate terms involving spatial derivatives,  such as the binormal vector and the curvature, we follow the methods of discrete differential geometry, see~\cite{bobenko2015DDG}. We then compute the time integration of each Hamiltonian vector field using the explicit Runge-Kutta method of fourth-order in time. We want to emphasize that our numerical examples are only for illustrative purposes and we do not guarantee any correctness of (even short-time) behaviors of the curve dynamics. 

In our experiments, we use length-weighted presymplectic structures $\Om^{\Phi(\ell)}$ (and symplectic structures $\bar\Om^{\Phi(\ell)}$ for unparametrized curves) as derived in Section \ref{sec:symp_length}. That is, we use functions of the form $\Phi(\ell)=C \ell^p$ with some $C>0$ and $p\in \R$. Note that $C$ only works as time-scaling and does not change the orbit under the Hamiltonian flow. This is because in the expression of the field $\on{hgrad}^{\Om^{\Phi(\ell)}} H$, cf. equation~\eqref{eq:ham_flow_length_weight}, the coefficient $C$ appears only in the factor $\frac{1}{3\Phi(\ell)}=\frac{1}{C \ell^p}$ shared by all the terms and the factor $\frac{\Phi'(\ell)}{3\Phi(\ell)+\Phi'(\ell)\ell}=\frac{p}{(3+p)\ell}$ does not depend on $C$. We choose $C$ to run each simulation with a reasonable discrete timestep, but it essentially does not affect the dynamics. 

We simulate two Hamiltonian flows (Example \ref{eg:flux} and \ref{eg:squaredScale}) from Section \ref{sec:symp_length}. These two examples involve only up to second-order spatial derivatives. Simulating other Hamiltonian flows, such as those discussed in Examples~\ref{ex:squaredcurv} and~\ref{ex:torsion} having third or higher-order derivatives is more challenging as one would have to discretize these higher-order derivatives more carefully.

As for the initial curve, we consider the trefoil
\begin{align}\label{eq:trefoil}
    c(\theta)=  
    \left((2 + \cos(2\th))\cos(3\th),(2 + \cos(2\th)) \sin(3\th), \sin(4\th)\right), \quad
\theta \in S^1=\R/2\pi\Z.
\end{align}
in both of our examples.


\begin{example}[Flux of a vector field]
We first simulate the Hamiltonian flow for the Hamiltonian that is defined as the flux of a vector field through a Seifert surface whose boundary is the curve $c$, cf. Example~\ref{eg:flux}. We chose the vector field of a rigid body rotation $V(x)=v\x x$ with the rotation axis $v=\frac{1}{\sqrt{3}}(1,1,1) \in \R^3$. This amounts to the Hamiltonian $H_{-2}$ in Example \ref{eg:flux}. 

The horizontal Hamiltonian field \eqref{eq:hgrad_flux} is a weighted sum of the rotation $\on{hgrad}^{\OmMW}H_{-2}=v\x c$ and the binormal field $\on{hgrad}^{\OmMW}\ell=D_s^2 c\x D_s c$ with time-constant coefficients. Since these two flows are Poisson commutative, we can simulate the flow by evolving the curves under the binormal equation and rotating it at each time, i.e., $c_t=\exp(t_1 \hat v) c_{t_2}^{\on{Binormal}}$ where $\hat v\in \mathfrak{so}(3)$ corresponds $v$ and $t_1, t_2$ are time $t$ weighted by the coefficients in \eqref{eq:hgrad_flux}. Figure~\ref{fig:flow_flux} illustrates our simulation using $\Phi(\ell\rev{(c)})=10\ell\rev{(c)}^{-2}$. 
The top row is the flow of  $\on{hgrad}^{\Om^{\Phi(\ell)}}H_{-2}$ and the bottom row is the flow by only the binormal equation part where the curve moves toward the $z$-direction while showing a rotational motion around the $z$-axis.
\end{example}
\begin{figure}[htbp]
    \centering
\vspace{5pt}
Flow of $\on{hgrad}^{\Phi(\ell)}E$ with $\Phi(\ell)=\frac{1}{20}$\\ \vspace{5pt}

    \begin{minipage}[t]{0.18\textwidth}
    \begin{center}
    \includegraphics[width=1.0\textwidth,trim={10cm 0 10cm 0},clip]{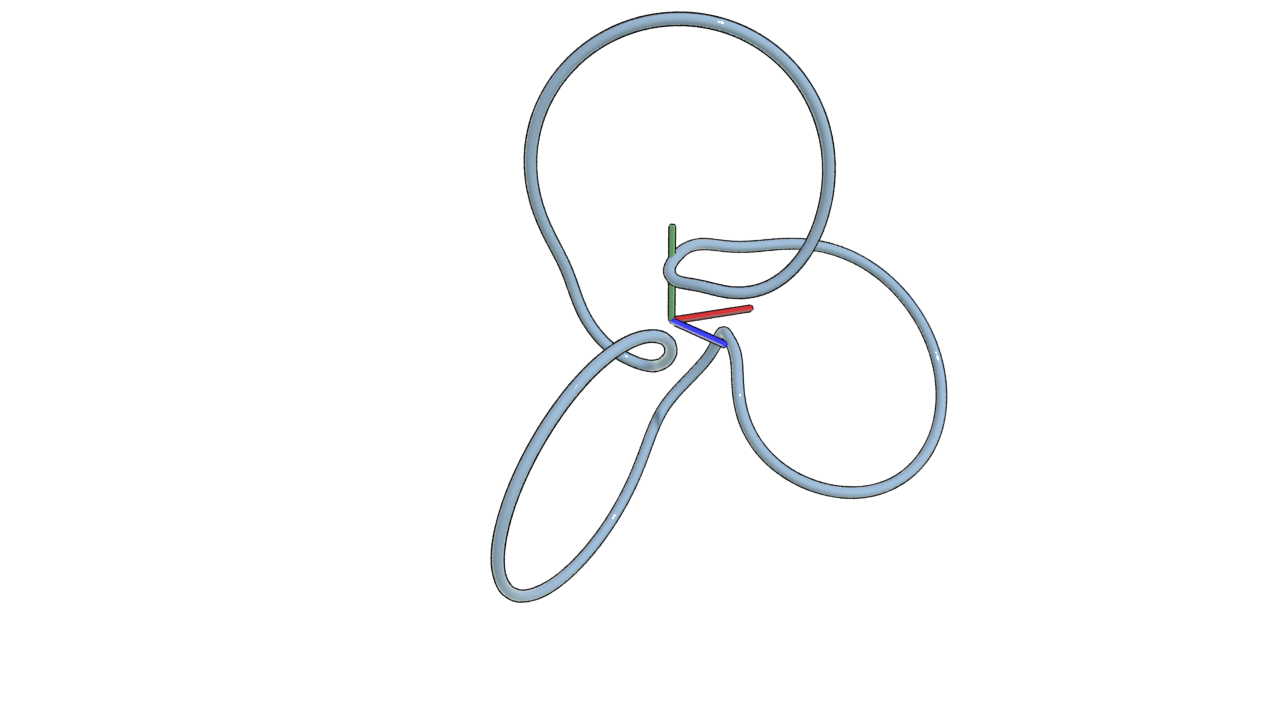}
    \end{center}
    \simtime{6}
    \end{minipage}
    \begin{minipage}[t]{0.18\textwidth}
    \begin{center}
    \includegraphics[width=1.0\textwidth,trim={10cm 0 10cm 0},clip]{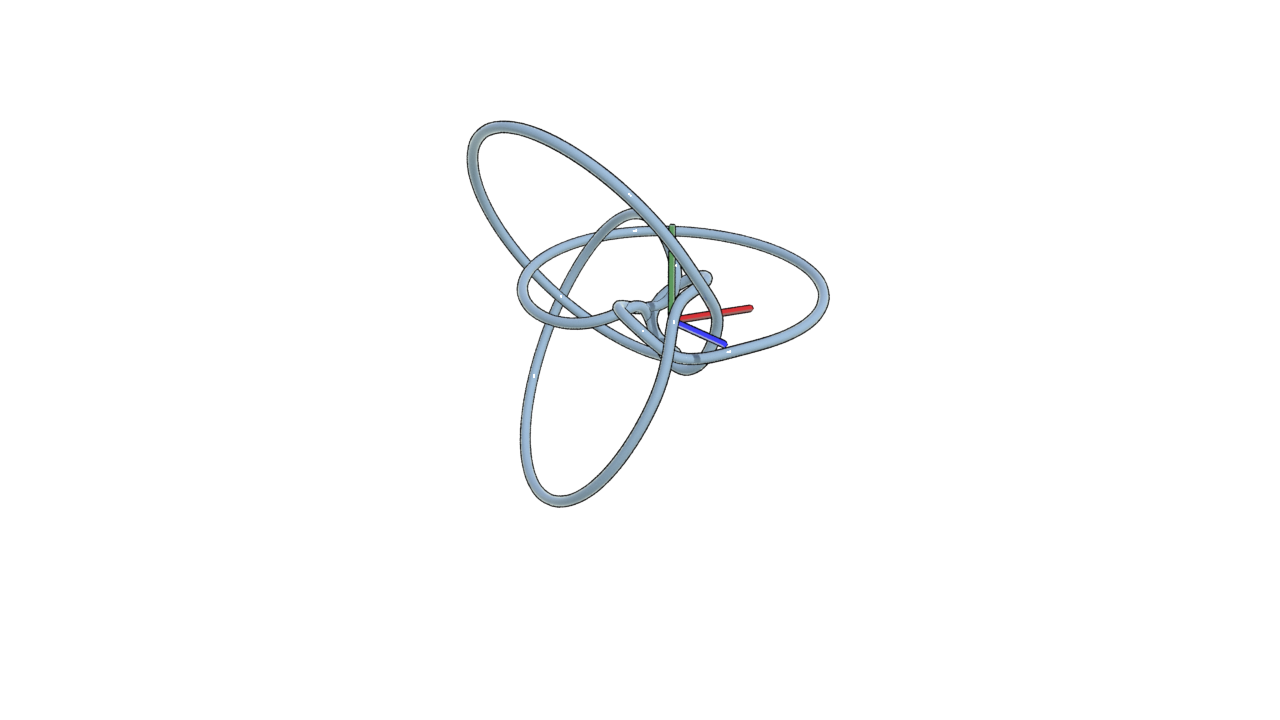}
    \end{center}
    \simtime{10}
    \end{minipage}
    \begin{minipage}[t]{0.18\textwidth}
    \begin{center}
    \includegraphics[width=1.0\textwidth,trim={10cm 0 10cm 0},clip]{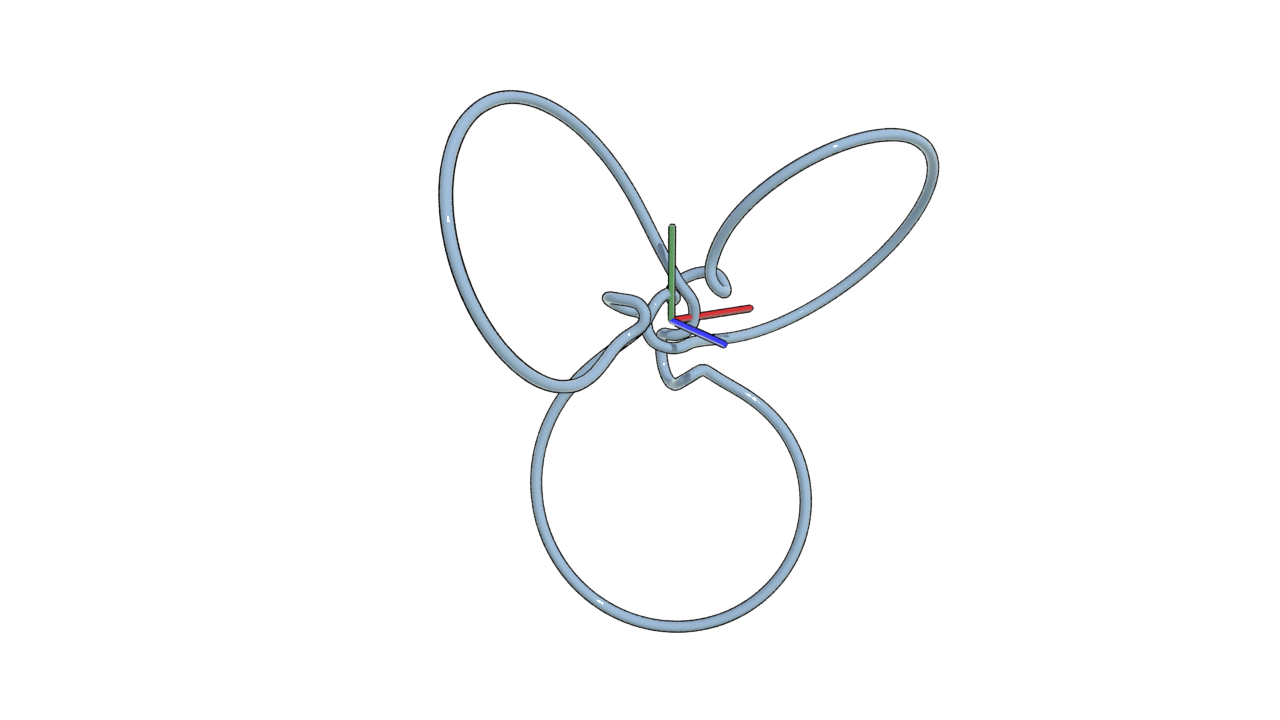}
    \end{center}
    \simtime{14}
    \end{minipage}
    \begin{minipage}[t]{0.18\textwidth}
    \begin{center}
    \includegraphics[width=1.0\textwidth,trim={10cm 0 10cm 0},clip]{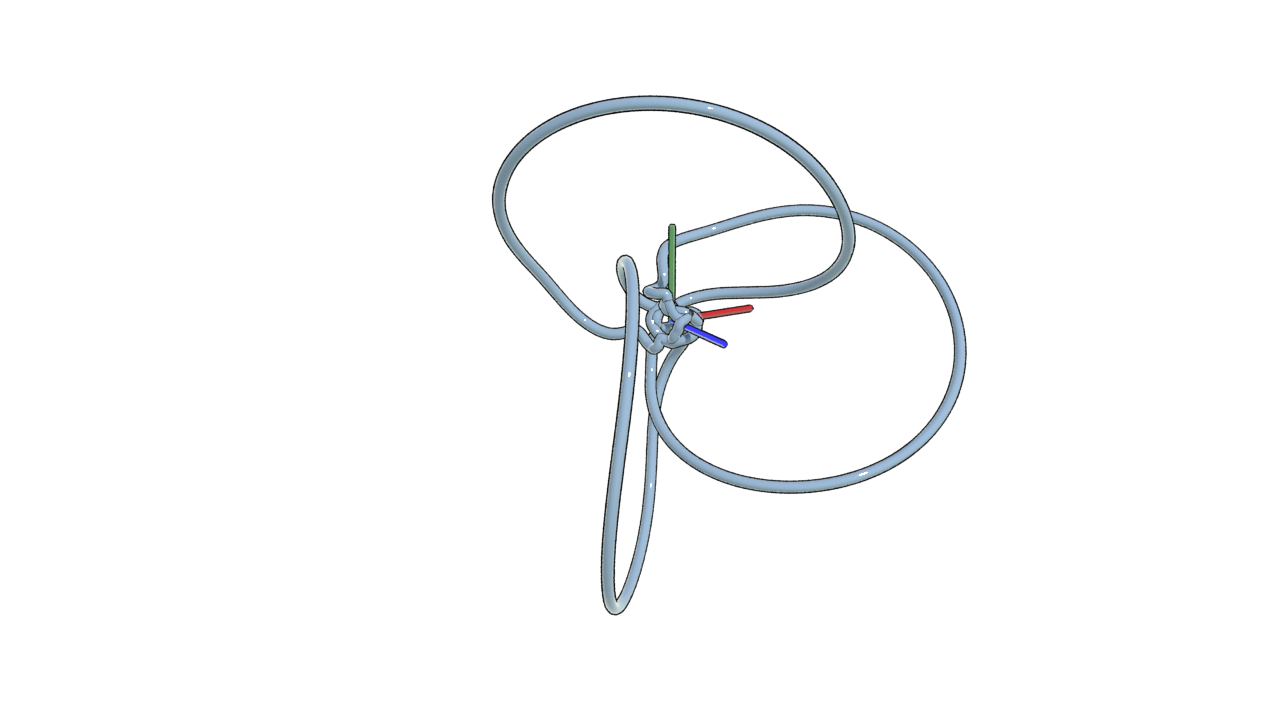}
    \end{center}
    \simtime{19}
    \end{minipage}
    \vrule
    \begin{minipage}[t]{0.17\textwidth}
        \includegraphics[width=1.0\textwidth,trim={10.cm -2cm 10.cm 0cm},clip]{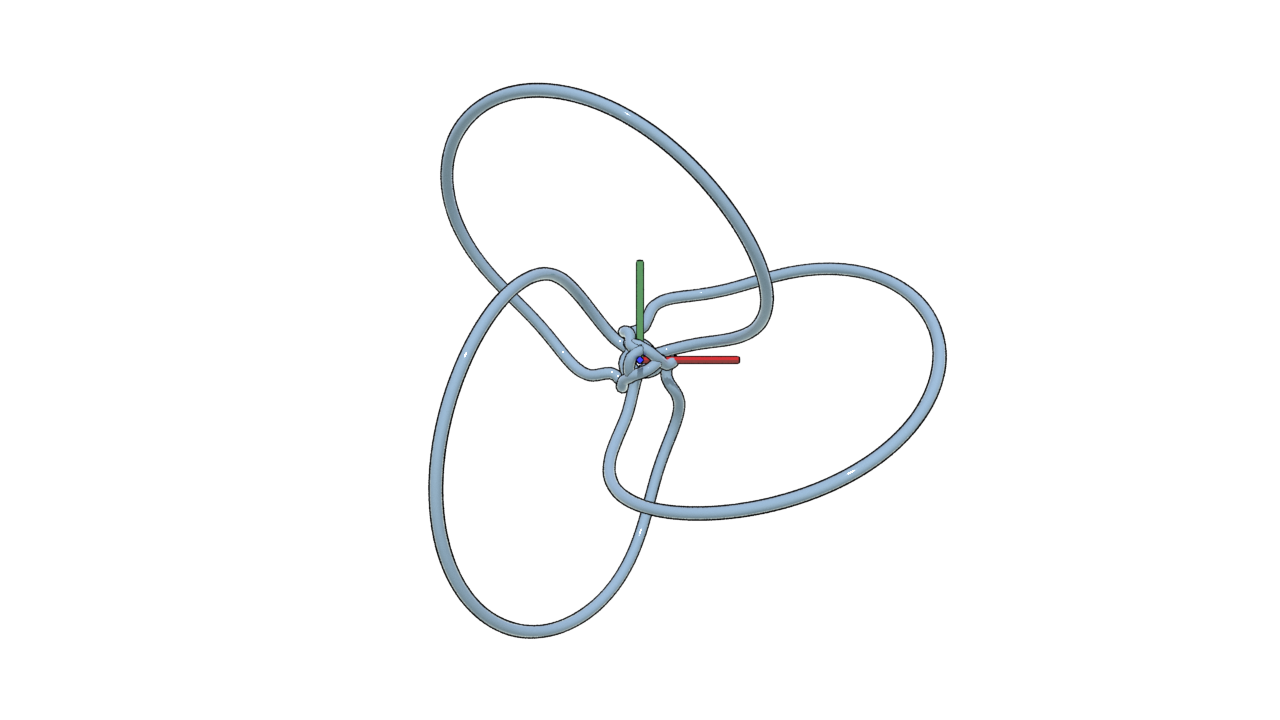}
    \end{minipage}
    \hrule

\vspace{15pt}

Flow of $\on{hgrad}^{\Phi(\ell)}E$ with $\Phi(\ell)=\frac{1}{20}\ell^{-\frac{1}{10}}$\\ \vspace{5pt}
    \begin{minipage}[t]{0.18\textwidth}
        \begin{center}
            \includegraphics[width=1.0\textwidth,trim={10cm 0 10cm 0},clip]{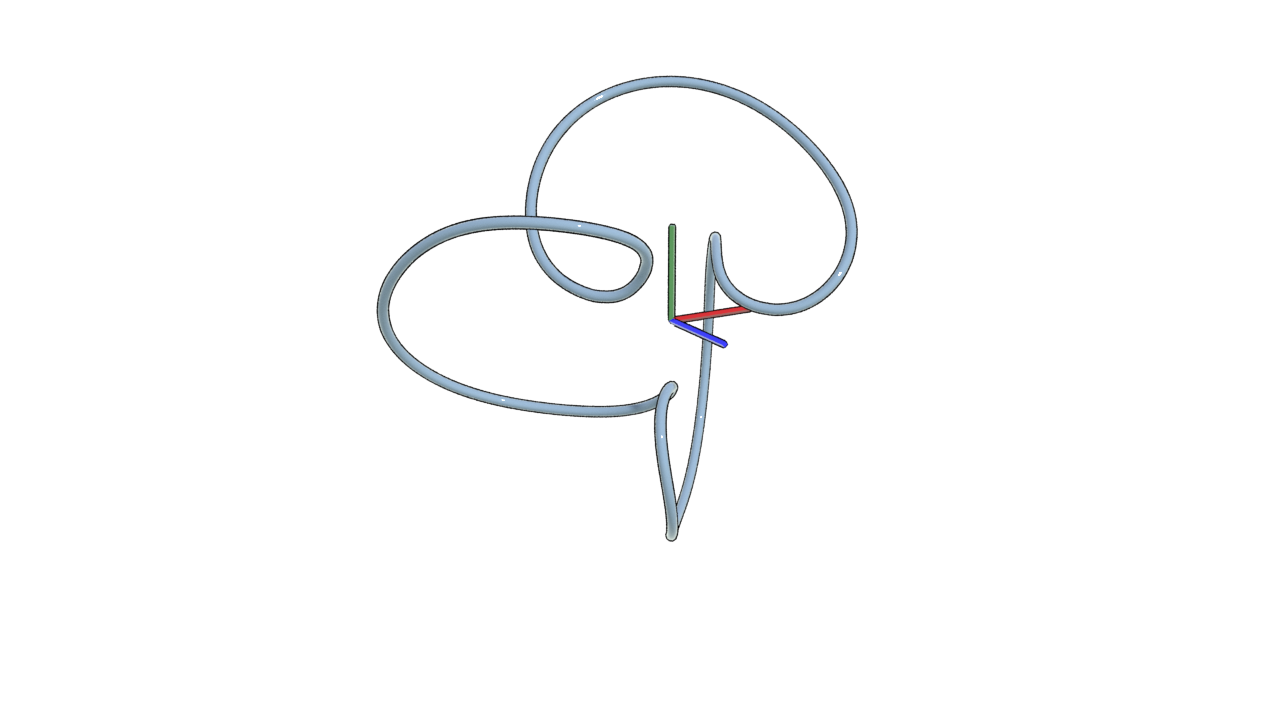}
        \end{center}
         \simtime{6}
        
    \end{minipage}
    \begin{minipage}[t]{0.18\textwidth}
        \begin{center}
        \includegraphics[width=1.0\textwidth,trim={10cm 0 10cm 0},clip]{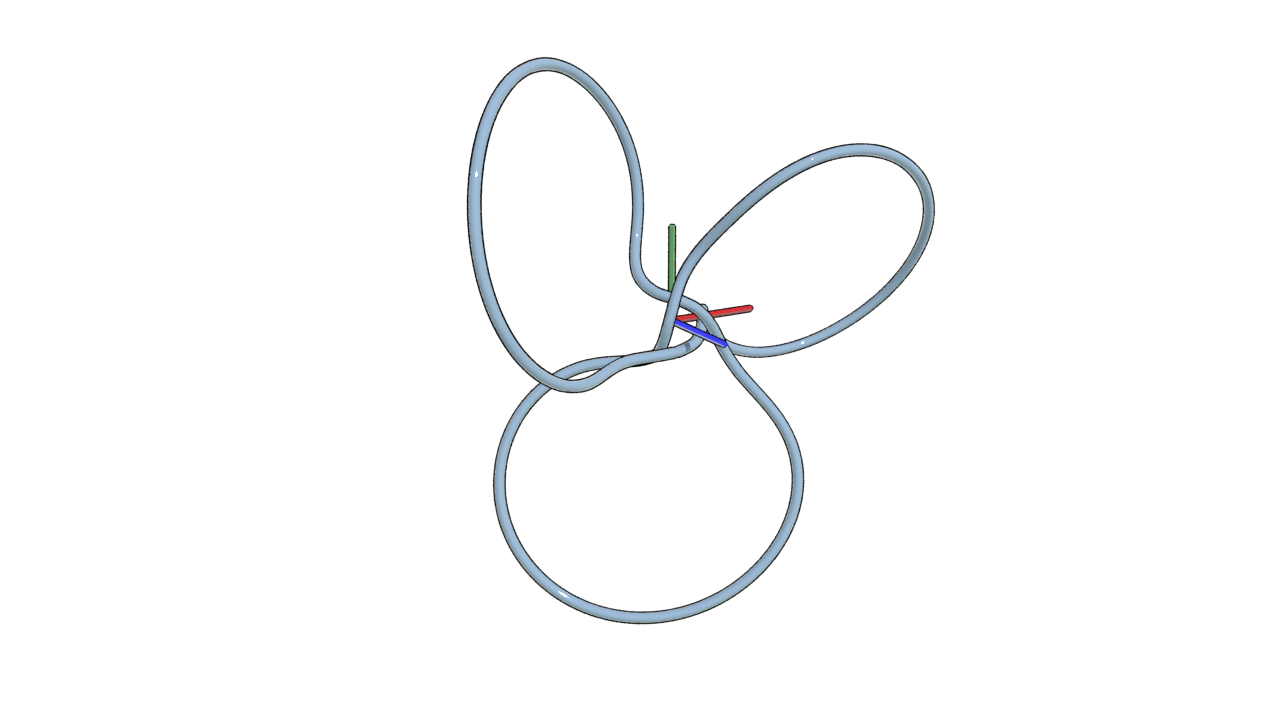}
        \end{center}
        \simtime{10}
    \end{minipage}
    \begin{minipage}[t]{0.18\textwidth}
    \begin{center}
    \includegraphics[width=1.0\textwidth,trim={10cm 0 10cm 0},clip]{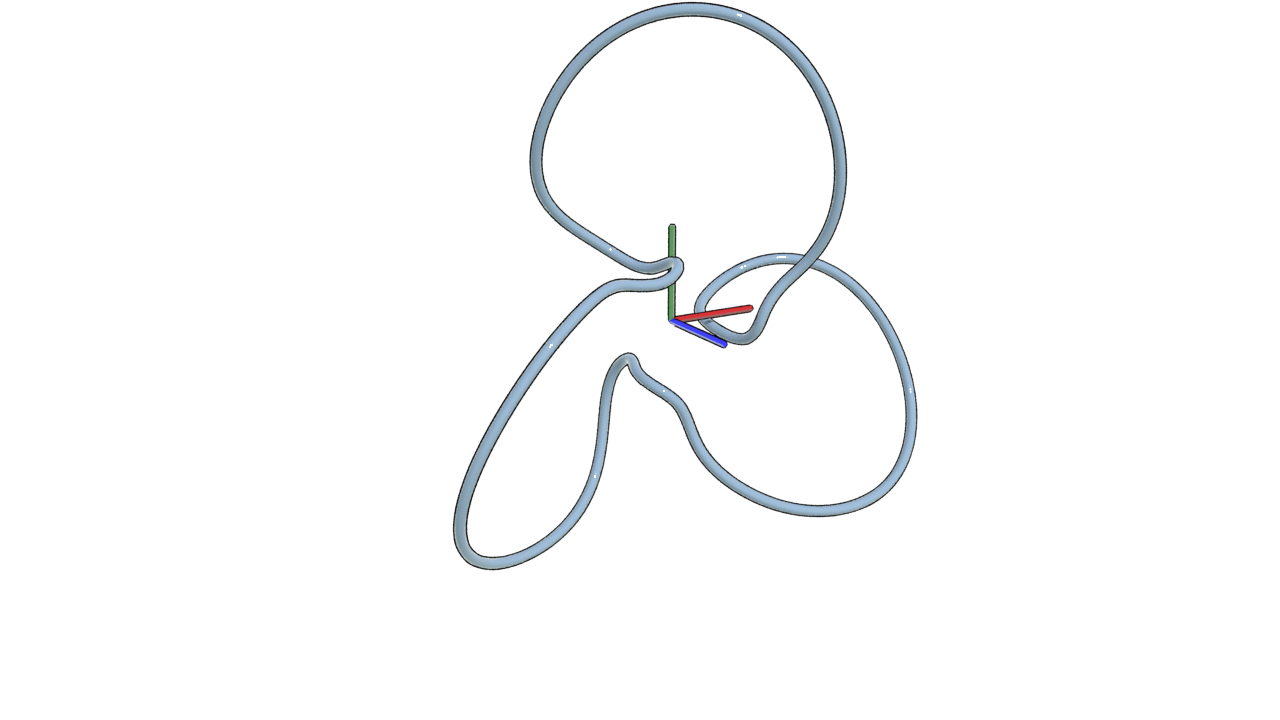}
    \end{center}
    \simtime{14}
    \end{minipage}
    \begin{minipage}[t]{0.18\textwidth}
        \begin{center}
        \includegraphics[width=1.0\textwidth,trim={10cm 0 10cm 0},clip]{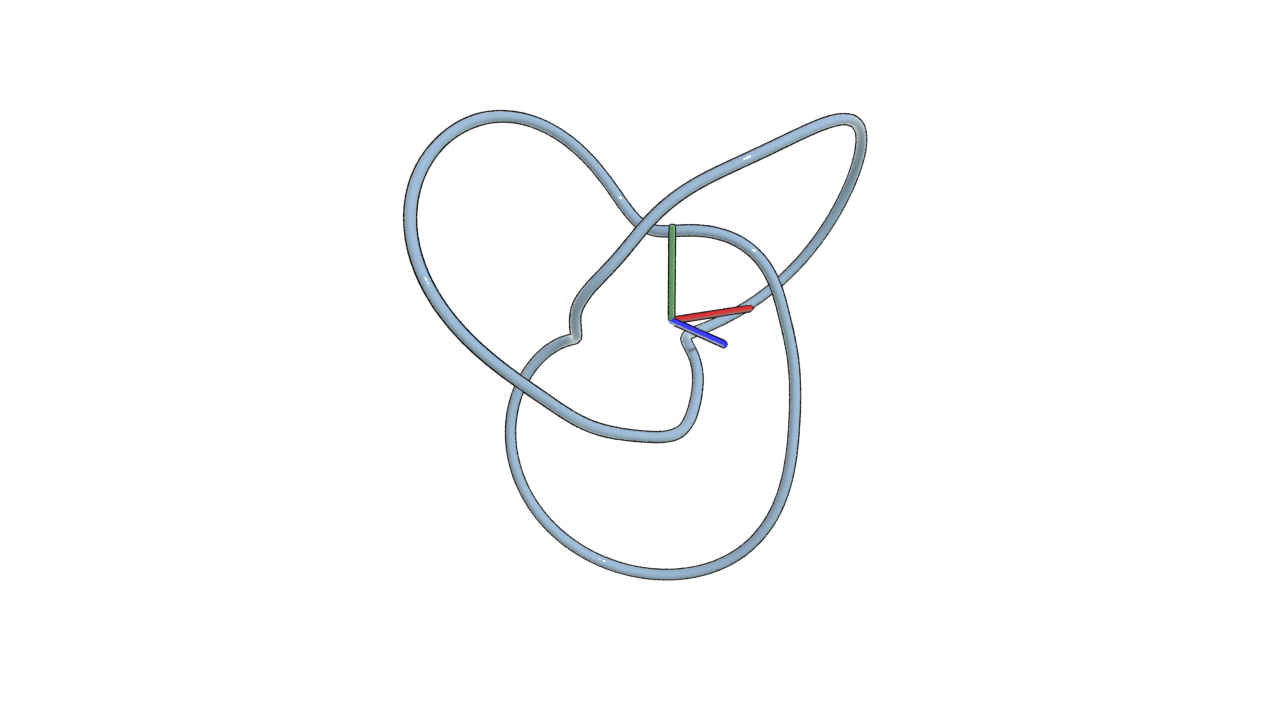}
        \end{center}
        \simtime{19}
    \end{minipage}
    \vrule
    \begin{minipage}[t]{0.17\textwidth}
    \begin{center}
        \includegraphics[width=1.0\textwidth,trim={10cm -2cm 10cm 0cm},clip]{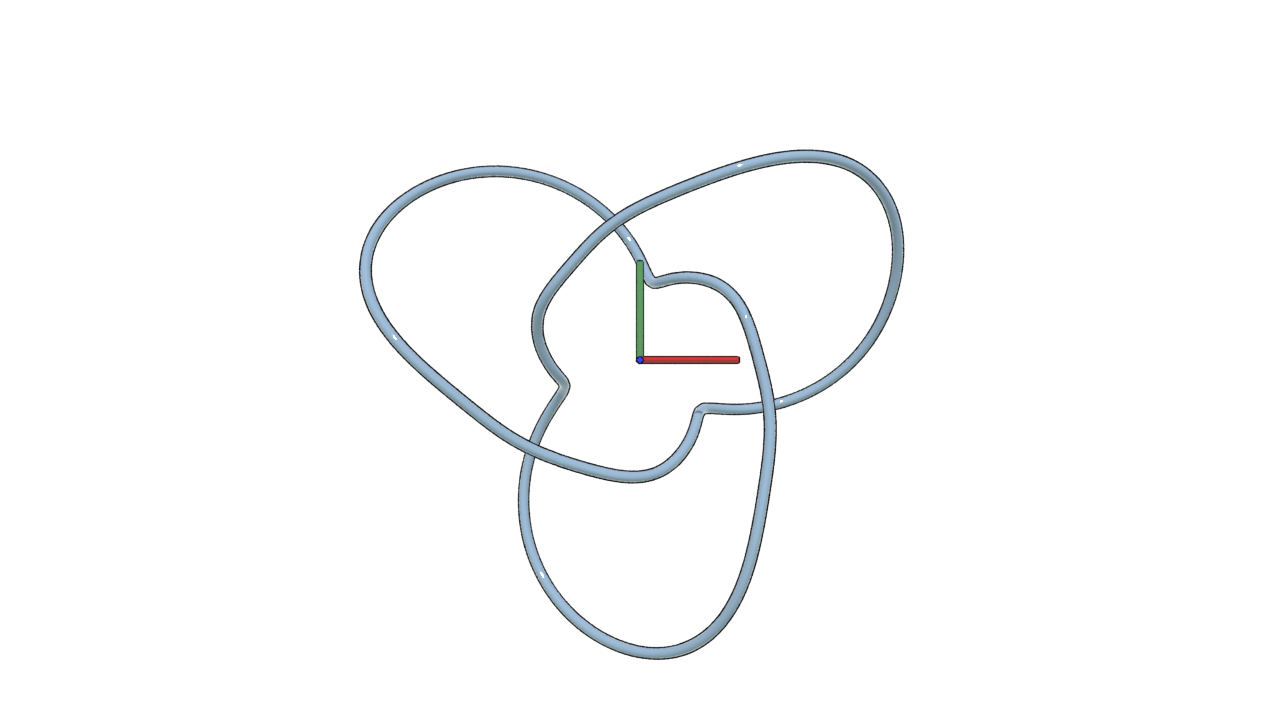}
    \end{center}
    \end{minipage}
     \hrule

\vspace{15pt}

Flow of $\on{hgrad}^{\Phi(\ell)}E$ with $\Phi(\ell)=10^{-5}\ell^2$\\ \vspace{5pt}
    \begin{minipage}[t]{0.18\textwidth}
    \begin{center}
    \includegraphics[width=1.0\textwidth,trim={10cm 0 10cm 0},clip]{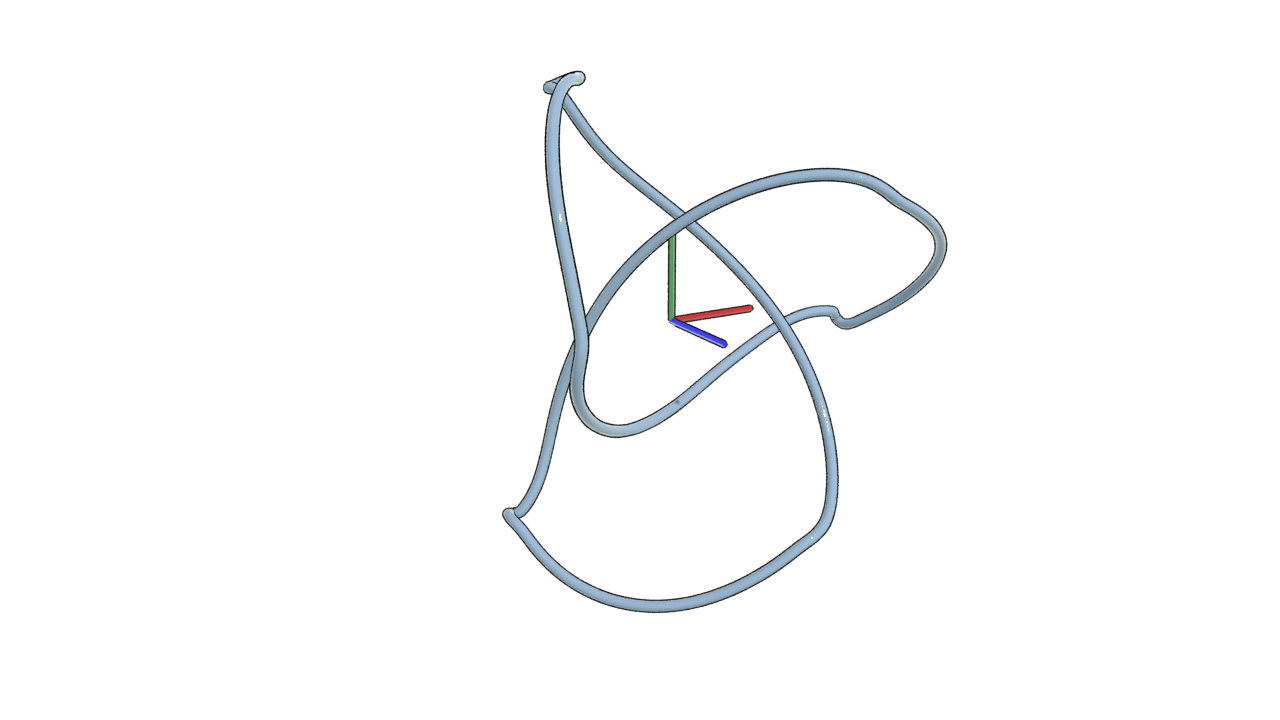}
    \end{center}
    \simtime{3}
    \end{minipage}
    \begin{minipage}[t]{0.18\textwidth}
    \begin{center}
    \includegraphics[width=1.0\textwidth,trim={10cm 0 10cm 0},clip]{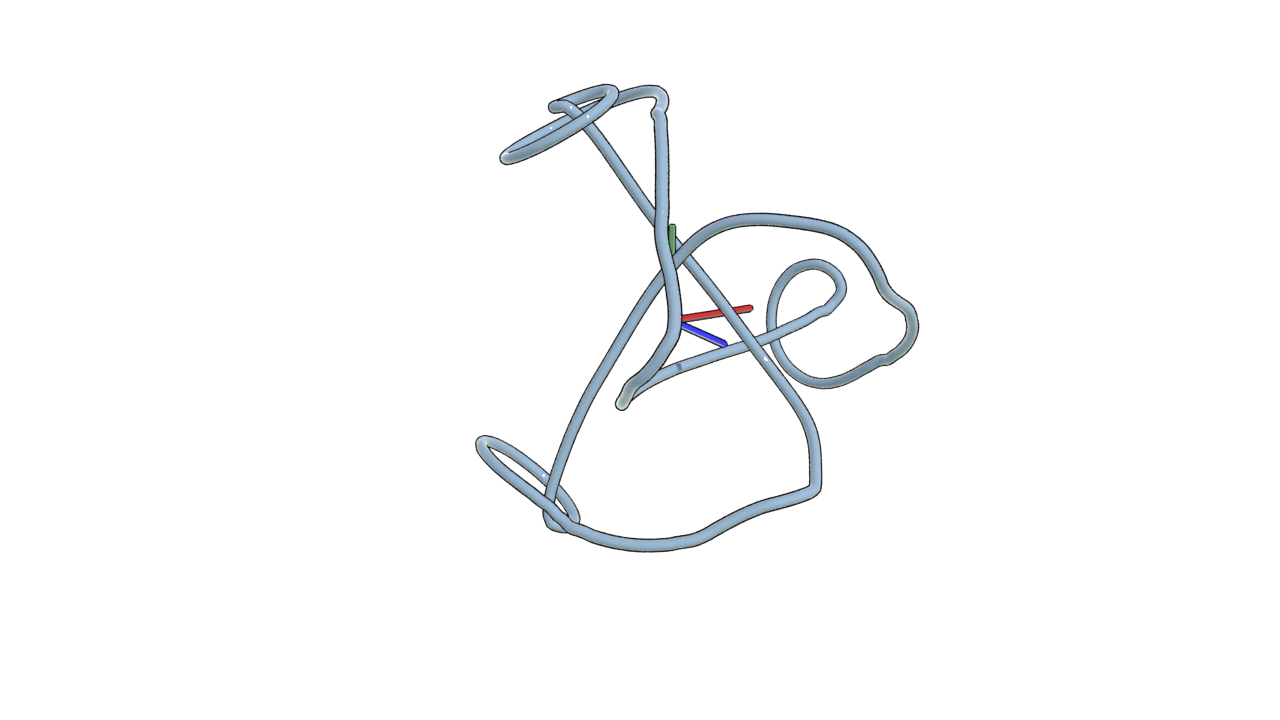}
    \end{center}
    \simtime{6}
    \end{minipage}
    \begin{minipage}[t]{0.18\textwidth}
    \begin{center}
    \includegraphics[width=1.0\textwidth,trim={10cm 0 10cm 0},clip]{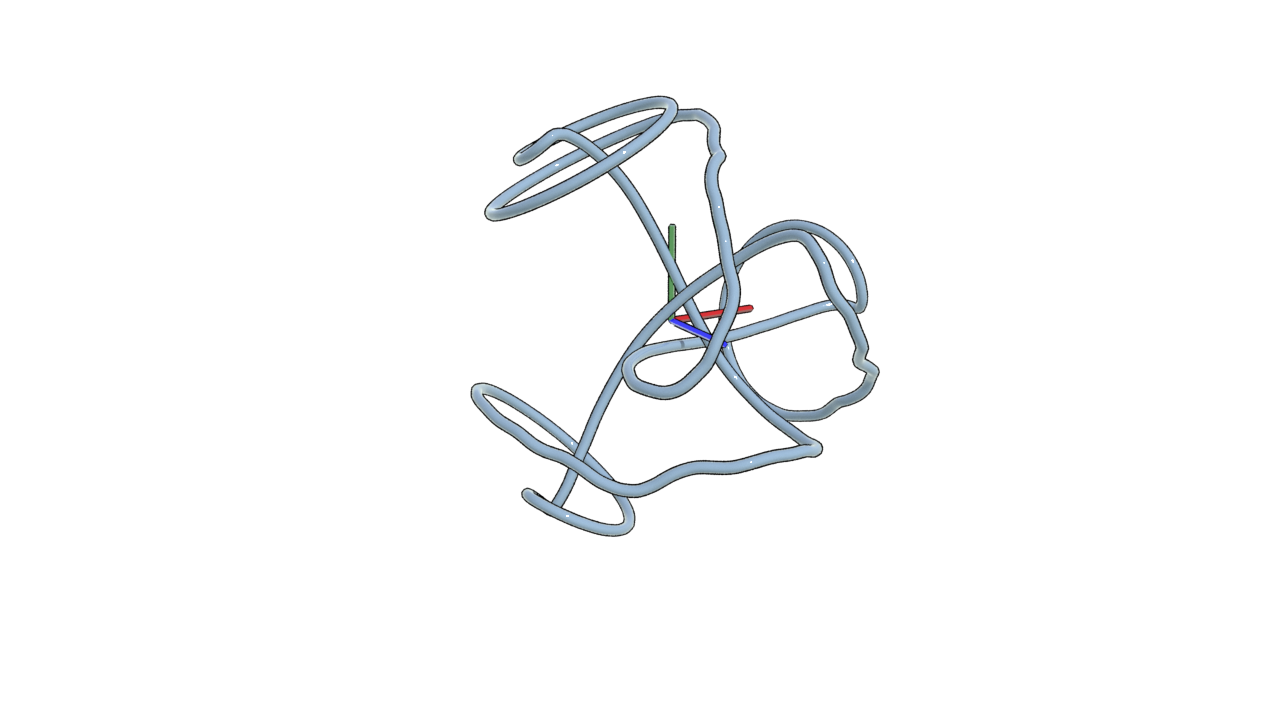}
    \end{center}
    \simtime{10}
    \end{minipage}
    \begin{minipage}[t]{0.18\textwidth}
    \begin{center}
    \includegraphics[width=1.0\textwidth,trim={10cm 0 10cm 0},clip]{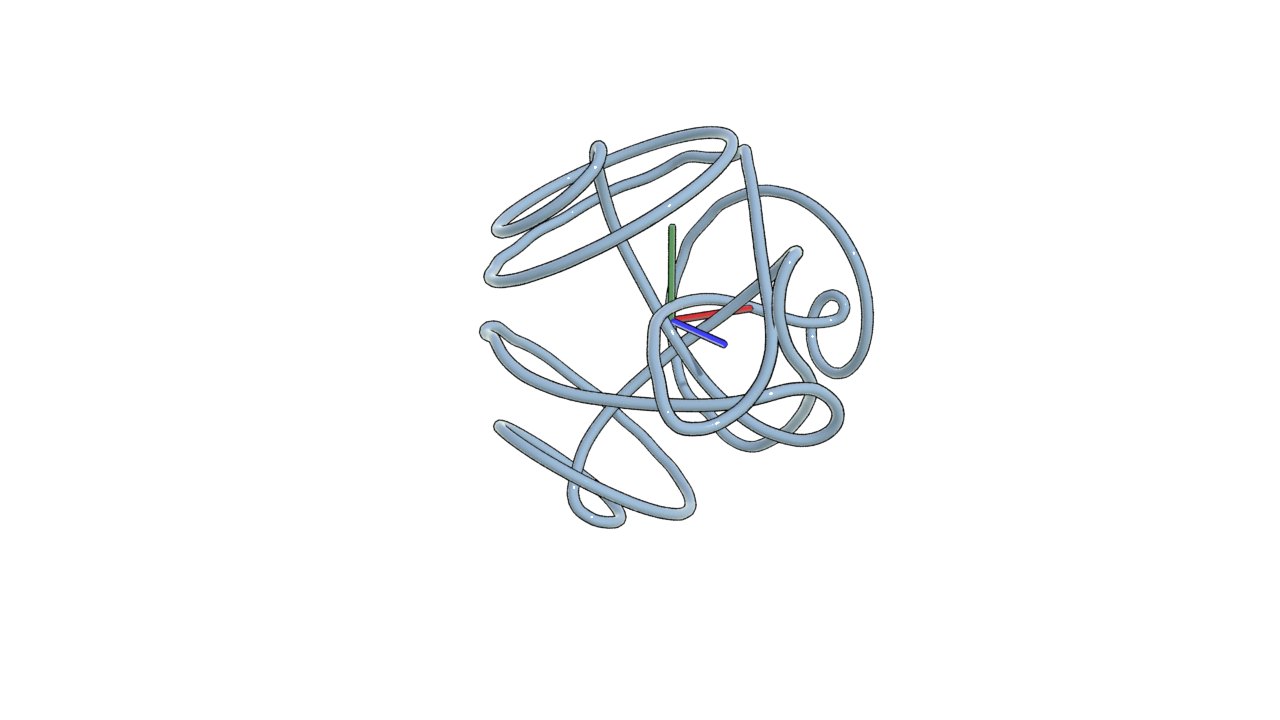}
    \end{center}
    \simtime{19}
    \end{minipage}
    \vrule
    \begin{minipage}[t]{0.17\textwidth}
        \includegraphics[width=1.\textwidth,trim={10cm -2cm 10cm 0cm},clip]{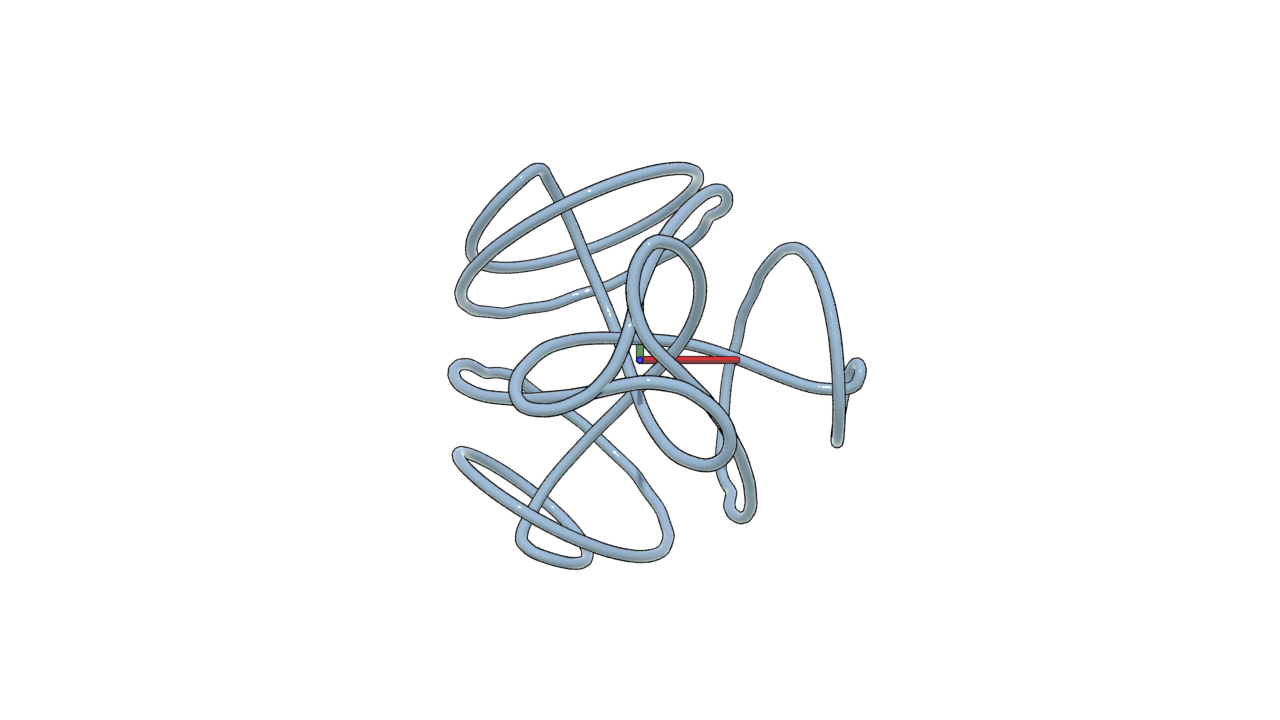}
    \end{minipage}
\hrule
    \caption{Hamiltonian flow of $\on{hgrad}^{\Om^{\Phi(\ell)}}E$ with different choices of $\Phi(\ell)$. In each row the initial curve, which is not shown, corresponds to the trefoil~\eqref{eq:trefoil}. The right-most images are the front-view of the last configurations of curves showing high symmetry for the  120-degree rotation around the z-axis.}
    \label{fig:flow_squaredScale}
\end{figure}

\begin{example}[Total squared scale]
    Our next example is the squared scale functional $E$ (Example~\ref{eg:squaredScale}). Here we test three different choices of $\Phi(\ell)=C\ell^{p}$. Note again that we vary $C$ only for computational purposes and this does not change the trajectory. The simulation results are shown in Figure \ref{fig:flow_squaredScale}.
    
    We first compute for $\Phi(\ell)=\frac{1}{20}$, which corresponds to (a constant multiple of) the Marsden-Weinstein flow $\on{hgrad}^{\OmMW}E$. The curve moves back and forth in the $z$-direction, but curve points tend to get stuck once they come closer to the origin as both the term $-D_s c\x c$ and the term $\frac{1}{2}|c|^2 D_s c\x D_s c$ decrease as $c$ goes to zero. As a result these parts form a complex shape around the origin. The next case is  $\Phi(\ell)=\frac{1}{20}\ell^{-1/10}$. This shows a behavior similar to the first case, but points do not get stuck near the origin due to the additional term in~\eqref{eq:hgrad_squaredScale}. While moving back and forth, the curve does not become as entangled as in the previous case and seems to alternately transform between a trefoil and a trivial knot. The last case is  $\Phi(\ell)=10^{-5}\ell^{2}$. This shows a very different evolution. Unlike the other test cases, the curve does not globally translate in the $z$-direction but forms a complex spiral  shape while shrinking slowly. In all three cases, the symmetry of the trefoil, i.e., that  rotation of 120 degrees around the $z$-axis does not change the shape, seems to be preserved in time. 
    

\end{example}


\appendix

\rev{
\section{Approach via an almost complex structure}\label{sec:almost_symp_approach}
In the main part of this article we constructed new symplectic structures on $B_i(S^1,\mathbb R^3)$, by alternating the Liouville 1-form of the MW symplectic form. Doing so
one obtained a skew symmetric 2-form on $\Imm(S^1,\R^3)$, which is automatically closed. Consequently to obtain a new symplectic structure it only remained to verify the non-degeneracy of this 2-form. 

One may find this approach somewhat artificial or ad-hoc and could imagine, that it would be easier to construct new symplectic structures via directly alternating the MW symplectic structure instead of its Liouville 1-form. Following this alternate strategy one would arrive directly at a non-degenerate 2-form, but would instead need to prove its closedness. 
In this appendix we will discuss that this approach fails in the sense that we were not able to construct any closed forms following this procedure; nevertheless it results in an interesting class of 2-forms and will discuss some of their properties in more details. 
}

On the shape space $B_i(S^1,\R^3)$, the mapping of 90 degrees rotation given  by
\begin{align*}
    \J \colon TB_i(S^1,\R^3) &\to T B_i(S^1,\R^3)\\
    \rev{\bar h} &\mapsto \rev{\overline{D_s c \times h}}
\end{align*}
is an almost complex structure, i.e., an isomorphism with $\J^2=-1$. 
The $L^2$-Riemannian metric $\bar G^\id$, the Marsden-Weinstein symplectic structure $\bar\Omega^{\operatorname{MW}}$, and the almost complex structure $\J$ formally define an almost K\"ahler structure (also called a compatible triple),
\begin{equation}
\bar\Omega^{\operatorname{MW}}(\rev{\bar h,\bar k})=\bar G^{\operatorname{id}}(\J(\rev{\bar h}),\rev{\bar k})
\end{equation}
on $B_i(S^1,\R^3)$ \footnote{This is not a K\"ahler structure in the classical sense, which additionally requires a complex structure i.e., the existence of holomorphic coordinates. Indeed the Marsden-Weinstein symplectic structure does not admit a complex structure \cite{Lempert1993LoopSA}; it has been shown that on the space of isometric mappings of a circle into $\R^3$ modulo Euclidean transformations there is indeed a K\"ahler structure closely related to the Marsden-Weinstein structure, but with a more complicated almost complex structure than $\J$~\cite{Millson1996Kaehler}, see also the comments in~\cite{needham2018kahler}.}.

This observation suggests to define a family of almost symplectic structures $\rev{\bar Z^L}$ via
\begin{equation}\label{eq:alternate_approach}
\rev{\bar Z}^{L}(\rev{\bar h,\bar k})=\bar G^{L}(\J(\rev{\bar h}),\rev{\bar k}).
\end{equation}
If $L$ is non-degenerate\rev{, self-adjoint with respect to $\bar G^{\id}$ and commutes with $\J$,} then $\rev{\bar Z}^{L}$ is by construction \emph{an almost symplectic structure}, i.e., $\rev{\bar Z}^{L}$ is skew-symmetric and non-degenerate.
To show that the induced forms are indeed symplectic it suffices thus to check the  closedness of $\rev{\bar Z}^{L}$. 

At a first glance this approach seems promising and simpler than the  approach \rev{presented in the main part of the paper}. However, \rev{as we will see in the following proposition, the 2-form}  $\rev{\bar Z}^{L}$ fails to be closed at least for all Riemannian metrics that are conformally equivalent (but not equal) to the $L^2$-metric: 
The non-closeness of $\rev{\bar Z}^{\la}$ \rev{is due to the following proposition.}
\rev{
\begin{proposition}\label{prop:no_conformal_symplectic}
    Let $(M,\omega)$ be a symplectic manifold (possibly orbifold) whose dimension is greater than $2$ (possibly infinite-dimensional). Then the only symplectic structures in the conformal class of $\omega$ are the connected component-wise constant multiples of $\omega$.
\end{proposition}
}
\begin{proof}
    \rev{
    Let $\omega^\lambda\coloneqq \lambda \omega$ be a 2-form with conformal factor $\lambda\colon M\to \R_{>0}$. Then we have
    \begin{align}
       d \omega^\lambda=\lambda d\omega+ d\lambda\wedge \omega 
       = d\lambda\wedge \omega.
    \end{align}
    }The 3-form \rev{$d\la\wedge \omega$} is not identically zero unless $\la$ is constant \rev{on each connected component}. To see this let us denote $X_H=\on{grad}^{\omega} H$ for a given function $H\colon M\to \R$. Then we have,
$$d\la\wedge \omega = 0 \iff 0=i_{X_H}(d\la\wedge \omega) = i_{X_H}d\la\wedge \omega - d\la \wedge i_{X_H}\omega = (\L_{X_H} \la).\omega - d\la\wedge dH \quad\forall H. $$
Since the two terms $(\L_{X_H} \la).\omega^\id$ and $ d\la\wedge dh$ have different ranks, they are both zero. Namely, $\L_{X_H} \la$ must be zero. Since at each point $x$ any tangent vector $h\in T_x M$ is locally realized as $X_H(x)$ by choosing  a suitable Hamiltonian $H$, $\la$ must be constant.
\end{proof}


While the above discussion is limited to Riemannian metrics that are conformally equivalent to the $L^2$-metric, it seems that a similar phenomenon is also true for more complicated (higher order) metrics. In particular, we were not able to construct any pair of an almost complex structure $\J$ and a non-conformal operator $L$  which satisfy the required invariance conditions and leads to a closed form \rev{$\bar Z^{L}$} on $B_i(S^1,\R^3)$.

This observation is the main reason why we proceeded to define our symplectic structures 
by altering the Liouville form, thereby ensuring closeness of the corresponding 2-form.   We may also reach types of symplectic forms by solving the non-closeness issue of the approach in this section via correcting the obtained 2-form, e.g. in the conformal case when $\la$ is not a constant and thus \rev{$Z^\la$} is not closed, we could add some $W\in d^{-1}(d\la\wedge \Om^\id) $ so that $Z^\lambda+W$ is closed. By doing this, the non-degeneracy property may be lost and thus one needs to check this again. Our approach for constructing a symplectic form from the Liouville form $\Th^\la$  amounts to choosing $W=-d\la \wedge \Th^\id$. We emphasize that there is a large degree of freedom in $d^{-1}(d\la\wedge \Om^\id)$ and our choice is not the unique one that makes the resulting form symplectic.

\rev{
\subsection*{(Locally) conformal symplectic structure}
Conformal symplectic geometry is a subbranch of symplectic geometry, which arose around 70's and has been studied mostly in finite-dimensional settings. 
 For references on the general theory of (locally) conformal symplectic geometries and their induced dynamics we refer the interested reader to~\cite{vaisman1985locally,yezzi2005conformal,allais2024conformalDampingDynamics}.

Following these references an almost symplectic structure  $\tilde\omega$ is said to be locally conformal symplectic if any point has an open neighborhood $U$ such that there is a function $f_U$ with $f_U\tilde \omega$ being symplectic in $U$. This condition is equivalent to $d\tilde\omega=\alpha\wedge \omega$ for some closed 1-form $\alpha$, called the Lee 1-form. If $\alpha$ satisfies the additional condition that $\alpha=df$ for some globally defined function $f$ then $f\tilde \omega$ is symplectic. In this case $\tilde\omega$ is simply called conformally symplectic.  

While the approach using \eqref{eq:alternate_approach} with a Riemannian metric that is conformally equivalent to the $L^2$-metric is unsuccessful for constructing new symplectic structures,  this procedure exactly leads to conformal symplectic structures. This raises the following open question:
\begin{question}
    [Locally but not globally conformal symplectic structures on $B_i(S^1,\R^3)$]
    A natural question one may ask in this context concerns the existence of a locally but not globally conformal symplectic structure. We note that $H^1(B_i(S^1,\R^3))\cong \R$ since 
 for the fundamental group we have $\pi_1(B_i(S^1,\R^3)) = \mathbb Z$; this follows from \cite{Smale59} and some standard arguments on fibrations and loop spaces; see \cite{Adachi93} for details. This observation suggests the existence of a non-degenerate $\omega$ s.t. $d\omega=\alpha\wedge \omega$ for a closed but non-exact Lie form $\alpha$.  We are, however,  unaware of the existence/non-existence of such an only locally conformal structure on $B_i(S^1,\R^3)$.
\end{question}
}

\rev{
\subsection{Hamiltonian vector fields of conformally symplectic structures}\label{sec:Hamiltonian_flow_locally_conformal}
Here we consider the analogue of Hamiltonian systems for conformally symplectic structures. 
We begin by recalling the definitions and basic properties of Hamiltonian systems induced by a locally conformally symplectic structure. In this setting  two types of Hamiltonian vector fields have been introduced: the first only preserves the Hamiltonian function but not  the conformal symplectic structure, whereas the second one only preserves the conformal symplectic structure (up to conformal changes), but not the Hamiltonian function. Note, that this in contrast to Hamiltonian systems induced by a genuine symplectic structure, which preserve both of them.

To define these two types of Hamiltonian vector fields let $\omega$ be a locally conformal symplectic structure on a manifold $M$ with a Lee form $\alpha$, i.e., $\alpha$ is a closed 1-form such that $d\omega=\alpha\wedge \omega$. Similar as for a classical Hamiltonian system we can defined the field
\begin{align}
    i_{X_H}\omega = dH
\end{align}
Indeed, the flow along $X_H$ preserves the Hamiltonian $H$ but does not preserve $\omega$, 
see e.g.~\cite{wojtkowski1998conformallyDynamics}.
The second type of Hamiltonian vector fields can be defined via
\begin{align}
    i_{Y_H}\omega = d_{\alpha}H
\end{align}
where $d_\alpha$ is the so-called twisted de Rham differential given by $d_\alpha=d-\alpha\wedge \cdot$. The vector field $Y_H$ preserves the locally conformal symplectic structure up to conformal change given by $\L_{Y_H}\omega=\alpha(Y_H)\omega$. In general, it does not preserve the Hamiltonian except under restrictive circumstances such as $H$ being constant, see~\cite{vaisman1985locally,maro2017aubry}. This type of Hamiltonian system can be used to model energy dissipating systems as studied in~\cite{allais2024conformalDampingDynamics, bhatt2016secondOrderConformal}. 

Next we will derive these two types of Hamiltonian systems in our case, i.e., for symplectic structures conformal to the MW structure on the space of space curves.


\subsubsection*{Horizontal Hamiltonian vector field via the standard differential}
We first compute the first type of Hamiltonian vector field $X_H$ via $i_{X_H}Z^\lambda = dH$. Using the same computational routine as in Section \ref{sec:hamiltonian_conformal}, we obtain the horizontal Hamiltonian vector field with respect to the standard $L^2$ metric $G^\id$,
\begin{align}
    X_H = - \frac{1}{\lambda}D_s c \times \on{grad}^{G^{\id}} H = \frac{1}{\lambda} \on{hgrad}^{\Omega^{\on{MW}}}H,
\end{align}
which is a multiple of the Marsden--Weinstein flow with scaling factor $\lambda$.
Note that $H$ is preserved in time by design and that
\begin{align}
    \L_{X_H}Z^{\lambda}
    &= i_{X_H} d(\lambda \OmMW) + d( i_{X_H}(\lambda\OmMW))\\
    & = i_{X_H}(d\lambda\wedge \OmMW)+ ddH
    = \L_{X_H}\lambda.\OmMW - \frac{d\lambda}{\lambda}\wedge dH,
\end{align}
which is not preserved (even conformally) unless $d\lambda$ and $dH$ are linearly dependent at each $c$.}

\rev{
\subsubsection*{Horizontal Hamiltonian vector field via the twisted differential}
Next we discuss horizontal Hamiltonian dynamics of the second type, i.e., we calculate the vector field $Y_H$ via the relation $i_{Y_H}Z^\lambda = d_{d\lambda}H$ with the twisted differential $d_{d\lambda}= d-d\la\wedge\cdot$ mentioned above. Using the relation
\begin{align}
    dH - Hd\lambda= d_{d\lambda}H=  i_{Y_H}Z^\lambda=\lambda i_{Y_H}\OmMW
\end{align}
and the computational routine in Section \ref{sec:hamiltonian_conformal}, we get 
\begin{align}
    \on{grad}^{G^{\id}}H - H  \on{grad}^{G^{\id}} \lambda = \lambda D_s c \times Y_H,
\end{align}
and hence 
\begin{align}
    Y_H 
    &= \frac{1}{\lambda}\left( -  D_s c \times \on{grad}^{G^{\id}} H  + H. D_s c \times \on{grad}^{G^{\id}} \lambda  \right) \\
    &= \frac{1}{\lambda}\left(\on{hgrad}^{\Omega^{\on{MW}}}H - H \on{hgrad}^{\Omega^{\on{MW}}} \lambda\right).
\end{align}
By direct computation, we see that
\begin{align}
    \L_{Y_H}H=H\L_{Y_H}\lambda, \quad \L_{Y_H}Z^\lambda=\L_{Y_H}\lambda.Z^\lambda
\end{align}
and 
\begin{align}
    \L_{Y_H}\lambda
    = \frac{1}{\lambda}(\L_{\on{hgrad}^{\Omega^{\on{MW}}}H}\lambda - H  \L_{\on{hgrad}^{\Omega^{\on{MW}}}\lambda}\lambda) 
    =  \frac{1}{\lambda}\L_{\on{hgrad}^{\Omega^{\on{MW}}}H}\lambda. 
\end{align}
Hence, if $\lambda$ is a conserved quantity under $\on{hgrad}^{\Omega^{\on{MW}}}H$ (equivalently $H$ is conserved along $\on{hgrad}^{\Omega^{\on{MW}}}\lambda$), then $H,\lambda$ and $Z^\lambda$ are exactly preserved and the flow of $Y_H$ behave like the Marsden--Weinstein Hamiltonian flow.  Note, that there are infinitely many commutative Hamiltonian systems with respect to the MW structure forming a  KdV-type hierarchy, cf.~\cite{chern_knöppel_pedit_pinkall_2020}. Therefore if we pick up any combination of the quantities within this hierarchy, we can find Hamiltonian systems induced by a conformal symplectic structure that resemble Hamiltonian systems induced by the MW structure. 


}

\section{Infinite dimensional weak symplectic manifolds}\label{app:weak_symplectic}
An infinite dimensional manifold modeled on convenient vector spaces, as described in \cite[Chapter VI]{kriegl_michor_1997convenient} admits a 2-form $\om\in \Om^2(M)$. We can view it as vector bundle homomorphism $\check\om:TM\to T^*M$. In general, this cannot be an isomorphism, but one can require that it injective:  the candidate for a \emph{weak symplectic structure}. This was the concept used in \cite[Section 48]{kriegl_michor_1997convenient} and in  \cite[Section 2]{Michor06a}. There was a gap in the proof of  \cite[Theorem 48.8]{kriegl_michor_1997convenient} which was repeated in \cite{Michor06a}: It was assumed that $\om$ in a local chart is constant. To remedy this  one has to add a further assumption to the definition of an infinte dimensional weak symplectic manifold; see \ref{app:def}\thetag{3}: The symplectic gradient of $\om$ with respect to itself should exist. For the convenience of the reader we present here the definition of a  weak symplectic manifolds with the further assumption and the basics up to \ref{app:thm} incuding  the proof which contains a gap in \cite[Theorem 48.8]{kriegl_michor_1997convenient}.

\subsection{Infinite dimensional weak symplectic manifolds}\label{app:def}
Let $M$ be a manifold,   
infinite dimensional in general, as described in \cite[Chapter VI]{kriegl_michor_1997convenient}.
                 
A $2$-form  $\om\in\Om^2(M)$ is called a 
{\it weak symplectic structure}\index{weak symplectic structure} 
on $M$ if the following three conditions holds:
\begin{enumerate}
\item $\om$ is closed, $d\om=0$. 
\item The associated vector bundle homomorphism
       $ \check\om: TM \to T^*M$ is injective. 
\item The gradient of $\om$ with respect to itself exists and is smooth;
this can be expressed most easily in charts, 
so let $M$ be open in a convenient vector space $E$. Then for $x\in M$ and 
$X,Y,Z\in T_xM=E$ we have \rev{for the derivative of $\omega$ in direction $X$, that}
$d\om(x)(X)(Y,Z)= \om(\Om_x(Y,Z),X) = \om(\tilde\Om_x(X,Y),Z)$
for smooth $\Om,\tilde\Om: M\x E\x E \to E$ which are bilinear in $E\x E$.  
\end{enumerate}

A $2$-form  $\om\in\Om^2(M)$ is called a 
{\it strong symplectic structure}\index{strong symplectic structure} 
on $M$ if it is closed $(d\om=0)$ and if 
its associated vector bundle homomorphism
$ \check\om: TM \to T^*M$ is invertible with smooth 
inverse. In this case, the vector bundle $TM$ has reflexive fibers 
$T_xM$: Let $i:T_xM\to (T_xM)''$ be the canonical mapping onto the 
bidual. Skew symmetry of $\om$ is equivalent to the fact that the 
transposed $(\check\om)^t=(\check\om)^*\o i:T_xM\to (T_xM)'$ satisfies 
$(\check\om)^t=-\check\om$. Thus, $i=-((\check\om)\i)^*\o  \check\om$  is an 
isomorphism.

\subsection{Cotangent bundles}\label{app:cotangent}
Every cotangent bundle $T^*Q$, viewed as a manifold, carries a 
canonical weak symplectic structure $\om_Q\in\Om^2(T^*Q)$, 
which is defined as follows. Note that this work only with convenient calculus.  
Let 
$\pi_Q^*:T^*Q\to Q$ be the projection. Then the 
{\it Liouville form}\index{Liouville form} $\th_Q\in \Om^1(T^*Q)$ is given by 
$\th_Q(X)=\langle \pi_{T^*Q}(X),T(\pi_Q^*)(X)\rangle $ for  
$X\in T(T^*Q)$,
where $\langle \quad,\quad\rangle$ denotes the duality pairing 
$T^*Q\x_Q TQ\to \mathbb R$. Then the symplectic structure on $T^*Q$ is 
given by $\om_Q = - d\th_Q$, which of course in a local chart looks 
like 
$\om_E((v,v'),(w,w'))=\langle w',v\rangle_E-\langle v',w\rangle_E$.
The associated mapping $ \check\om: T_{(0,0)}(E\x E')=E\x E' 
\to E'\x E''$ is given by $(v,v')\mapsto (-v',i_E(v))$,
where $i_E:E\to E''$ is the embedding into the bidual. 
So the canonical symplectic structure on $T^*Q$ is strong if and only if 
all model spaces of the manifold $Q$ are reflexive and Hilbert spaces. 

\subsection{The $\om$-smooth cotangent space}\label{app:om-cotangent} 
For a weak symplectic manifold $(M,\om)$ 
let $T_x^\om M$ denote the real linear subspace 
$T_x^\om M= \check\om_ x(T_x M) \subset T_x^*M = L(T_x M,\mathbb R)$,
and let us call it the {\it $\om$-smooth cotangent space} with respect to 
the symplectic structure $\om$
of $M$ at $x$ in view of the embedding of test functions
into distributions. \rev{Note, that the convenient structure on $T_x^\om M$ 
is the one from $T_xM$ and not the one from $T_x^*M$.
These vector spaces fit together to form a vector bundle  which 
is isomorphic to the tangent bundle $TM$ via 
$ \check\om: TM\to T^\om M \subseteq T^*M$, together with a smooth injective mapping $T^\om M \to T^*M$. }

Note that only for strong symplectic structures the mapping 
$\check\om_x:T_xM\to T_x^*M$ is a diffeomorphism onto 
$T_x^\om M$ with the structure induced from $T_x^*M$.

\subsection{$\om$-smooth functions} \label{app:om-smooth}
For a weak symplectic manifold  
$(M,\om)$
let 
$$
C^\infty_\om (M,\mathbb R) \subset C^\infty(M,\mathbb R)
$$
denote the subalgebra consisting of all smooth functions  $f:M \to  \mathbb R$ satisfying the following equivalent (by \cite[Lemma 48.6]{kriegl_michor_1997convenient}) conditions:
These are exactly those smooth functions on $M$ which admit a smooth 
$\om$-gradient $\on{grad}^\om f\in \X(M)$. 
\begin{enumerate}
\rev{
\item $df:M \to   T^*M'$ factors to a
        smooth mapping $M\to  T^\om$. 
\item $f$ has a smooth $\om$-gradient 
        $\on{grad}^\om f\in \X(M) = \Ga(TM)$ which for each $Y\in T_xM$ 
        satisfies  
        $df(x)Y = \om(\on{grad}^\om f(x),Y)$.
}\end{enumerate}

\begin{theorem} \label{app:thm}
Let $(M,\om)$ be a weak symplectic manifold. 
The Hamiltonian mapping 
$\on{grad}^\om :C^\infty_\om (M,\mathbb R) \to  \X(M,\om) :=\{X\in \X(M):\L_X\om =0 \}$, which is given by
\begin{displaymath}
i_{\on{grad}^\om f}\om=df\quad\text{ or }\quad 
     \on{grad}^\om f:= (\check\om)\i\o df
\end{displaymath}
is well defined. Also the Poisson bracket 
\begin{align*}
\{\quad,\quad\}&: C^\infty_\om (M,\mathbb R) \x
C^\infty_\om (M,\mathbb R) \to C^\infty_\om (M,\mathbb R)
\\
\{f,g\}&:= i_{\on{grad}^\om f}i_{\on{grad}^\om g}\om = 
     \om(\on{grad}^\om g,\on{grad}^\om f) = 
dg(\on{grad}^\om f) = (\on{grad}^\om f)(g) 
\end{align*}is well defined and gives a Lie algebra structure to
the space $C^\infty_\om (M,\mathbb R)$, which also fulfills
\begin{displaymath}
\{f,gh\}=\{f,g\}h+g\{f,h\}.
\end{displaymath}
We equip $C^\infty_\om(M,\mathbb R)$ with the initial structure 
with respect to the the two following mappings:
$$
C^\infty_\om(M,\mathbb R) \xrightarrow{\subset} C^\infty(M,\mathbb R),\qquad
C^\infty_\om(M,\mathbb R) \xrightarrow{\on{grad}^\om} \X(M).
$$
Then the Poisson bracket is bounded bilinear on $C^\infty_\om(M,\mathbb R)$. 

We have the following long exact sequence of Lie algebras
and Lie algebra homomorphisms:
\begin{displaymath}
0\to  H^0(M)\to C^\infty_\om (M,\mathbb R)  
     \xrightarrow{{\on{grad}^\om}}\X(M,\om) \xrightarrow{\ga} H^1_\om(M) \to 0,
\end{displaymath}
where $H^0(M)$ is the space of locally constant functions, and 
\begin{displaymath}
H^1_\om(M)=\frac{\{\ph\in C^\infty(M\gets T^\om M):d\ph=0\}}
     {\{df:f\in C^\infty_\om(M,\mathbb R)\}}
\end{displaymath}
is the first symplectic cohomology space of $(M,\om)$, a linear 
subspace of the De~Rham cohomology space $H^1(M)$.
\end{theorem}

\begin{proof}
It is clear from \ref{app:om-smooth}, that the
Hamiltonian mapping $\on{grad}^\om$ is well defined 
and has values in $\X(M,\om)$, 
since by \cite[34.18.6 ]{kriegl_michor_1997convenient}, we have
\begin{displaymath}
\mathcal{L}_{\on{grad}^\om f}\om = i_{\on{grad}^\om f}d\om + 
     di_{\on{grad}^\om f}\om =ddf =0.
\end{displaymath}
By \cite[34.18.7]{kriegl_michor_1997convenient}, the space $\X(M,\om)$ is a Lie 
subalgebra of $\X(M)$.
The Poisson
bracket is well defined as a mapping
$\{\quad,\quad\}: C^\infty_\om (M,\mathbb R) \x
C^\infty_\om (M,\mathbb R) \to  
C^\infty(M,\mathbb R) $;
it only remains to check that it has values in the subspace
$C^\infty_\om (M,\mathbb R)$.  

This is a local question, so we may assume that $M$ is an open subset 
of a con\-ven\-ient vector space $E$ equipped with a (non-constant) weak 
symplectic structure. 
So let $f$, $g\in C^\infty_\om (M,\mathbb R)$ and $X,Y,Z\in E$ then
$\{f,g\}(x) = dg(x)(\on{grad}^\om f(x))$, and thus
\begin{align*} 
&d(\{f,g\})(x)y
= d(dg(\quad)y)(x).\on{grad}^\om f(x) + dg(x)(d(\on{grad}^\om f)(x)y)
\\&
 = d\Big(\om(\on{grad}^\om g(\quad),y)\Big)(x).\on{grad}^\om f(x)
     + \om\Bigl(\on{grad}^\om g(x),d(\on{grad}^\om f)(x)y\Bigr)
\end{align*}
We have $\on{grad}^\om f\in\X(M,\om)$ and for any
$X\in\X(M,\om), Y\in\X(M), y\in E$ the condition 
$\mathcal{L}_X\om=0$ implies, using \ref{app:def}.3,
\begin{align*}
0&=(\mathcal L_X\om)(Y,y) = (d\om(X))(Y,y) - \om([X,Y],y) -\om(Y,[X,y])  
\\&
= \om(\tilde\Om(X,Y),y) - \om([X,Y],y) + \om(Y,dX(y_2)).
\end{align*}
Again by \ref{app:def}.3 we have
\begin{align*}
&d(\om(\on{grad}^\om g,y)(\on{grad}^\om f) =
\\&
= d\om(\on{grad}^\om f)(\on{grad}^\om g,y)  
+\om(d(\on{grad}^\om g)(\on{grad}^\om f),y)
\\&
= \om(\tilde\Om(\on{grad}^\om f,\on{grad}^\om g),y)  
+\om(d(\on{grad}^\om g)(\on{grad}^\om f),y)
\end{align*}
Collecting all terms we get
\begin{align*}
&d(\{f,g\})(x)y =
\\&
= d\Big(\om(\on{grad}^\om g(\quad),y)\Big)(x).\on{grad}^\om f(x)
     + \om\Bigl(\on{grad}^\om g(x),d(\on{grad}^\om f)(x)y\Bigr)
\\&
=\om\Big(
\tilde\Om_x(\on{grad}^\om f(x),\on{grad}^\om g(x))
+ d(\on{grad}^\om g)(x)(\on{grad}^\om f(x))
\\&\qquad\qquad
+[\on{grad}^\om f,\on{grad}^\om f](x)
-\tilde\Om_x(\on{grad}^\om f(x),\on{grad}^\om g(x)),
y
\Big)
\\&
=\om\Big(
d(\on{grad}^\om g)(x)(\on{grad}^\om f(x))
+[\on{grad}^\om f,\on{grad}^\om f](x),
y
\Big)
\end{align*}
So \ref{app:om-smooth} is satisfied, and thus 
$\{f,g\}\in C^\infty_\om (M,\mathbb R)$.

If $X\in\X(M,\om)$ then $di_X\om=\mathcal{L}_X\om = 0$, so 
$[i_X\om]\in H^1(M)$ is well defined, and by $i_X\om= \check\om\ o X$ 
we even have $\ga(X):=[i_X\om]\in H^1_\om(M)$, so $\ga$ is well 
defined. 
\end{proof}

\bibliography{bibliography}
\bibliographystyle{abbrv}

\end{document}